\documentclass[11pt]{article}
\usepackage{amsmath,amsfonts,amsthm,amscd,amssymb,graphicx}
\numberwithin{equation}{section}

\newtheorem{theorem}{Theorem}[section]

\newtheorem{lemma}[theorem]{Lemma}

\newtheorem{proposition}[theorem]{Proposition}

\newtheorem{corollary}[theorem]{Corollary}

\newtheorem{remark}[theorem]{Remark}

\newcommand{\RR}{\mathbb{R}}

\newcommand{\cN}{\mathcal{N}}

\def\veps{\varepsilon}
\def\eps{\veps}

\newcommand{\barr}{\begin{array}}
\newcommand{\earr}{\end{array}}
\newcommand{\bmat}{\begin{pmatrix}}
\newcommand{\emat}{\end{pmatrix}}

\begin{document}

\title{Prandtl boundary layer expansions of steady Navier-Stokes flows over
a moving plate}
\author{ Yan Guo\footnotemark[1]  \and Toan T. Nguyen\footnotemark[2]  }
\date{\today}
\maketitle

\begin{abstract}
This paper concerns the validity of the Prandtl boundary layer theory in the
inviscid limit for steady incompressible Navier-Stokes flows. The stationary
flows, with small viscosity, are considered on $[0,L]\times \mathbb{R}_{+}$,
assuming a no-slip boundary condition over a moving plate at $y=0$. We
establish the validity of the Prandtl boundary layer expansion and its
error estimates.
\end{abstract}

\tableofcontents

\renewcommand{\thefootnote}{\fnsymbol{footnote}}

\footnotetext[1]{%
Division of Applied Mathematics, Brown University, 182 George street,
Providence, RI 02912, USA. Email: Yan\underline{~}Guo@Brown.edu}

\footnotetext[2]{%
Department of Mathematics, Pennsylvania State University, State College, PA
16802, USA. Email: nguyen@math.psu.edu.}

\renewcommand{\thefootnote}{\arabic{footnote}}

\section{Introduction}

In this paper, we consider the stationary incompressible Navier-Stokes
equations 
\begin{equation}  \label{ns}
\begin{aligned} UU_{X}+VU_{Y}+P_{X} &=\varepsilon U_{XX}+\varepsilon U_{YY}
\\ UV_{X}+VV_{Y}+P_{Y} &=\varepsilon V_{XX}+\varepsilon V_{YY} \\ U_X + V_Y
& =0 \end{aligned}
\end{equation}
posed in a two dimensional domain $\Omega =\{(X,Y):0\leq X\leq L, Y\ge 0\}$,
with a ``no-slip'' boundary condition 
\begin{equation}
U(X,0)=u_{b}>0,\ \ \ V(X,0)=0  \label{nonslip}
\end{equation}
on the boundary at $Y=0$. The given constant $u_b$ can be viewed as the
moving speed of the plate (that is, the boundary edge $Y=0$). The case when $%
u_{b}\equiv 0$ refers to the classical no-slip boundary condition on a
motionless boundary $Y=0$. The boundary conditions at $x=0,L$ will be
prescribed explicitly in the text.

\emph{We are interested in the problem when $\varepsilon \to0$.} The study
of the inviscid limit and asymptotic boundary layer expansions of
Navier-Stokes equations (\ref{ns}) in the presence of no-slip boundary
condition is one of the central problems in the mathematical analysis of
fluid mechanics. A formal limit $\varepsilon \rightarrow 0$ should lead the
Euler flow $[U_{0},V_{0}]$ inside $\Omega $ which satisfies \textit{only}
non-penetration condition at $Y=0:$ 
\begin{equation*}
V_{0}(X,0)=0.
\end{equation*}
Throughout this paper, we assume that the outside Euler flow is a shear flow 
\begin{equation}
\lbrack U_{0},V_{0}]\equiv \lbrack u_{e}^{0}(Y),0],  \label{shear}
\end{equation}%
for some smooth functions $u_e^0(Y)$. We note that there is no pressure $p_e$
for this Euler flow. Generically, there is a mismatch between the tangential
velocities of the Euler flow $u_{e}\equiv U_{0}(X,0)\equiv u_e^0(0)$ and the
prescribed Navier-Stokes flows $U(X,0) = u_{b}$ on the boundary.

Due to the mismatch on the boundary, Prandtl in 1904 proposed a thin fluid
boundary layer of size $\sqrt{\varepsilon }$ to connect different velocities 
$u_{e}$ and $u_{b}.$ We shall work with the scaled boundary layer, or
Prandtl's, variables: 
\begin{equation*}
x = X, \qquad y=\frac{Y}{\sqrt{\varepsilon }}.
\end{equation*}%
In these variables, we express the solution of the NS equation $[U,V]$ via $%
[U^{\varepsilon },V^{\varepsilon }]$ as 
\begin{equation*}
\lbrack U(X,Y), V(X,Y) \rbrack = \lbrack U(x,\sqrt{\varepsilon }y),\sqrt{%
\varepsilon }\left\{ \frac{V(x,\sqrt{\varepsilon }y)}{\sqrt{\varepsilon }}%
\right\} ] = \lbrack U^{\varepsilon }(x,y),\sqrt{\varepsilon }V^{\varepsilon
}(x,y)]
\end{equation*}%
in which we note that the scaled normal velocity $V^{\varepsilon }$ is $%
\frac{1}{\sqrt{\varepsilon }}$ of the original velocity $V$. Similarly, $%
P(X,Y) = P^\veps(x,y)$. In these new variables, the Navier-Stokes equations %
\eqref{ns} now read 
\begin{equation}  \label{scaledns}
\begin{aligned} U^\veps U^\veps_{x}+V^\veps U^\veps_{y}+P^\veps_{x}
&=U^\veps_{yy}+\varepsilon U^\veps_{xx} \\ U^\veps V^\veps_{x}+V^\veps
V^\veps_{y}+\frac{P^\veps_{y}}{\varepsilon } &=V^\veps_{yy}+\varepsilon
V^\veps_{xx} \\ U^\veps_x + V^\veps_y & =0. \end{aligned}
\end{equation}
Throughout the paper, we perform our analysis directly on these scaled
equations together with the same no-slip boundary conditions \eqref{nonslip}%
. Prandtl then hypothesized that the Navier-Stokes flow can be approximately
decomposed into two parts: 
\begin{equation}  \label{scaledexpansion}
\begin{aligned} \lbrack U^{\varepsilon },V^{\varepsilon }] \quad &\approx
\quad \lbrack u_{e}^{0}(\sqrt{\varepsilon
}y),0]+[\bar{u}(x,y),\bar{v}(x,y)], \\ P^\veps \quad &\approx \quad \bar
p(x,y), \end{aligned}
\end{equation}%
in which $u_e^0(\sqrt \varepsilon y)$ denotes the Euler flow as in %
\eqref{shear}. Putting the ansatz into the Navier-Stokes equations %
\eqref{scaledns} and collecting the leading terms in $\varepsilon$, we
obtain the Prandtl layer corrector $[\bar{u},\bar{v}]$ which satisfies 
\begin{equation}  \label{prandtl}
\begin{aligned} \lbrack u_e+\bar{u}]\bar{u}_{x}+v\bar{u}_{y}+\bar{p}_{x}
&=\bar{u}_{yy} \\ \bar{u}_{x}+\bar{v}_{y} &=0 \end{aligned}
\end{equation}
where $\bar{p}_{x}=\bar{p}_{x}(x),$ since by the second equation in %
\eqref{scaledns}, $\bar p_y =0$. Evaluating the first equation at $y =
\infty $, one gets $\bar p_x=0$, which is precisely due to the assumption
that the trace of Euler on the boundary does not depend on $x$ (a.k.a,
Bernoulli's law). Hence, the boundary layer satisfies 
\begin{equation}
\bar{u}(x,0)=-u_{e}+u_{b}<0,\text{ \ \ \ \ }\bar{v}(x,0)=0,\text{ \ \ \ \ \
\ \ \ \ \ }\lim_{y\rightarrow \infty }\bar{u}(x,y)=0.  \label{prandtlbc}
\end{equation}%
The Prandtl layer $[\bar u , \bar v]$ is subject to an ``initial'' condition
at $x=0$: 
\begin{equation}  \label{initial}
\bar u (0,y) = \bar u_0(y).
\end{equation}


Regarded as one of the most important achievements of modern fluid
mechanics, Prandtl's the boundary layer expansion (\ref{scaledexpansion})
connects the theory of ideal fluid (Euler flows) with the real fluid
(Navier-Stokes flows) near the boundary, for a large Reynolds number (or
equivalently, $\varepsilon \ll1$)$.$ Such a theory has led to tremendous
applications and advances in science and engineering. In particular, since
the Prandtl layer solution $[\bar{u},\bar{v}]$ satisfies an evolution
equation in $x,$ it is much easier to compute its solutions numerically than
those of the original NS flows $[U^{\varepsilon },V^{\varepsilon }]$ which
satisfy an elliptic boundary-value problem. Many other shear layer phenomena
in fluids, such as wake flows (\cite[page 187]{SG}), plane jet flows (\cite[%
page 190]{SG}), as well as shear layers between two parallel flows, can also
be described by the Prandtl layer theory (\ref{prandtl}) and %
\eqref{prandtlbc}.

In spite of the huge success of Prandtl's boundary layer theory in
applications, it remains an outstanding open problem to rigorously justify
the validity of expansion (\ref{scaledexpansion}) in the inviscid limit. 
\emph{The purpose of this paper is to provide an affirmative answer along
this direction.}


As it turns out, we will need higher order approximations, as compared to %
\eqref{scaledexpansion}, in order to be able to control the remainders.
Precisely, we search for asymptotic expansions of the scaled Navier-Stokes
solutions $[U^\veps, V^\veps, P^\veps]$ in the following form: 
\begin{equation}  \label{expansion}
\begin{aligned} U^\veps(x,y) &= u^0_e(\sqrt \varepsilon y) + u^0_p (x,y) +
\sqrt \veps u^1_e (x,\sqrt \veps y) + \sqrt \veps u^1_p(x,y) +
\varepsilon^{\gamma+\frac 12} u^\veps(x,y) \\ V^\veps(x,y) &= v^0_p (x,y) +
v^1_e (x,\sqrt \veps y) + \sqrt \veps v^1_p(x,y) + \varepsilon^{\gamma+\frac
12} v^\veps(x,y) \\ P^\veps (x,y)&= \sqrt \veps p_e^1(x,\sqrt \veps y) +
\sqrt \eps p^1_p(x,y) + \varepsilon p_p^2(x,y)+ \varepsilon^{\gamma+\frac
12} p^\veps(x,y) \end{aligned}
\end{equation}%
for some $\gamma>0$, in which $[u_e^j, v_e^j, p_e^j]$ and $[u_p^j, v_p^j,
p_p^j]$, with $j = 0,1$, denote the Euler and Prandtl profiles,
respectively, and $[u^\veps ,v^\veps,p^\veps]$ collects all the remainder
solutions. Here, we note that these profile solutions also depend on $%
\varepsilon$, and the Euler flows are always evaluated at $(x,\sqrt
\varepsilon y)$, whereas the Prandtl profiles are at $(x,y)$.

Formally speaking, plugging the above ansatz into \eqref{scaledns} and
matching the order in $\varepsilon$, we easily get that $[u_p^0, v_p^0, 0]$
solves the nonlinear Prandtl equation \eqref{prandtl}, whereas the next
Euler profile $[u_e^1, v_e^1, p_e^1]$ solves the linearized Euler equations
around $(u_e^0,0)$: 
\begin{equation}  \label{sys-ve1}
\begin{aligned} u^{0}_e u^{1}_{ex} + v^{1}_e u_{ey}^{0} + p_{ex}^{1} &=0, \\
u^{0}_e v_{ex}^{1} + p_{ey}^{1} & =0, \\ u_{ex}^{1} + v_{e y}^{1}
&=0,\end{aligned}
\end{equation}
or equivalently in the vorticity formulation, 
\begin{equation}  \label{ve1-elliptic}
\begin{aligned} - u^{0}_e \Delta v^{1}_e + u_{eyy}^{0} v_e^{1}
&=0,\end{aligned}
\end{equation}
with $[u_e^1, p_e^1]$ being recovered from the last two equations in %
\eqref{sys-ve1}. The next Prandtl layer $[u_p^1, v_p^1, p_p^1]$ solves the
linearized Prandtl equations around $[u_e+ u_p^0, v_p^0+v_e^1]$, with a
source term $E_p$: 
\begin{equation}  \label{eqs-Pr02}
\begin{aligned} (u_e+ u_p^0)u^{1}_{px} + u^{1}_p u_{px}^{0} + (v^{0}_p
+v_e^1)u^{1}_{py} + v^{1}_p u_{px}^{0}+ p^{1}_{px} - u^{1} _{pyy} &= E_p,\\
u^{1}_{px} + v^{1}_{py} & =0, \end{aligned}
\end{equation}
with $p_{p}^{1} = p_{p}^{1}(x)$. Certainly, the remainder solutions $%
[u^\veps ,v^\veps ,p^\veps]$ solve the linearized Navier-Stokes equations
around the approximate solutions, with a source that contains nonlinearity
in $[u^\veps ,v^\veps]$; see Section \ref{sec-proof} for details.

Note however that as we deal with functions in Sobolev spaces, all profile
solutions are required to vanish at $y = \infty$. As it will be clear in the
text, the actual Prandtl layers will introduce nonzero normal velocity at
infinity, and is one of the issues in controlling the remainders, since the
profiles then won't even be integrable. As a result, our Prandtl layers $%
[u_p^1, v_p^1]$ in the expansion \eqref{expansion} are being cut-off for
large $y$ of the actual layers solving \eqref{eqs-Pr02}. In Section \ref%
{sec-profiles}, we shall provide detailed construction of the approximate
solutions and derive sufficient estimates for our analysis.

\subsection{Boundary conditions}

\label{sec-BCs} The zeroth Euler flow $u_e^0$ is given. Due to the no-slip
boundary condition at $y =0$, we require that 
\begin{equation*}
u_e + u_{p}^{0} (x,0) = u_b, \qquad v_{e}^{1}(x,0) + v_{p}^{0}(x,0) =0,
\end{equation*}
from the zeroth order in $\varepsilon$ of the expansion \eqref{expansion},
and 
\begin{eqnarray*}
u_{p}^{1}(x,0) = - u_{e}^{1}(x,0) , \qquad v_{p}^{1}(x,0) = 0,
\end{eqnarray*}%
from the $\sqrt\varepsilon$-order layers. We also assume that 
\begin{equation*}
\lim_{y\to \infty} u_p^0(x,y) = 0, \qquad \lim_{y\to \infty} u_p^1(x,y) = 0,
\qquad \lim_{Y\to \infty} [u_e^j, v_e^j](x,Y) = 0.
\end{equation*}
The normal velocities $v_p^0(x,y), v_p^1(x,y)$ in the boundary layers are
constructed from $u_p^0(x,y), u_p^1(x,y)$, respectively, through the
divergence-free condition. We note that in general $\lim_{y\rightarrow
\infty }v_{p}^{j}(x,y)\neq 0$ and hence, cut-off functions will be
introduced to localize $v_p^j$.

Next, we discuss boundary conditions at $x=0,L$. Since the Prandtl layers
solve parabolic-type equations, we require only \textquotedblleft
initial\textquotedblright\ conditions at $x=0$: 
\begin{equation*}
u_{p}^{0}(0,y)=\bar{u}_{0}(y),\qquad u_{p}^{1}(0,y)=\bar{u}_{1}(y),
\end{equation*}%
whereas we prescribe boundary values for the Euler profiles at both $x=0,L$: 
\begin{equation*}
u_{e}^{1}(0,Y)=u_{b}^{1}(Y),\qquad v_{e}^{1}(0,Y)=V_{b0}(Y),\text{ \ \ \ \ \ 
}v_{e}^{1}(L,Y)=V_{bL}(Y)
\end{equation*}%
with compatibility conditions $V_{b0}(0)=v_{p}^{0}(0,0),$ $%
V_{bL}(0)=v_{p}^{0}(L,0)$ at the corners of the domain $[0,L]\times \mathbb{R%
}_{+}$.

Finally, we impose the following boundary conditions for the remainder
solution $[u^\veps ,v^\veps]$: 
\begin{equation}  \label{bc}
\begin{aligned} \lbrack u^\veps,v^\veps]_{y=0} &=0\text{ (no-slip)},\text{ \
\ \ \ \ \ \ \ \ \ \ }[u^\veps ,v^\veps]_{x=0}=0\text{ \ (Dirichlet)}, \\
p^\veps-2\varepsilon u^\veps_{x} &=0,\text{ \ \ \ \ }u^\veps_{y}+\varepsilon
v^\veps_{x}=0\quad \text{ \ \ at }x=L\text{ (Neumann or stress-free)}.
\end{aligned}
\end{equation}
Certainly, one may wish to consider different boundary conditions for $%
[u^\veps ,v^\veps]$ at $x=L$. However, to avoid a possible formation of
boundary layers with respect to $x$ near the boundary $x = L$, the above
Neumann stress-free condition appears the most convenient candidate to
impose.

\subsection{Main result and discussions}

We are ready to state our main result:

\begin{theorem}
\label{main} Let $u_{e}^{0}(Y)$ be a given smooth Euler flow, and let $%
u_{b}^{1}(Y)$, $V_{b0}(Y),V_{bL}(Y)$, with $Y=\sqrt{\varepsilon }y$, and $%
\bar{u}_{0}(y)$ and $\bar{u}_{1}(y)$ be given smooth data and decay
exponentially fast at infinity in their arguments, and let $u_{b}$ be a
positive constant. We assume that $|V_{bL}(Y)-V_{b0}(Y)|\lesssim L$ for
small $L$, and 
\begin{equation}
\min_{0\leq y\leq \infty }\left\{ u_{e}^{0}(\sqrt{\varepsilon }y)+\bar{u}%
_{0}(y)\right\} >0.  \label{nozero}
\end{equation}%
Then, there exists a positive number $L$ that depends only on the given data
so that the boundary layer expansions \eqref{expansion}, with the profiles
satisfying the boundary conditions in Section \ref{sec-BCs}, hold for $%
\gamma \in (0,\frac{1}{4})$. Precisely, $[U^{\varepsilon},V^{%
\varepsilon},P^{\varepsilon}]$ as defined in \eqref{expansion} is the unique
solution to the Navier-Stokes equations \eqref{scaledns}, so that the
remainder solutions $[u^{\varepsilon},v^{\varepsilon}]$ satisfy 
\begin{equation*}
\Vert \nabla _{\varepsilon }u^{\varepsilon}\Vert _{L^{2}}+\varepsilon ^{%
\frac{1}{2}}\Vert \nabla _{\varepsilon }v^{\varepsilon}\Vert
_{L^{2}}+\varepsilon ^{\frac{\gamma }{2}}\Vert u^{\varepsilon}\Vert
_{L^{\infty }}+\varepsilon ^{\frac{1}{2}+\frac{\gamma }{2}}\Vert
v^{\varepsilon}\Vert _{L^{\infty }}\leq C_{0},
\end{equation*}%
for some constant $C_{0}$ that depends only on the given data. Here, $\nabla
_{\varepsilon}=(\sqrt{\varepsilon }\partial _{x},\partial _{y})$, and $\Vert
\cdot \Vert _{L^{p}}$ denotes the usual $L^{p}$ norm over $[0,L]\times 
\mathbb{R}_{+}$.
\end{theorem}

As a direct corollary of our main theorem above, we obtain the inviscid limit
of the steady Navier-Stokes flows, with prescribed data up to the order of
square root of viscosity.

\begin{corollary}
Under the same assumption as made in Theorem \ref{main}, there are exact
solutions $[U,V]$ to the original Navier-Stokes equations \eqref{ns} on $%
\Omega =[0,L]\times \mathbb{R}_{+}$, with $L$ being as in Theorem \ref{main}%
, so that 
\begin{eqnarray*}
\sup_{(x,y)\in \Omega }\Big |U(x,y)-u_{e}^{0}(y)-u_{p}^{0}(x,\frac{y}{\sqrt{%
\varepsilon }})\Big | &\lesssim &\varepsilon ^{\frac{1}{2}} \\
\sup_{(x,y)\in \Omega }\Big |V(x,y)-\sqrt{\varepsilon }v_{p}^{0}(x,\frac{y}{%
\sqrt{\varepsilon }})-\sqrt{\varepsilon }v_{e}^{1}(x,y)\Big | &\lesssim
&\varepsilon ^{\frac{\gamma }{2}+\frac{1}{2}} 
\end{eqnarray*}%
as $\varepsilon \rightarrow 0$, for given Euler flow $u_{e}^{0}$, the
constructed Euler flows $[u_{e}^{1},v_{e}^{1}]$ and Prandtl layers $%
[u_{p}^{0},\sqrt{\varepsilon }v_{p}^{0}]$. In particular, we have the
convergence $(U,V)\rightarrow (u_{e}^{0},0)$ in the usual $L^{p}$ norm, with
a rate of convergence of order $\varepsilon ^{1/{2p}}$, $1\leq p<\infty $,
in the inviscid limit of $\varepsilon \rightarrow 0$.
\end{corollary}

\begin{remark}
\emph{We note that the Prandtl layers $[u_p^0, v_p^0]$ are not exactly the
layers $[\bar u,\bar v]$ solving \eqref{prandtl}-\eqref{initial}. Indeed,
whereas $u_p^0 = \bar u$, we have the normal velocity $v_p^0 =
\int_y^{\infty} \bar u_x$, but $\bar v = - \int_0^y \bar u_x$ for the true
Prandtl layer. The introduction of the Euler layer $[u_e^1, v_e^1]$ was
necessary to correct this. }
\end{remark}

Let us give a few comments about the main result. First, the nonzero
condition (\ref{nozero}) and $u_b>0$ are naturally related to the situation
where boundary layers are near a moving plate: such as a wake flow of a
moving body, a moving plane jet flow, and a shear layer between two parallel
flows. It may also be related to the well-known fact in engineering that
injection of moving fluids at the surface prevents the boundary layer
separation.

It is widely known that the mathematical study of Prandtl boundary layers
and the inviscid limit problem is challenging due to its characteristic
nature at the boundary (that is, $v=0$ at $y=0$) and the instability of
generic boundary layers (\cite{GVD, Grenier, GGN, GN}). Here, for steady
flows, we are able to justify the Prandtl boundary layer theory. There are
several issues to overcome. The first is to carefully construct Euler and
Prandtl solutions and derive sufficient estimates. The complication occurs
due to the fact that we have to truncate the actual layers in order to fit
in our functional framework, and the lack of a priori estimates for
linearized Prandtl equations. The construction of the approximate solutions
is done in Section \ref{sec-profiles}.

Next, once the approximate solutions are constructed, we need to derive
stability estimates for the remainder solutions. Due to the limited
regularity obtained for the Prandtl layers $[u_{p}^{1},v_{p}^{1}]$, we shall
study the linearization around the following approximate solutions: 
\begin{equation}
u_{s}(x,y):=u_{e}^{0}(\sqrt{\varepsilon }y)+u_{p}^{0}(x,y)+\sqrt{\varepsilon 
}u_{e}^{1}(x,\sqrt{\varepsilon }y),\qquad
v_{s}(x,y):=v_{p}^{0}(x,y)+v_{e}^{1}(x,\sqrt{\varepsilon }).  \label{us}
\end{equation}%
A straightforward calculation (Section \ref{sec-proof}) yields the equations
for the remainder solutions $[u^{\varepsilon},v^{\varepsilon},p^{%
\varepsilon}]$ in \eqref{expansion}: 
\begin{equation*}
\begin{aligned} u_{s}u^\veps_{x}+u^\veps u_{sx}+v_{s}u^\veps_{y}+v^\veps
u_{sy}+p^\veps _{x}-\Delta _{\varepsilon }u^\veps &=R_{1}(u^\veps,v^\veps)
\\ u_{s}v^\veps_{x}+u^\veps v_{sx}+v_{s}v^\veps_{y}+v^\veps
v_{sy}+\frac{p^\veps_{y}}{\varepsilon }-\Delta _{\varepsilon }v^\veps
&=R_{2}(u^\veps,v^\veps) \\ u^\veps_x + v^\veps_y & =0 , \end{aligned}
\end{equation*}%
with $\Delta _{\varepsilon}=\partial _{y}^{2}+\varepsilon \partial _{x}^{2}$%
. Here, $[u_{s},v_{s}]$ denotes the leading approximate solutions (see %
\eqref{us}), and the remainders $R_{1,2}(u^{\varepsilon},v^{\varepsilon})$
are defined as in \eqref{remainders}. The standard energy estimate (Section %
\ref{sec-EE}) yields precisely a control on $\Vert \nabla
_{\varepsilon}u^{\varepsilon}\Vert _{L^{2}} $ and $\sqrt{\varepsilon }\Vert
\nabla _{\varepsilon}v^{\varepsilon}\Vert _{L^{2}}$, but cannot close the
analysis, due to the large convective term: $\int
u_{sy}u^{\varepsilon}v^{\varepsilon}$, for instance. Indeed, this is a very
common and central difficulty in the stability theory of boundary layers.

The most crucial ingredient (Section \ref{sec-positivity}) in the proof is
to give bound on $\Vert \nabla _{\varepsilon }v^{\varepsilon }\Vert _{L^{2}}$
(in order one, instead of order $\sqrt{\varepsilon }$ from the energy
estimate). The key is to study the vorticity equation, 
\begin{equation*}
-u_{s}\Delta _{\epsilon }v^{\varepsilon }+v_{s}\Delta _{\epsilon
}u^{\varepsilon }-u^{\varepsilon }\Delta _{\epsilon }v_{s}+v^{\varepsilon
}\Delta _{\epsilon }u_{s}=\Delta _{\epsilon }\omega ^{\varepsilon
}+R_{1y}-\varepsilon R_{2x}.
\end{equation*}%
with a new multiplier $\frac{v^{\varepsilon }}{u_{s}}$. Here, the assumption %
\eqref{nozero} and $u_{b}>0$, together with the Maximum Principle for the
Prandtl equations (see estimate \eqref{minw}), assure that $u_{s}$ is
bounded away from zero. Formally, without worrying about boundary terms, the
integral $\int \Delta _{\varepsilon }\omega ^{\varepsilon }v^{\varepsilon }$
vanishes. Hence, the leading term in the vorticity estimate lies in the
convection: $-u_{s}\Delta _{\epsilon }v^{\varepsilon }+v^{\varepsilon
}\Delta _{\epsilon }u_{s}$, or to leading order in the boundary layer
analysis, $-u_{s}\partial _{y}^{2}v^{\varepsilon }+u_{syy}v^{\varepsilon }$.
Our key observation is then the positivity of the second-order operator: 
\begin{equation*}
-\partial _{yy}+\frac{u_{syy}}{u_{s}}.
\end{equation*}%
Indeed, a direct calculation yields 
\begin{eqnarray*}
-\int v_{yy}v &=&\int |v_{y}|^{2}=\int \left\vert \partial _{y}\left\{ \frac{%
v}{u_{s}}u_{s}\right\} \right\vert ^{2} \\
&=&\int u_{s}^{2}\left\vert \partial _{y}\left\{ \frac{v}{u_{s}}\right\}
\right\vert ^{2}+\int u_{sy}^{2}\left\{ \frac{v}{u_{s}}\right\} ^{2}+2\int
\partial _{y}\left\{ \frac{v}{u_{s}}\right\} u_{s}u_{sy}\left\{ \frac{v}{%
u_{s}}\right\} \\
&=&\int u_{s}^{2}\Big |\partial _{y}\left\{ \frac{v}{u_{s}}\right\} \Big |%
^{2}+\int u_{sy}^{2}\left\{ \frac{v}{u_{s}}\right\} ^{2}-\int \left\{ \frac{v%
}{u_{s}}\right\} ^{2}[u_{s}u_{sy}]_{y} \\
&=&\int u_{s}^{2}\Big |\partial _{y}\left\{ \frac{v}{u_{s}}\right\} \Big |%
^{2}-\int \frac{u_{syy}}{u_{s}}v^{2}
\end{eqnarray*}%
which gives the positivity estimate: 
\begin{equation}
\int \left\{ -\partial _{yy}+\frac{u_{syy}}{u_{s}}\right\} vv=\int |\partial
_{y}v|^{2}+\int \frac{u_{syy}}{u_{s}}v^{2}=\int u_{s}^{2}\Big |\partial
_{y}\left\{ \frac{v}{u_{s}}\right\} \Big |^{2}>0.  \label{positivity-intro}
\end{equation}%
The desired bound on $v_{y}^{\varepsilon }$, and in fact, $\nabla
_{\varepsilon }v^{\varepsilon }$ is derived from this positivity estimate
and the weighted estimates from the vorticity equation. Precisely, 
\begin{equation}
\begin{aligned}\iint v_{y}^{2} &=\iint \left\{ \partial _{y}\left[
u_{s}\left\{ \frac{v}{u_{s}}\right\} \right] \right\} ^{2}\leq 2\iint
u_{s}^{2}\left\{ \frac{v}{u_{s}}\right\} _{y}^{2}+2\iint \{\partial
_{y}u_{s}\}^{2}\left\{ \frac{v}{u_{s}}\right\} ^{2} \\ &\leq 2\iint
u_{s}^{2}\left\vert \partial _{y}\left\{ \frac{v}{u_{s}}\right\} \right\vert
^{2}+2\iint \left[ \partial _{y}\left\{ \frac{v}{u_{s}}\right\} \right]
^{2}\times \sup_{x}\int y\{\partial _{y}u_{s}\}^{2}dy \\ &\leq 2\iint
u_{s}^{2}\left\vert \partial _{y}\left\{ \frac{v}{u_{s}}\right\} \right\vert
^{2}[1+\frac{1}{\min u_{s}^{2}}\times \sup_{x}\int y\{\partial
_{y}u_{s}\}^{2}dy] \\ &\lesssim \iint u_{s}^{2}\left\vert \partial
_{y}\left\{ \frac{v}{u_{s}}\right\} \right\vert ^{2}, \end{aligned}
\label{low-vy}
\end{equation}%
in which the last inequality used the estimate (\ref{uintegral}) on $u_{s}$.

In addition, the Dirichlet boundary condition at $x=0$ and the stress-free
boundary condition at $x=L$ as imposed in (\ref{bc}) are carefully designed
to ensure boundary contributions at $x=0$ and $x=L$ are controllable.

Our second ingredient is to derive $L^{\infty }$ estimate for the remainder
solution $[u^{\varepsilon},v^{\varepsilon}]$ and to close the nonlinear
analysis; Sections \ref{sec-sup} and \ref{sec-proof}. We have to overcome
the issue of regularity of solutions to the elliptic problem in domains with
corners. In particular, it is a subtlety to justify the integrability of all
terms in integration by parts, given the limited regularity provided for the
solution near the corners. We remark that in the case $u_{b}=0$, our
analysis does not directly apply due to the presence of zero points of the
profile solutions $u_{s}$, and hence the function $\frac{v^{\varepsilon}}{%
u_{s}}$ can no longer be used as a multiplier. Our positivity estimate is
lost in this limiting, but classical, case.

Finally, the third ingredient is the construction of profiles (or
approximate solutions) which enables us to establish the error estimates and
to close our nonlinear iteration. Such constructions are delicate (Section %
\ref{sec-profiles}), due to the regularity requirement of ${v_{pxx}^{1}}$ in
the remainder $R_{2}(u^{\varepsilon},v^{\varepsilon})$. In order to control
it, we need to create artificial new boundary layer at $y=0$ in (\ref{def-Eb}%
) to guarantee sufficient regularity for the first order Euler correction $%
[u_{e}^{1},v_{e}^{1},p_{e}^{1}]$. More importantly, to construct both $%
v_{e}^{1}$ and $v_{p}^{1}$, the positivity estimate \eqref{positivity-intro}
once again plays the decisive role; see Sections \ref{sec-1Euler} and \ref%
{sec-1Prandtl}.

We are not aware of any work in the literature that deals with the validity
of the Prandtl boundary layer theory for the steady Navier-Stokes flows. For
unsteady flows, there are very interesting contributions \cite{Asano,
SC1,SC2} in the analyticity framework, \cite{Maekawa} in the case where the
initial vorticity is assumed to be away from the boundary, or \cite{MT} for
special Navier-Stokes flows. An analogous program for unsteady flows as done
for the steady case in the precent paper appears not possible, due to the
fact that (unsteady) boundary layers are known to be very unstable; see, for
instance, \cite{EE, GVD, Grenier, GGN, GN}.

\textbf{Notation.} Throughout the paper, we shall use $\langle y \rangle = 
\sqrt{1+y^2}$, and $\|\cdot \|_{L^p}$ or occasionally $\|\cdot \|_p$ to
denote the usual $L^p$ norms, $p\ge 1$, with integration taken over $\Omega
= [0,L]\times \mathbb{R}_+$. We also use $\| \cdot \|_{L^p(0,L)}$ and $%
\|\cdot \|_{L^p(\mathbb{R}_+)}$ to denote the $L^p$ norms with integration
taken over $[0,L]$ and $\mathbb{R}_+$, respectively. We shall denote by $%
C(u_s, v_s)$ a universal constant that depends only on the given Euler flow $%
u_e^0$ and boundary data. Occasionally, we simply write $C$ or use the
notation $\lesssim$ in the estimates. By uniform estimates, we always mean
those that are independent of smallness of $\varepsilon$ and $L$. The
smallness of $L$ is determined depending only on the given data, whereas $%
\varepsilon$ is taken arbitrarily small, once the given data and $L$ are
fixed. In particular, $\varepsilon \ll L$.

%

\section{Construction of the approximate solutions}

\label{sec-profiles} In order to construct the approximate solutions, we
plug the Ansatz \eqref{expansion} into the scaled Navier-Stokes equations %
\eqref{scaledns}, and match the order in $\varepsilon$ to determine the
equations for the profiles. For our own convenience, let us introduce 
\begin{equation}  \label{approximate-soln}
\begin{aligned} u_\mathrm{app}(x,y) &= u^0_e(\sqrt \varepsilon y) + u^0_p
(x,y) + \sqrt \veps u^1_e (x,\sqrt \veps y) + \sqrt \veps u^1_p(x,y) \\
v_\mathrm{app}(x,y) &= v^0_p (x,y) + v^1_e (x,\sqrt \veps y) + \sqrt \veps
v^1_p(x,y) \\ p_\mathrm{app}(x,y)&= \sqrt \veps p_e^1(x,\sqrt \veps y) +
\sqrt \eps p^1_p(x,y) +\varepsilon p_p^2(x,y). \end{aligned}
\end{equation}%
We then calculate the error caused by the approximation: 
\begin{subequations}
\begin{align}
R^u_\mathrm{app}&: = [ u_\mathrm{app}\partial _{x} + v_\mathrm{app}
\partial_y ] u_\mathrm{app} + \partial_x p_\mathrm{app} - \Delta_\veps u_%
\mathrm{app}  \label{tangential} \\
R^v_\mathrm{app}&: = [ u_\mathrm{app}\partial _{x} + v_\mathrm{app}
\partial_y ] v_\mathrm{app} + \frac1\varepsilon\partial_y p_\mathrm{app} -
\Delta_\veps v_\mathrm{app},  \label{normal}
\end{align}%
or explicitly, 
\end{subequations}
\begin{equation*}
\begin{aligned} R^u_\mathrm{app} & = \Big[
\{u_{e}^{0}+u_{p}^{0}+\sqrt{\varepsilon }[u_{e}^{1}+u_{p}^{1}]\}\partial_x
+\{v_{p}^{0}+v_{e}^{1}+\sqrt{\varepsilon }v_{p}^{1} \}\partial _{y} \Big
]\{u_{e}^{0}+u_{p}^{0}+\sqrt{\varepsilon }[u_{e}^{1}+u_{p}^{1}] \} \\ &\quad
+\sqrt{\varepsilon }\partial _{x}\{p_{e}^{1}+ p_{p}^{1} + \sqrt \veps
p_p^2\} - (\partial_y^2 + \varepsilon \partial_x^2
)\{u_{e}^{0}+u_{p}^{0}+\sqrt{\varepsilon }[u_{e}^{1}+u_{p}^{1}] \} \\
R^v_\mathrm{app}& = \Big[ \{u_{e}^{0}+u_{p}^{0}+\sqrt{\varepsilon
}[u_{e}^{1}+u_{p}^{1}]\}\partial _{x}
+\{v_{p}^{0}+v_{e}^{1}+\sqrt{\varepsilon }v_{p}^{1}\}\partial _{y} \Big]
\{v_{p}^{0}+v_{e}^{1}+\sqrt{\varepsilon }v_{p}^{1}\} \\ &\quad
+\frac{1}{\sqrt{\varepsilon }}\partial _{y}\{p_{e}^{1}+p_{p}^{1}+\sqrt \veps
p_p^2\} - (\partial_y^2 + \varepsilon \partial_x^2
)\{v_{p}^{0}+v_{e}^{1}+\sqrt{\varepsilon }v_{p}^{1}\}, \end{aligned}
\end{equation*}
in which we recall that the Euler profiles are always evaluated at $(x,z) =
(x,\sqrt \varepsilon y)$. We shall construct the approximate solutions so
that $R^{u,v}_\mathrm{app}$ are being small in $\varepsilon$. In this
long section, we shall prove that

\begin{proposition}
\label{prop-approximate} Under the same assumptions as in Theorem \ref{main}%
, there are approximate solutions $[u_{\mathrm{app}},v_{\mathrm{app}},p_{%
\mathrm{app}}]$ so that 
\begin{equation*}
\Vert R_{\mathrm{app}}^{u}\Vert _{L^{2}}+\sqrt{\varepsilon }\Vert R_{\mathrm{%
app}}^{v}\Vert _{L^{2}}\leq C(L,\kappa )\varepsilon ^{3/4-\kappa },
\end{equation*}%
for arbitrarily small $\kappa >0$. Furthermore, there hold various
regularity estimates on the approximate solutions which are summarized in
Corollary \ref{cor-Prandtl0} for the zeroth-order Prandtl layers $%
[u_{p}^{0},v_{p}^{0}]$, Section \ref{sec-defE1} for the Euler profiles $%
[u_{e}^{1},v_{e}^{1}]$, and Section \ref{sec-defP1} for the Prandtl layers $%
[u_{p}^{1},v_{p}^{1}]$.
\end{proposition}

\subsection{Zeroth-order Prandtl layers}

\label{sec-defP0} The (leading) zeroth order terms on the right-hand side of
(\ref{tangential}) consist of 
\begin{equation*}
R^{u,0}: = \{u_{e}^{0}+u_{p}^{0}\}\{ u_e^0 + u_{p}^{0}\}_x
+\{v_{p}^{0}+v_{e}^{1}\} \{u_{e}^{0}+u_{p}^{0}\}_{y}-\{ u_e^0 +
u_{p}^{0}\}_{yy}
\end{equation*}
in which we note that $\{ u_e^0 \}_x =0$. Since the Euler flows are
evaluated at $(x,z)=(x,\sqrt \varepsilon y)$, we may write 
\begin{equation*}
\{v_{p}^{0}+v_{e}^{1}\}\partial _{y}u_{e}^{0}=\sqrt{\varepsilon }%
\{v_{p}^{0}+v_{e}^{1}\}u_{ez}^{0}
\end{equation*}%
and, with $u_e = u_e^0(0)$, 
\begin{equation*}
u_{e}^{0}u_{px}^{0}+v_{e}^{1}u_{py}^{0} =
u_{e}u_{px}^{0}+v_{e}^{1}(x,0)u_{py}^{0} + \sqrt{\varepsilon }u_{ez}(\sqrt{%
\varepsilon }y)yu_{px}^{0}+\sqrt{\varepsilon }v_{ez}^{1}(\sqrt{\varepsilon }%
y)yu_{py}^{0}+E^{0}
\end{equation*}
in which $E^{0}$ satisfies 
\begin{equation}  \label{e0}
\begin{aligned} E^{0} &:=\sqrt{\varepsilon }
u_{px}^{0}\int_{0}^{y}\{u_{ez}^{0}(\sqrt{\varepsilon }\theta
)-u_{ez}^{0}(\sqrt{\varepsilon }y)\}d\theta +\sqrt{\varepsilon
}u_{py}^{0}\int_{0}^{y}\{v_{ez}^{1}(\sqrt{\varepsilon }\theta
)-v_{ez}^{1}(\sqrt{\varepsilon }y)\}d\theta \\ &=\varepsilon
u_{px}^{0}\int_{0}^{y}\int_{y}^{\theta }u_{ezz}^{0}(\sqrt{\varepsilon }\tau
)d\tau d\theta +\varepsilon u_{py}^{0}\int_{0}^{y}\int_{y}^{\theta
}v_{ezz}^{1}(\sqrt{\varepsilon }\tau )d\tau d\theta . \end{aligned}
\end{equation}
In particular, $E_0$ is in the high order in $\varepsilon$, as to be proved
rigorously in the next section; see \eqref{est-E0}. To leading order, this
yields the nonlinear Prandtl problem for $u_p^0$: 
\begin{equation}  \label{prandtl0}
\left\{ \begin{aligned} \{u_{e} +u_{p}^{0}\} &u_{px}^{0}
+\{v_{p}^{0}+v_{e}^{1}(x,0)\}u_{py}^{0}=u_{pyy}^{0} , \qquad v_p^0(x,y): =
\int_y^\infty u^0_{px}\; dy , \\ u_{p}^{0} (x,0) &= u_b - u_e, \qquad
v_{p}^{0}(x,0) + v_{e}^{1}(x,0) = 0, \qquad u_p^0(0,y) = \bar u_0(y).
\end{aligned}\right.
\end{equation}%
Having constructed the Prandtl layer $[u_p^0, v_p^0]$, the zeroth order term 
$R^{u,0}$ is reduced to 
\begin{equation}  \label{Ru0}
R^{u,0} = \sqrt{\varepsilon } \{v_{p}^{0}+v_{e}^{1}\}u_{ez}^{0} + \sqrt{%
\varepsilon }u_{ez}(\sqrt{\varepsilon }y)yu_{px}^{0}+\sqrt{\varepsilon }%
v_{ez}^{1}(\sqrt{\varepsilon }y)y u_{py}^{0}- \varepsilon u_{ezz}^0+E^{0} ,
\end{equation}
which will be put into the next order in $\varepsilon$.

%

\begin{lemma}
\label{lem-Pr0} Let ${u_{p}^{0}}(0,y): = \bar u_0(y)$ be an arbitrary smooth
boundary data for the Prandtl layer at $x=0$. Assume that $%
\min_{y}\{u_{e}+\bar u_0(y)\}>0$. Then, there exists a positive number $L$
so that the problem (\ref{prandtl0}) has the unique smooth solution $%
u_p^0(x,y)$ in $[0,L]\times \mathbb{R}_+$. Furthermore, for all $n \ge 0$, $%
k\ge 0$, there exists a constant $C_0(n,k,\bar u_0)$ so that there holds the
uniform bound: 
\begin{equation}  \label{Pr0-bound}
\sup_{x\in [0,L]} \|\langle y\rangle^{n/2} \partial_x^k u_p^0\|_{L^2(\mathbb{%
R}_+)} +\| \langle y\rangle ^{n/2} \partial_x^k\partial_y u_p^0\|_{L^2(0,L;
L^2(\mathbb{R}_+))} \le C_0(n, k,\bar u_0),
\end{equation}
with $\langle y \rangle = \sqrt{1+|y|^2}$. Here, the constant $C(n,k, \bar
u_0)$ depends on $n,k$, and the $\langle y \rangle^n$-weighted $H^{2k}(%
\mathbb{R}_+)$ norm of the boundary value $\bar u_0(y)$.
\end{lemma}

\begin{corollary}
\label{cor-Prandtl0} Let $u_p^0$ be the Prandtl layer constructed as in
Lemma \ref{lem-Pr0}. Then, there holds 
\begin{equation}  \label{Pr0-Hbound}
\sup_{x\in [0,L]} \|\langle y\rangle^{n/2} \partial_x^k\partial_y^j
[u_p^0,v_p^0]\|_{L^2(\mathbb{R}_+)} \le C_0(n, k, j, \bar u_0),
\end{equation}
for arbitrary $n,k,j$.
\end{corollary}

\begin{proof} Indeed, the proof follows directly from Lemma \ref{lem-Pr0} and a use of equations \eqref{prandtl0} for the Prandtl layer to bound $\partial_y^2 u_p^0$ by those of lower-order derivative terms. 
\end{proof}

\begin{proof}[Proof of Lemma \ref{lem-Pr0}] Let $u_e = u_e^0(0)$. Following Oleinik \cite{Oleinik}, we use the von Mises transformation: 
$$\eta : =\int_{0}^{y} ( u_e+u_{p}^{0} (x,\theta ) )d\theta , \qquad \mathbf{w}(x,\eta ): =  u_e+u_{p}^{0}(x,y(\eta )).$$
The function $\mathbf{w}$ then solves
\begin{equation*}
\mathbf{w}_{x}=\{\mathbf{ww}_{\eta }\}_{\eta }
\end{equation*}%
on $[0,L]\times \RR_+$. Note that by the standard Maximum Principle (to the equation for ${\bf w}^2$), we have 
\begin{equation}\label{minw}{\bf w} \ge \min_y \{u_b, u_e+u_{p}^{0}|_{x=0}\} \ge c_0>0,\end{equation} 
for some positive constant $c_0$. Hence, the above is a non-degenerate
parabolic equation. Since ${\bf w}$ does not vanish on the boundary, we introduce $w=\mathbf{w}-u_{e}-[u_{b}-u_{e}]e^{-\eta }$. Hence, it follows that $w$ vanishes at both $%
y=0$ and $y=\infty ,$ and there holds
\begin{equation}\label{eqs-w}
{w}_{x}=[\mathbf{w}{w}_{\eta }]_{\eta }-[u_{b}-u_{e}][ we^{-\eta }]_{\eta } - F_\eta , \qquad F(\eta ): = [u_{b}-u_{e}]
[u_{e}+[u_{e}-u_{b}]e^{-\eta }]e^{-\eta }.
\end{equation}%
Clearly, $\langle \eta  \rangle^n F(\cdot) \in W^{k,p}(\RR_+)$, for arbitrary $n,k \ge 0$ and  $p \in [1,\infty]$. We shall solve this equation via the standard contraction mapping.

First, let us derive a priori weighted estimates. We introduce the following weighted iterative norm:
\begin{equation}\label{def-Nj} \cN_j(x): = \sum_{k=0}^j \int \langle \eta \rangle ^{n}| \partial_x^k{w}|^{2}+\sum_{k=0}^j\int_0^x \int \langle \eta \rangle ^{n}|%
\partial_x^k {w}_{\eta }|^{2}, \qquad j\ge 0.\end{equation}
By multiplying the equation \eqref{eqs-w} by $\langle \eta \rangle^n  w$, it follows the standard weighted energy estimate:
\begin{equation*}
\frac{1}{2}\frac{d}{dx}\int \langle \eta \rangle ^{n}| {w}|^{2}+\int \langle
\eta \rangle ^{n} {\bf w}| {w}_{\eta }|^{2}\le  \int \Big[ 
\langle \eta \rangle ^{n-1} | w| | w_\eta | + \langle \eta  \rangle^n e^{-\eta } | w  w_\eta | + \langle \eta  \rangle^n | w F_\eta | \Big],  
\end{equation*}%
for $n\ge 0$, which together with the Young's inequality yields 
\begin{equation*}
\frac{1}{2}\frac{d}{dx}\int \langle \eta \rangle ^{n}| {w}|^{2}+\int \langle
\eta \rangle ^{n} {\bf w}| {w}_{\eta }|^{2}\le C \int 
\langle \eta \rangle^n | w|^2 + C \int \langle \eta  \rangle^n |F_\eta |^2.
\end{equation*}%
The Gronwall inequality then yields 
\begin{equation}\label{L2bound-Pr0}
\sup_{0\le x \le L} \int \langle \eta \rangle ^{n}| {w}|^{2}+\int_0^L \int \langle \eta \rangle ^{n}|%
 {w}_{\eta }|^{2}\le C(L),
\end{equation}%
which gives $\cN_0(L)\le C(L)$, for some constant $C(L)$ that depends only on large $L$ and the give data $u_e, u_b, c_0$ in the problem.

Next, taking $x-$derivatives of \eqref{eqs-w}, we get  
\begin{equation*}
\partial _{x}^{j} {w}_{x}=\{{\bf w}\partial _{x}^{j} {w}_{\eta }\}_{\eta }+\sum_{\alpha =0}^j
C_{j}^{j-\alpha }\{\partial _{x}^{j-\alpha }\mathbf{w}\partial _{x}^{\alpha }%
 {w}_{\eta }\}_{\eta }-[u_{b}-u_{e}][\partial _{x}^{j} we^{-\eta }]_{\eta }
\end{equation*}%
for $j \ge 1$. Similarly as above, multiplying the equation by $\langle \eta \rangle ^{n}\partial _{x}^{j} {w}$ yields the inequality 
\begin{eqnarray}
&&\frac{1}{2}\frac{d}{dx}\int \langle \eta \rangle ^{n}|\partial _{x}^{j} {w}%
|^{2}+\int \langle \eta \rangle ^{n}\mathbf{w}|\partial _{x}^{j} {w}_{\eta }|^{2}
\notag\\
&\le  & n\int \langle \eta \rangle ^{n-1}\mathbf{w}\partial _{x}^{j} {w}%
_{\eta }\partial _{x}^{j} {w} + C\sum_{\alpha<j} \int \langle \eta \rangle ^{n}\{\partial
_{x}^{j-\alpha }\mathbf{w}\partial _{x}^{\alpha } {w}_{\eta }\}\partial
_{x}^{j} {w}_{\eta } \label{Hj-est} 
\\&&+C\sum_{\alpha<j}\int \langle \eta \rangle ^{n-1}\{\partial _{x}^{j-\alpha }%
\mathbf{w}\partial _{x}^{\alpha } {w}_{\eta }\}\partial _{x}^{j} {w} - [u_{b}-u_{e}]
\int[\partial _{x}^{j} we^{-\eta }]_\eta   \langle \eta \rangle^n \partial _{x}^{j}%
 {w} .\notag 
\end{eqnarray}
Let us treat each term on the right. For arbitrary positive constant $\delta$, we get 
\begin{eqnarray*}
\int \langle \eta \rangle ^{n-1}\mathbf{w}\partial _{x}^{j} {w}%
_{\eta }\partial _{x}^{j} {w} &\le& \delta \int \mathbf{w}\langle \eta \rangle ^{n}|\partial _{x}^{j} {%
w}_{\eta }|^{2}+C_\delta\|\mathbf{w}\|_\infty\int \langle \eta \rangle ^{n}|\partial _{x}^{j}%
 {w}|^{2}
\\ 
\int[\partial _{x}^{j} we^{-\eta }]_\eta   \langle \eta \rangle^n \partial _{x}^{j}%
 {w} &\le &\delta \int \mathbf{w}\langle \eta \rangle ^{n}|\partial _{x}^{j} {%
w}_{\eta }|^{2} + C_\delta \int \langle \eta \rangle ^{n}|\partial _{x}^{j}%
 {w}|^{2} .
\end{eqnarray*}
By choosing $\delta$ sufficiently small, the first term in the above inequalities can be absorbed into the left-hand side of the inequality \eqref{Hj-est}.  Next, for $0<\alpha<j$, we have
\begin{eqnarray*}
\int \langle \eta \rangle ^{n}\{\partial
_{x}^{j-\alpha }\mathbf{w}\partial _{x}^{\alpha } {w}_{\eta }\}\partial
_{x}^{j} {w}_{\eta } &\le & \delta \int \mathbf{w}\langle \eta \rangle ^{n}|\partial _{x}^{j} {%
w}_{\eta }|^{2} + C_\delta
\|\partial _{x}^{j-\alpha }\mathbf{w}\|^2_{\infty }
\int \langle \eta \rangle ^n{\bf w}|\partial _{x}^{\alpha} {w}%
_{\eta }|^2 
\\
\int \langle \eta \rangle ^{n-1}\{\partial _{x}^{j-\alpha }%
\mathbf{w}\partial _{x}^{\alpha } {w}_{\eta }\}\partial _{x}^{j} {w} &\le & C\|\partial _{x}^{j-\alpha }\mathbf{w}\|_{\infty }
\|\sqrt{\langle \eta \rangle^n {\bf w}}\partial _{x}^{\alpha } {w}%
_{\eta }\|_{L^2(\RR_+)} \|\langle \eta \rangle ^{\frac{n-2}{2}}\partial _{x}^{j} {w}\|_{L^2(\RR_+)}
\\
&\le & \delta \int \mathbf{w}\langle \eta \rangle ^{n}|\partial _{x}^{\alpha} {%
w}_{\eta }|^{2} + C_\delta
\|\partial _{x}^{j-\alpha }\mathbf{w}\|^2_{\infty }
\int \langle \eta \rangle ^n |\partial _{x}^{j} {w}|^2 
\end{eqnarray*}%
Whereas in the case $\alpha =0$, we instead estimate
\begin{eqnarray*}
\int \langle \eta \rangle ^{n}\{\partial
_{x}^{j }\mathbf{w} {w}_{\eta }\}\partial
_{x}^{j} {w}_{\eta } &\le & \delta \int \mathbf{w}\langle \eta \rangle ^{n}|\partial _{x}^{j} {%
w}_{\eta }|^{2} + 
C_\delta\| w_\eta \|^2_{\infty }
\int \langle \eta \rangle ^n|\partial_x^j {\bf w}|^2 
\\
\int \langle \eta \rangle ^{n-1}\{\partial _{x}^{j}%
\mathbf{w} {w}_{\eta }\}\partial _{x}^{j} {w} &\le &\|w_\eta \|_{\infty }
\|\sqrt{\langle \eta \rangle^n} \partial _{x}^{j} {\bf w} \|_{L^2(\RR_+)} \|\langle \eta \rangle ^{\frac{n-2}{2}}\partial _{x}^{j} {w}\|_{L^2(\RR_+)}.
\end{eqnarray*}%

It remains to give bounds on $\|w_\eta \|_\infty$ and $\|\partial _{x}^{j-\alpha }\mathbf{w}\|_{\infty }$, for $0<\alpha \le j$. We recall the definition $\mathbf{w} = w+ u_{e}+ [u_{b}-u_{e}]e^{-\eta }$. Using the Sobolev embedding, we get
\begin{eqnarray*}
\|\partial _{x}^{j-\alpha }\mathbf{w}\|^2_{\infty } 
&\le & C + \|\partial _{x}^{j-\alpha } w\|^2_{\infty } \quad \le\quad  C +  \|\partial _{x}^{j-\alpha } w\|_{L^2(\RR_+)} \|\partial _{x}^{j-\alpha } w_\eta \|_{L^2(\RR_+)}
\\&\le & C +  \|\partial _{x}^{j-\alpha } w\|_{L^2(\RR_+)} \Big[ \|\partial _{x}^{j-\alpha } {w_\eta }_{\vert_{x=0}}\|_{L^2(\RR_+)} +  \|\partial _{x}^{j-\alpha } w_\eta \|_{L^2} ^{1/2} \|\partial _{x}^{j-\alpha +1} w_\eta \|_{L^2}^{1/2} \Big]
\end{eqnarray*}%
in which $\| \cdot \|_{L^2}$ denotes the usual $L^2$ norm over $[0,x]\times \RR_+$. In term of the iterative norm $\cN_j(x)$, this yields \begin{eqnarray}\label{sup-bfw}
\|\partial _{x}^{j-\alpha }\mathbf{w}\|^2_{\infty } 
&\leq &C+ \cN_{j-\alpha}^{1/2}(x)  \Big[ \|\partial _{x}^{j-\alpha } {w_\eta }_{\vert_{x=0}}\|_{L^2(\RR_+)} +  \cN_{j}^{1/2}(x) \Big] ,
\end{eqnarray}%
for $0<\alpha \le j$. Next, to estimate $\| w_\eta \|_\infty$, we use the embedding: 
$ \|w_\eta  \|_{\infty}^2 \le C\| w_\eta \|_{L^2(\RR_+)} \| w_{\eta }\|_{H^1(\RR_+)}$. Using the equation \eqref{eqs-w} and the lower bound on ${\bf w}$, we get 
$$ \| w_{\eta \eta }\|_{L^2(\RR_+)} \le C  \Big(  \| w_x\|_{L^2(\RR_+)}  + \| w\|_{H^1(\RR_+)}+\| w_\eta ^2 \|_{L^2(\RR_+)} + \| F_\eta \|_{L^2(\RR_+)}\Big)$$
in which the Sobolev embedding for the supremum norm again yields 
$$\| w_\eta ^2 \|_{L^2(\RR_+)} \le C \| w_\eta  \|_{L^2(\RR_+)}^{3/2} \| w_\eta \|_{H^1}^{1/2} \le \delta \| w_{\eta \eta }\|_{L^2(\RR_+)} + C_\delta\| w_\eta \|_{L^2(\RR_+)}^3.$$ 
Choosing $\delta$ sufficiently small, we conclude that 
\begin{equation}\label{sup-wz}\| w_{\eta }\|_\infty \quad \lesssim \quad 1 +  \| w_x\|_{L^2(\RR_+)}  + \| w\|_{H^1(\RR_+)}+\| w_\eta  \|^2_{L^2(\RR_+)} .\end{equation}
Hence, integrating the above inequality over $[0,x]$, recalling the definition of the iterative norm, and using the uniform bound on $\cN_0(L)$, we obtain 
\begin{equation}\label{L2sup-wz}
\begin{aligned}
\int_0^x\| w_{\eta }\|^2_\infty 
&\lesssim 1+  \int_0^x (\| w_x\|^2_{L^2(\RR_+)}  + \| w\|^2_{H^1(\RR_+)})+\sup_{s\in [0,x]} \| w_\eta \|^2_{L^2(\RR_+)}\int_0^x \| w_\eta  \|^2_{L^2(\RR_+)}
\\
&\lesssim 1+  \int_0^x \| w_x\|^2_{L^2(\RR_+)}  +\| {w_\eta }_{\vert_{x=0}}\|^2_{L^2(\RR_+)} + \| w_{x\eta }\|_{L^2}^2
\\
&\lesssim 1+ \| {w_\eta }_{\vert_{x=0}}\|^2_{L^2(\RR_+)}  +  \int_0^x \cN_1(s)\; ds.
\end{aligned}
\end{equation}

Putting the above estimates altogether into the $j^{th}$ weighted estimates \eqref{Hj-est}, integrating the result over $[0,x]$ and rearranging terms, we obtain 
\begin{equation}\label{iter-Pr0}\begin{aligned}
 N_j(x) &\lesssim \int_0^x(1+ \|\mathbf{w}\|_\infty + \| w_\eta \|_\infty^2) \int \langle \eta \rangle^n|\partial _{x}^{j}%
 {w}|^{2} + \|\partial _{x}^{j-\alpha }\mathbf{w}\|^2_{\infty } \sum_{\alpha <j} \iint \langle \eta \rangle ^n{\bf w}|\partial _{x}^{\alpha} {w}%
_{\eta }|^2 
 \\
 & \quad+ \|\partial _{x}^{j-\alpha }\mathbf{w}\|_{\infty } \sum_{\alpha<j}
\|\sqrt{\langle \eta \rangle^n {\bf w}}\partial _{x}^{\alpha } {w}%
_{\eta }\|_{L^2} \|\langle \eta \rangle ^{\frac{n-2}{2}}\partial _{x}^{j} {w}\|_{L^2} 
\\
&\lesssim \cN_j(0) + \int_0^x \Big[ 1+ \cN_1(s) +  \cN_{j-1}(x) \Big]   \cN_j(s) \; ds .
 \end{aligned}\end{equation}
The Gronwall inequality then yields 
$$ \cN_j(L) \le C \cN_j(0),\qquad \forall j\ge 0,$$
for $L$ sufficiently small. Here, we note that the smallness of $L$ and the constant $C$ depend only on the given data in the problem. The standard contraction mapping, together with a priori bounds, yields the existence and the uniform bound of the solutions to \eqref{eqs-w} in $[0,L]\times \RR_+$. Changing back to the original coordinates yields the lemma, upon noting that $y \sim \eta $ thanks to the upper and lower bound of ${\bf w}$. 
\end{proof}

\begin{remark}
\emph{It is possible to iterate our above scheme to obtain a global-in-$x$
solution to the Prandtl equation. Indeed, the $L^2$ estimate %
\eqref{L2bound-Pr0} yields the global existence of a bounded weak solution.
The standard Nash-Moser's iteration applied to the parabolic equation %
\eqref{eqs-w} then yields a uniform bound in $C^1$ and hence $H^1$ spaces.
By a view of the iterative estimate \eqref{iter-Pr0} which is only nonlinear
at the first step for $\mathcal{N}_1$, it follows a uniform bound for all $%
\mathcal{N}_k$, for $k \ge 1$, uniformly in small $L$. This yields the
global smooth solution. }
\end{remark}

\subsection{$\protect\varepsilon ^{1/2}$-order corrections}

\label{sec-1Euler} Next, we collect all terms with a factor $\sqrt{%
\varepsilon }$ from (\ref{tangential}), together with the new $\sqrt{%
\varepsilon }$-order terms arising from $R^{u,0}$ (see \eqref{Ru0}), to get 
\begin{equation*}
\begin{aligned} R^{u,1}&:=[u_{e}^{1}+u_{p}^{1}]\partial_x
[u_e^0+u_{p}^{0}]+[u_{e}^{0}+u_{p}^{0}]\partial
_{x}[u_{e}^{1}+u_{p}^{1}]+v_{p}^{1}\partial
_{y}[u_{e}^{0}+u_{p}^{0}]+[v_{p}^{0}+v_{e}^{1}]\partial_y
[u_{e}^{1}+u_{p}^{1}] \\ &\quad +p_{ex}^{1} + p_{px}^1 -\partial_y^2 [u_e^1
+
u_{p}^{1}]+[yu_{px}^{0}+v_{p}^{0}+v_{e}^{1}]u_{ez}^{0}+yv_{ez}^{1}u_{py}^{0}. \end{aligned}
\end{equation*}%
We construct the Euler and Prandtl layers so that $R^{u,1}$ is of order $%
\sqrt{\varepsilon }$. We rearrange terms with respect to the interior
variables $(x,\sqrt{\varepsilon }y)$ and the boundary-layer variables $(x,y)$%
, respectively. We stress that when the partial derivative $\partial _{y}$
hits an interior term with scaling $\sqrt{\varepsilon }y$, that term can be
moved to the next order. For instance, $[v_{p}^{0}+v_{e}^{1}]\partial
_{y}u_{e}^{1}=\sqrt{\varepsilon }[v_{p}^{0}+v_{e}^{1}]u_{ez}^{1}(\sqrt{%
\varepsilon }y)$. Having this in mind, the leading interior terms consist of 
\begin{equation}
u_{e}^{0}u_{ex}^{1}+v_{e}^{1}u_{ez}^{0}+p_{ex}^{1}=0  \label{Euler1-u}
\end{equation}%
and the boundary-layer terms consisting of 
\begin{equation}
\begin{aligned} ~[u_{e}^{1}+u_{p}^{1}]u_{px}^{0}
&+u_{p}^{0}u_{ex}^{1}+[u_{e}^{0}+u_{p}^{0}]u_{px}^{1}+v_{p}^{1}[ u_{py}^{0}
+ \sqrt \veps u_{ez}^0]+[v_{p}^{0}+v_{e}^{1}]u_{py}^{1} + p_{px}^1
-u_{pyy}^{1}\\ &+ [yu_{px}^{0}+v_{p}^{0}]u_{ez}^{0}+yv_{ez}^{1}u_{py}^{0} =
0, \end{aligned}  \label{Prandtl1-u}
\end{equation}%
in which the equalities are made to precisely get rid of these leading
terms. Hence, having constructed these layers, the error is then reduced to 
\begin{equation}
\sqrt{\varepsilon }[v_{p}^{0}+v_{e}^{1}]u_{ez}^{1} - \varepsilon u_{ezz}^1.
\label{est-Ru1}
\end{equation}

Next, let us consider the normal component (\ref{normal}). Clearly, the
leading term is $\frac{1}{\sqrt \varepsilon} p_{py}^1$, which leads to the
fact that Prandtl's pressure is independent of $y$: 
\begin{equation}  \label{Prandtl1-v}
p_p^1 = p_p^1(x).
\end{equation}
The next (zeroth) order in \eqref{normal} consists of 
\begin{equation*}
R^{v,0}: = [u_{e}^{0}+u_{p}^{0}] [
v_{px}^{0}+v_{ex}^{1}]+[v_{p}^{0}+v_{e}^{1}]\partial_y [v_{p}^{0}+ v_e^1]
+p_{ez}^{1} + p^2_{py}- \partial_y^2 [v_{p}^{0} + v_e^1].
\end{equation*}%
Again as above, we shall enforce $R^{v,0} =0$ (possibly, up to error of
order $\sqrt \varepsilon$). Note that $v_p^1$ has now been determined
through the divergence-free condition and the construction of $u_p^1$. We
take the interior layer $[u_e^1, v_e^1, p_e^1]$ to satisfy 
\begin{equation}  \label{Euler1-v}
u_{e}^{0} v_{ex}^{1}+p_{ez}^{1} =0.
\end{equation}
Whereas, the next layer pressure $p_p^2$ is taken to be of the form 
\begin{equation}  \label{def-pressure2}
p_p^2(x,y) = \int_y^\infty \Big[ [u_{e}^{0}+u_{p}^{0}] v_{px}^{0} + u_p^0
v_{ex}^1+[v_{p}^{0}+v_{e}^{1}][v_{py}^{0} + \sqrt \varepsilon v_{ez}^1] -
v_{pyy}^{0} - \varepsilon v_{ezz}^1 \Big] (x,\theta)\; d\theta.
\end{equation}
With this choice of $p_p^2$ and \eqref{Euler1-v}, the error term $R^{v,0}$
in this leading order is reduced to 
\begin{equation}  \label{est-Rv0}
R^{v,0} = \sqrt \varepsilon [v_{p}^{0}+v_{e}^{1}]v_{ez}^1 - \varepsilon
v_{ezz}^1.
\end{equation}
The set of equations \eqref{Euler1-u}, \eqref{Euler1-v}, together with the
divergence-free condition, constitutes the profile equations for the Euler
correction $[u_e^1, v_e^1, p_e^1]$, whereas equations \eqref{Prandtl1-u} and %
\eqref{Prandtl1-v} are for the divergence-free Prandtl layers $[u_p^1,
v_p^1, p_p^1]$.

\subsection{Euler correctors}

We construct the Euler corrector $[v_{e}^{1}{},u_{e}^{1},p_{e}^{1}]$ solving %
\eqref{Euler1-u}, \eqref{Euler1-v}. For sake of presentation, we drop the
superscript $1$. Writing the equation for vorticity $w_{e}=u_{ez}-v_{ex}$
and note that $w_{ex}=-\Delta v_{e}$. This leads to the following elliptic
problem for $v_{e}$: 
\begin{equation}
\begin{aligned} - u^{0}_e \Delta v_e + u_{ezz}^{0} v_e &=0,\end{aligned}
\label{eqs-ve1}
\end{equation}%
with $\Delta =\partial _{x}^{2}+\partial _{z}^{2}$, together with the
boundary conditions as in Section \ref{sec-BCs}, which we recall 
\begin{equation}
v_{e}(x,0)=-v_{p}^{0}(x,0),\qquad v_{e}(0,z)=V_{b0}(z),\qquad
v_{e}(L,z)=V_{bL}(z)  \label{E-BCs}
\end{equation}%
with the comparability assumption: $V_{b0}(0)=-v_{p}^{0}(0,0)$ and $%
V_{bL}(0)=-v_{p}^{0}(L,0)$. We need to derive higher regularity estimates
for $v_{e}$. Due to the presence of corners in the domain for $(x,z)$,
singularity could occur. To avoid this, we instead consider the following
elliptic problem: 
\begin{equation}
\begin{aligned} - u^{0}_e \Delta v_e + u_{ezz}^{0} v_e &=E_b,\end{aligned}
\label{eqs-ve1-mod}
\end{equation}%
with the same boundary conditions \eqref{E-BCs}, in which $E_{b}$ is
introduced as a boundary layer corrector. To define $E_{b}$, let us
introduce 
\begin{equation}
\begin{aligned} B(x,z) &: = (1-\frac xL)
\frac{V_{b0}(z)}{v_p^0(0,0)}v_p^0(x,0) + \frac xL \frac{
V_{bL}(z)}{v_p^0(L,0)} v^0_p(x,0) \\
&=\frac{V_{b0}(z)}{v_p^0(0,0)}v_p^0(x,0) + \frac xL\Big( \frac{
V_{bL}(z)}{v_p^0(L,0)} - \frac{V_{b0}(z)}{v_p^0(0,0)} \Big)
v_p^0(x,0).\end{aligned}  \label{def-Bw}
\end{equation}%
Here, without loss of generality, assume that $v_{p}^{0}(0,0),v_{p}^{0}(L,0)%
\not=0$. It is clear that $B(x,z)$ satisfies the boundary conditions %
\eqref{E-BCs}, thanks to the compatibility assumption at the corners. In
addition, if we assume $|\partial _{z}^{k}(V_{bL}-V_{b0})(z)|\leq CL$, it
follows that $B\in W^{k,p}$ for arbitrary $k,p$. In particular, the function 
\begin{equation*}
F_{e}(x,z):=-u_{e}^{0}\Delta B+u_{ezz}^{0}B
\end{equation*}%
is arbitrarily smooth and there holds 
\begin{equation}
\Vert \langle z\rangle ^{n}F_{e}\Vert _{W^{k,q}}\leq C,  \label{bound-Fve}
\end{equation}%
for any $n,k\geq 0$ and $q\in \lbrack 1,\infty ]$, for some constant $C$
that is independent of small $L$. Let us then introduce the function $w$
through 
\begin{equation*}
v_{e}=B+w.
\end{equation*}%
The function $w$ solves the following elliptic problem with homogenous
boundary conditions: 
\begin{equation}
\begin{aligned} - u^{0}_e \Delta w + u_{ezz}^{0} w &=E_b + F_e(x,z),\qquad
w_{\vert_{\partial \Omega}} = 0.\end{aligned}  \label{eqs-we1}
\end{equation}%
To obtain high regularity for $w$, we introduce the boundary layer
corrector: 
\begin{equation}
E_{b}(x,z):=-\chi (\frac{z}{\varepsilon })F_{e}(x,0),  \label{def-Eb}
\end{equation}%
in which $\chi (\cdot )$ is a cut-off function with support in $[0,1]$ and
with $\chi (0)=1$. It follows that 
\begin{equation}
\Vert \langle z\rangle ^{n}\partial _{z}^{k}E_{b}\Vert _{L^{q}}\leq
C\varepsilon ^{-k+\frac{1}{q}},\qquad q\geq 1,\qquad n,k\geq 0.
\label{est-Eb}
\end{equation}

Let us now derive sufficient estimates on $v_e$. We prove the following:

\begin{lemma}
\label{lem-ve} Assume that $V_{b0}$ and $V_{bL}$ are sufficiently smooth,
rapidly decaying at infinity, and satisfy $|\partial
_{z}^{k}(V_{bL}-V_{b0})(z)|\leq CL$, for $k\geq 0$, uniformly in $z$. There
exists a unique smooth solution $v_{e}$ to the elliptic problem \eqref{E-BCs}%
, \eqref{eqs-ve1-mod}, and \eqref{def-Eb}, and there holds 
\begin{equation*}
\Vert v_{e}\Vert _{\infty }+\Vert \langle z\rangle ^{n}v_{e}\Vert
_{H^{2}}\leq C_{0},\qquad \Vert \langle z\rangle ^{n}v_{e}\Vert _{H^{3}}\leq
C_{0}\varepsilon ^{-1/2},\qquad \Vert \langle z\rangle ^{n}v_{e}\Vert
_{H^{4}}\leq C_{0}\varepsilon ^{-3/2}
\end{equation*}%
for $n\geq 0$ and for some constant $C_{0}$ that depends on the given
boundary data, and but does not depend on $L$, when $L$ is small. In
addition, 
\begin{equation*}
\Vert \langle z\rangle ^{n}v_{e}\Vert _{W^{2,q}}\leq C(L),\qquad \Vert
\langle z\rangle ^{n}v_{e}\Vert _{W^{3,q}}\leq C(L)\varepsilon
^{-1+1/q},\qquad \Vert \langle z\rangle ^{n}v_{e}\Vert _{W^{4,q}}\leq
C(L)\varepsilon ^{-2+1/q},
\end{equation*}%
for $n\geq 0$, $q\in (1,\infty )$, and for some constant $C(L)$ that could
depend on small $L$. Here, we note that the integration in the above $\Vert
\cdot \Vert _{W^{k,q}}$ norms is taken with respect to $(x,z)$ in $%
[0,L]\times \mathbb{R}_{+}$.
\end{lemma}

\begin{proof}
We write $v_{e}=B+w$ with the smooth $B$
defined as in \eqref{def-Bw}. It suffices to derive estimates on $w$,
solving \eqref{eqs-we1}. We first perform the basic $L^{2}$ estimate.
Multiplying the equation by $w/u_{e}^{0}$ and using the zero boundary
conditions yield 
\begin{equation*}
\iint \Big(-\Delta ww+\frac{u_{ezz}^{0}}{u_{e}^{0}}|w|^{2}\Big)=\iint \Big(%
|\nabla w|^{2}+\frac{u_{ezz}^{0}}{u_{e}^{0}}|w|^{2}\Big)=\iint \frac{%
w(E_{b}+F_{e})}{u_{e}^{0}}.
\end{equation*}%
Thanks to the crucial positivity estimate (see \eqref{positivity-intro} and %
\eqref{low-vy} with $u_{e}^{0}$), we have 
\begin{equation*}
\int \Big(|w_{z}|^{2}+\frac{u_{ezz}^{0}}{u_{e}^{0}}|w|^{2}\Big)=\int
|u_{e}^{0}|^{2}\left\{ \frac{w}{u_{e}^{0}}\right\} _{z}^{2}\geq \theta
_{0}\int |w_{z}|^{2},
\end{equation*}%
in which we have used $|w/u_{0}^{e}|\leq \sqrt{z}\Vert
(w/u_{e}^{0})_{z}\Vert _{L^{2}(\mathbb{R}_{+})}$and the fast decay property
of $u_{ez}^{0}$ to obtain the lower bound, for some constant $\theta _{0}$
independent of $L$. In addition, we estimate 
\begin{equation*}
\iint \frac{w(E_{b}+F_{e})}{u_{e}^{0}}\leq \iint \frac{\sqrt{x}|E_{b}+F_{e}|%
}{u_{e}^{0}}\Vert w_{x}\Vert _{L^{2}(0,L)}\leq C\Vert E_{b}+F_{e}\Vert
_{L^{2}}\Vert \nabla w\Vert _{L^{2}}\leq C\Vert \nabla w\Vert _{L^{2}}.
\end{equation*}%
Putting the above estimates into the energy estimate, together with a use of
the standard Young's inequality, we get 
\begin{equation}
\Vert w\Vert _{H^{1}}\leq C,  \label{H1-ve1}
\end{equation}%
for some constant $C$ that is independent of small $L$, in which the $L^{2}$
norm of $w$ is bounded by the Poincare's inequality.

Next, to derive high order energy estimates, we write the equation as 
\begin{equation}  \label{eqs-wG}
- \Delta w = G_e, \qquad G_e: = \frac{1}{u_e^0} \Big( E_b + F_e - u_{ezz}^0 w%
\Big).
\end{equation}
Clearly, $\| G_e\|_{L^2} \le C$. In addition, since $E_b(x,0) + F_e(x,0) = 0$%
, we have $G_e = 0$ and hence $w = w_{zz} =0$ on the boundary $z=0$. We thus
have the following $H^2$ energy estimate by multiplying the elliptic
equation by $w_z$: 
\begin{equation*}
\iint |\nabla w_z|^2 + \int_0^Lw_{zz} w_z(x,0) + \int_0^\infty w_{zx}w_z\Big|%
_{x=0}^{x=L} = \iint (G_e)_z w_z = - \iint G_e w_{zz} \le \| G_e\|_{L^2}
\|w_{zz}\|_{L^2} .
\end{equation*}
Since $w_{zz}(x,0) = 0$ and $w_z(0,z) = w_z(L,z) =0$, the boundary terms
vanish. Hence, together with a use of the Young's inequality, we have
obtained the uniform bound $\| w_z\|_{H^1} \le C$. Using the equation to
estimate $w_{xx}$ in term of the rest, we thus obtain the full $H^2$ bound
of the solution $w$, and hence of $v_e$, uniformly in small $L$. In
addition, since $w =0$ on the boundary, we have 
\begin{equation*}
\begin{aligned}
 |w(x,z)| & \le \int_0^x |w_x(\theta, z)|\; d\theta \le \int_0^x \Big( \int_0^z |w_x w_{xz}| (\theta, \eta)\; d\eta \Big)^{1/2} \; d\theta 
 \\
 &\le \sqrt x \| w_x\|_{L^2}^{1/2} \| w_{xz}\|_{L^2}^{1/2} \le C_0 \sqrt L, 
 \end{aligned}
\end{equation*}
thanks to the $H^2$ bound on $w$. This proves the uniform boundedness of $v_e
$.

As for the weighted estimates, we consider the elliptic problem for $\langle
z \rangle^n w$, with $n\ge 1$, which solves 
\begin{equation*}
- \Delta (\langle z \rangle^n w) = \frac{\langle z \rangle^n }{u_e^0} \Big( %
E_b + F_e - u_{ezz}^0 w\Big)  - w \partial_z^2 \langle z \rangle^n - 2 w_z
\partial_z \langle z \rangle^n
\end{equation*}
the homogenous boundary conditions. By induction, $\langle z \rangle^{n-1}w$
is uniformly bounded in $H^2$, and hence the right-hand side of the above
elliptic problem is uniformly bounded in $H^1$. The same proof given just
above for the unweighted norm yields $\| \langle z \rangle^n w\|_{H^2} \le C$%
, for all $n \ge 1$.

Next, we derive higher regularity estimates for $w$. We recall that from %
\eqref{bound-Fve} and \eqref{est-Eb}, there holds 
\begin{equation*}
\|\langle z \rangle^n \partial_z^k (E_b+F_e)\|_{L^q} \le C (1+\varepsilon
^{-k+\frac 1q}), \qquad q\ge 1, \qquad n,k\ge 0.
\end{equation*}
Let us now consider the elliptic problems for $w_z$ and $w_{zz}$: 
\begin{equation*}
-\Delta w_{z} = \partial_z \Big[ \frac{1}{u_e^0} \Big( F_e + E_b - u_{ezz}^0
w\Big) \Big], \qquad {w_z}_{\vert_{x=0,L}} = {w_{zz}}_{\vert_{z=0}} =0
\end{equation*}
and 
\begin{equation*}
-\Delta w_{zz} = \partial_z^2 \Big[ \frac{1}{u_e^0} \Big( F_e + E_b -
u_{ezz}^0 w\Big) \Big], \qquad {w_{zz}}_{\vert_{\partial\Omega}} =0.
\end{equation*}
Here, we note that $w_{zz}=0$ on the boundary, precisely due to the layer
corrector $E_b$ and the equation \eqref{eqs-wG}. Next, note that the source
term in the above equations has its $L^2$ norm bounded by $%
C\varepsilon^{-1/2}$ and $C \varepsilon^{-3/2}$, respectively. Again, the
above $H^2$ energy estimates then give 
\begin{equation*}
\|\langle z \rangle^n \partial_z^kw\|_{H^2} \le C \varepsilon^{-k+1/2},
\qquad k=1,2, \quad n\ge 0.
\end{equation*}
We now estimate $w$ in $H^3$ and $H^4$ norms. Indeed, thanks to the $H^2$
bound on $w_z, w_{zz}$, it remains to estimate $w_{xxx}$ in $L^2$ and $H^1$,
respectively. Thanks to \eqref{eqs-wG}, we may write 
\begin{equation*}
w_{xxx} = - w_{zzx} + \Delta w_x = - w_{zzx} - \partial_x \Big[ \frac{1}{%
u_e^0} \Big( F_e + E_b - u_{ezz}^0 w\Big) \Big].
\end{equation*}
This yields at once the desired weighted $L^2$ and $H^1$ estimates on $%
w_{xxx}$, and hence the full weighted $H^3$ and $H^4$ estimates on $w$.

Finally, the $W^{k,q}$ estimates follow simply from the standard elliptic
theory in $[0,L]\times \RR$, when we make the odd extension to $z<0$
for (\ref{eqs-wG}). We note that the boundary layer construction (\ref%
{def-Eb}) ensures that the odd extension of $G_{e}\in W^{2,q}([0,L]\times 
\RR)$. This completes the proof of the lemma.
\end{proof}

\subsubsection{Euler profiles}

\label{sec-defE1}

We now construct the Euler corrector $[u_{e}^{1},v_{e}^{1},p_{e}^{1}]$ that
is used in the boundary-layer expansion \eqref{expansion}. We take $%
v_{e}^{1}=v_{e}$, where $v_{e}$ solves the modified elliptic problem %
\eqref{eqs-ve1-mod}, with an extra source $E_{b}$. By a view of %
\eqref{Euler1-v} and the divergence-free condition, we take 
\begin{equation}
\begin{aligned}u^1_{e} &:= u_b^1 (z)-\int_{0}^{x}v^1_{ez}(\xi ,z)d\xi , \\
p^1_{e} &:=p_b-\int_{0}^{x}[-u_{e}^{0}v^1_{ez}+v_{e}^1u_{ez}^{0}](\xi
,0)d\xi -\int_{0}^{z}[u_{e}^{0}v^1_{ex}](x,\theta )d\theta , \end{aligned}
\label{def-Epressure}
\end{equation}%
for any constant $p_{b}$. Without loss of generality, we take $p_{b}=0$.
Clearly, by definition and the uniform $H^{2}$ estimates on $v_{e}^{1}$
(Lemma \ref{lem-ve}), we have 
\begin{equation}
\Vert u_{e}^{1}\Vert _{\infty }+\Vert \langle z\rangle ^{n}u_{e}^{1}\Vert
_{H^{1}}\leq C_{0},\qquad n\geq 0.  \label{est-ue1}
\end{equation}%
By construction, $[u_{e}^{1},v_{e}^{1},p_{e}^{1}]$ solves \eqref{Euler1-v},
the divergence-free condition, and instead of \eqref{Euler1-u}, the equation 
\begin{equation}
u_{e}^{0}u_{ex}^{1}+v_{e}^{1}u_{ez}^{0}+p_{ex}^{1}=-\int_{z}^{\infty
}E_{b}(x,\theta )\;d\theta .  \label{Euler1-u-mod}
\end{equation}%
As compared to \eqref{Euler1-u}, this contributes a new error term into %
\eqref{est-Ru1}, which is now defined as 
\begin{equation}
\begin{aligned} R^{u,1}&= \sqrt{\varepsilon }[v_{p}^{0}+v_{e}^{1}]u_{ez}^{1}
- \varepsilon u_{ezz}^1 + \int_z^\infty E_b(x,\theta) \; d\theta.
\end{aligned}  \label{def-Ru1-mod}
\end{equation}

To give an estimate on the error term, we first note that throughout the
paper we work with the coordinates $(x,y)$, whereas the Euler flows are
evaluated at $(x,z)=(x,\sqrt \varepsilon y)$. Thanks to Corollary \ref%
{cor-Prandtl0}, the boundedness of $v_e$ and \eqref{est-ue1}, we have 
\begin{equation*}
\| \sqrt{\varepsilon }[v_{p}^{0}+v_{e}^{1}]u_{ez}^{1}\|_{L^2} \le \sqrt
\varepsilon \| v_p^0 + v_e^1\|_\infty \| u_{ez}^1(\sqrt \varepsilon
\cdot)\|_{L^2} \le C\varepsilon^{1/4} \| u_{ez}^1(\cdot)\|_{L^2} .
\end{equation*}
Similarly, by definition and the estimates from Lemma \ref{lem-ve}, we have 
\begin{equation*}
\varepsilon \| u_{ezz}^1(\sqrt \varepsilon \cdot)\|_{L^2} \le
\varepsilon^{3/4} \| u^1_{bz}\|_{L^2} + C\varepsilon^{3/4}\|\langle z
\rangle^n v^1_{ezzz} \|_{L^2} \le C\varepsilon^{1/4},
\end{equation*}
and by the estimate \eqref{est-Eb}, 
\begin{equation*}
\Big\| \int_{\sqrt \varepsilon y}^\infty E_b(x,\theta) \; d\theta \Big\|%
_{L^2} \le C \varepsilon^{-1/4} \| \langle z \rangle^n E_b(\cdot)
\|_{L^2}\le C \varepsilon^{1/4}.
\end{equation*}
Hence, we obtain the uniform error estimate: 
\begin{equation}
\begin{aligned} \| R^{u,1}\|_{L^2} \le C \varepsilon^{1/4}. \end{aligned}
\label{est-Ru1-mod}
\end{equation}

Similarly, we give an estimate on $R^{v,0}$ defined as in \eqref{est-Rv0}.
Thanks to the boundedness of $v_p^0$ and $v_e^1$, and the estimates on $%
v_e^1 $ (Lemma \ref{lem-ve}), we get 
\begin{equation}  \label{bound-Rv0}
\begin{aligned} \| R^{v,0}\|_{L^2} &\le \sqrt \veps
\|v_{p}^{0}+v_{e}^{1}\|_{L^\infty } \|v_{ez}^1 \|_{L^2} + \varepsilon \|
v_{ezz}^1\|_{L^2} \le C \varepsilon^{1/4}. \end{aligned}
\end{equation}

Finally, we estimate $E^0$ defined as in \eqref{e0}. Using the fact that $%
u_p^0$ is rapidly decaying at infinity and $v^1_{ezz}$ is in $L^2$, we
obtain 
\begin{equation}  \label{est-E0}
\| E^{0}\|_{L^2} \le \varepsilon \| \langle y \rangle u_{px}^{0}\|_{L^2} \|
u_{ezz}^{0}\|_{L^\infty} +\varepsilon \| \langle y \rangle
u_{py}^{0}\|_{L^2} \| v_{ezz}^{1}\|_{L^2} \le C \varepsilon^{3/4}.
\end{equation}

\subsection{Prandtl correctors}

\label{sec-1Prandtl}

In this subsection, we shall construct Prandtl layer $[u_p^1, v_p^1, p_p^1]$%
, solving \eqref{Prandtl1-u}, \eqref{Prandtl1-v}. For convenience, let us
denote 
\begin{equation*}
u^{0}\equiv u_{e}^{0}(\sqrt{\varepsilon }y)+u_{p}^{0}.
\end{equation*}%
After rearranging terms, the equation \eqref{Prandtl1-u} for the $%
\sqrt\varepsilon$-order Prandtl corrector $[u_{p}{},u_{p},p_{p}]$, after
dropping the superscript $1$, becomes 
\begin{equation}  \label{prandtl1u}
\begin{aligned} u^{0}u_{px} &+ u_{p}u_{x}^{0} +v_{p}u^{0}_y +
[v_{p}^{0}+v_e^1]u_{py} -u_{pyy} \\&= - u_{ez}^{0}[yu_{px}^{0}+v_{p}^{0}]-
yv_{ez}^{1}u_{py}^{0} - u_{e}^{1}u_{px}^{0} -u_{p}^{0}u_{ex}^{1} - p_{px}
\\&=:F_p \end{aligned}
\end{equation}
in which we note that the source term $F_p$ includes the unknown pressure $%
p_{p}$. Thanks to \eqref{Prandtl1-v}, $p_p = p_p(x)$ and hence, by
evaluating the equation \eqref{prandtl1u} at $y = \infty$, we get $p_{px} =0$%
. We shall solve \eqref{prandtl1u} together with the divergence-free
condition 
\begin{equation*}
u_{px} + v_{py} =0
\end{equation*}
and the boundary conditions: 
\begin{eqnarray*}
u_p(0,y) = \bar u_1(y), \qquad u_{p}(x,0) = - u_{e}^{1}(x,0) , \qquad
v_{p}(x,0) = 0,\qquad \lim_{y\to \infty} u_p(x,y) = 0.
\end{eqnarray*}

\subsubsection{Construction of Prandtl layers}

There is a natural energy estimate associated with the linearized Prandtl
equation \eqref{prandtl1u}, yielding bound on $u_{py}$ in term of $v_p$. The
difficulty is in controlling the unknown $v_p$. Our construction starts with
the crucial positivity estimate (\ref{positivity-intro}). Indeed, we
introduce the inner product 
\begin{equation*}
\lbrack \lbrack v,w]]\equiv \int_{\mathbb{R}_+}\Big[ v_y w_y+\frac{u^0_{yy}}{%
u^0}vw\Big] \; dy.
\end{equation*}
Note that by \eqref{positivity-intro} the quantity $[[v,v]]$ yields a bound
on $\| v_y\|_{L^2(\mathbb{R}_+)}^2$. We therefore shall re-write the
equation in term of $v_p$, putting $u_p$ in the source term. Precisely,
taking $y$-derivative of \eqref{prandtl1u} and using the divergence-free
condition yield 
\begin{equation*}
\begin{aligned} - u^{0}v_{pyy} &+v_{p}u^{0}_{yy} + u_{p}u_{xy}^{0} +
[v_{p}^{0}+v_e^1]u_{pyy} + \sqrt \veps v^1_{ez}u_{py} -u_{pyyy} = F_{py},
\end{aligned}
\end{equation*}
or equivalently, in a view of the inner product, 
\begin{equation}  \label{prandtl1u-dy}
\begin{aligned} - &v_{pyy} +\frac{u^{0}_{yy}}{u^0} v_p -\partial_y^2
\Big(\frac{u_{py}}{u^0} \Big) \\&= \frac{1}{u^0}\Big( F_{py} -
u_{p}u_{xy}^{0} - [v_{p}^{0}+v_e^1]u_{pyy} - \sqrt \veps v^1_{ez}u_{py}
\Big) - 2 u_{pyy} \Big( \frac{1}{u^0}\Big)_y - u_{py}\Big(
\frac{1}{u^0}\Big)_{yy} \\&=: G_p.\end{aligned}
\end{equation}
Furthermore, taking $x$-derivative of \eqref{prandtl1u-dy} yields 
\begin{equation}  \label{prandtl1u-dxy}
\begin{aligned} - v_{pxyy} +\frac{u^{0}_{yy}}{u^0} v_{px} +
\Big(\frac{v_{pyy}}{u^0} \Big)_{yy} &= G_{px} - \Big(
\frac{u^{0}_{yy}}{u^0}\Big)_x v_p + \Big(u_{py} \Big(\frac{1}{u^0}
\Big)_x\Big)_{yy} \end{aligned}
\end{equation}
with $G_p$ defined as in \eqref{prandtl1u-dy}. We shall solve the problem %
\eqref{prandtl1u-dy}-\eqref{prandtl1u-dxy} for $v_p$, with $u_{px} +
v_{py}=0 $, and the boundary conditions: 
\begin{equation}
v_{p} (x,0)=0,\qquad v_{py}(x,0)=u_{ex}^{1}(x,0).  \label{prandtl1bc}
\end{equation}%
We prove the following:

\begin{lemma}
\label{lem-1stP} There exists a unique smooth divergence-free solution $%
[u_{p},v_{p}]$ solving \eqref{prandtl1u} with initial condition $u_{p}(0,y)=%
\bar{u}_{1}(y)$ and the boundary conditions \eqref{prandtl1bc}. Furthermore,
there hold 
\begin{equation}
\begin{aligned} \| [u_p, v_p]\|_{L^\infty} + \sup_{0\leq x\leq L} \| \langle
y \rangle^n v_{pyy} \|_{L^2(\RR_+)} + \| \langle y\rangle ^{n}v_{pxy}
\|_{L^2} &\lesssim C(L,\kappa)\varepsilon^{-\kappa}, \end{aligned}
\label{prandtl1bound}
\end{equation}%
for arbitrary small $\kappa $, and high regularity estimates 
\begin{equation}
\begin{aligned} \sup_{0\leq x\leq L} \| \langle y \rangle^n v_{pxyy}
\|_{L^2(\RR_+)} + \| \langle y\rangle ^{n}v_{pxxy} \|_{L^2} \lesssim
C(L)\varepsilon^{- 1} , \end{aligned}  \label{prandtl1bound-high}
\end{equation}%
uniformly in small $\varepsilon ,L$, in which the bounds depend only on the
constructed profiles $[u^{0},v_{p}^{0}]$, the given boundary data, and small 
$L$.
\end{lemma}

The proof consists of several steps. First, we express the boundary
conditions of $v_p$ in term of the given data $u_p(0,y) = \bar u_1(z)$.

\begin{lemma}
\label{lem-initialv} For smooth solutions $[u_p,v_p]$ solving %
\eqref{prandtl1u} with $u_p(0,y) = \bar u_1(y)$. Then, there hold 
\begin{equation*}
\begin{aligned} \|\langle y \rangle^n v_p(0,\cdot)\|_{\dot H^{k+1}(\RR_+)} &
\le C_0 \Big( 1 + \| \langle y \rangle^{n}\bar u_1\|_{H^{k+3}(\RR_+)}\Big)
\\ \|\langle y \rangle^n v_{px}(0,\cdot)\|_{\dot H^{k+1}(\RR_+)} & \le C_0
\Big( 1 + \| \langle y \rangle^n\bar u_1\|_{H^{k+3}(\RR_+)}+\| \langle
y\rangle^{-m}u_{exx}^1(0, \sqrt \veps \cdot)\|_{H^{k}(\RR_+)}\Big)
\end{aligned}
\end{equation*}
for $n,m, k \ge 0$, for some constant $C_0=C_0(u^0, v_p^0,
[u_e^1,v_{e}^1](0,\cdot))$ depending only on the profile $[u^0, v_p^0]$ and
given data $[u_e^1,v_{e}^1](0,\cdot)$.
\end{lemma}

\begin{proof} Define the stream function $\psi = \int_0^y u_p \; dy$. Then, $u_p = \psi_y$ and $v_p = - \psi_x$. We introduce the quantity $w = u_0 \psi_y - u^0_y \psi$. Using \eqref{prandtl1u}, we get 
\begin{equation}\label{def-wx} w_x =   - u^0_{xy} \psi  - [v_{p}^{0}+v_e^1]u_{py}   + u_{pyy}  +F_p,\end{equation}
for $F_p$ defined as in \eqref{prandtl1u}. Hence, the boundary values of $w$ and $w_x$ can be computed directly from the given boundary data $\bar u_1(y)$, $[u_p^0, v_p^0]$, and $F_p(0,y)$. Precisely, for $k\ge 0$, we get 
$$ \| \partial_y^k w(0,\cdot)\|_{L^2} \le C(u^0, v_p^0) \| \bar u_1\|_{H^k(\RR_+)}, \qquad  \| \partial_y^k w_x(0,\cdot)\|_{L^2} \le C(u^0, v_p^0) \| \bar u_1\|_{H^{k+2}(\RR_+)} + \| \partial_y^kF_p(0, \cdot)\|_{L^2(\RR_+)}$$
for some constant $C(u^0, v_p^0)$ that depends on the high regularity norms of $u^0, v_p^0$; see the estimates on $[u_p^0, v_p^0]$ in Corollary \ref{cor-Prandtl0}. By definition of $F_p$ in \eqref{prandtl1u} and the fact that $u_e^1, v_{e}^1$, and hence $\partial_z^k [u_e^1,v_{e}^1, u_{ex}^1]$, are all given on the boundary $x=0$, we get 
$$ \| \partial_y^kF_p(0, \cdot)\|_{L^2(\RR_+)} \le C(u^0, v_p^0,[u_e^1,v_e^1](0,\cdot)) ,$$
in which the $y$-decay factor comes directly from the decay property of $u_p^0$ and $C(u^0, v_p^0,[u_e^1,v_e^1](0,\cdot))$ depends on high regularity norms of $u^0, v_p^0,$ and $ [u_e^1,v_{e}^1](0,\cdot)$. 

Next, from the definition of $w$, we can write  
$$\psi = u^0 \int_0^y \frac{w(x,\theta)}{\{u^0\}^2} \;d\theta$$
and so we have 
$$\|\langle y \rangle^n \partial_y^{k+1} v_p(0,\cdot)\|_{L^2(\RR_+)} = \|\langle y\rangle^n\partial_y^{k+1} \psi_x(0,\cdot)\|_{L^2(\RR_+)} \le C ( \|w(0,\cdot)\|_{H^{k}(\RR_+)} + \| w_x(0,\cdot)\|_{H^{k}(\RR_+)}).$$
This proves the claimed estimate for $v_p(0,\cdot)$ in $\dot H^{k+1}(\RR_+)$. Next, as for $v_x$ estimate, we differentiate \eqref{def-wx} with respect to $x$ and get 
$$ \| w_{xx}(0,\cdot)\|_{H^k(\RR_+)}\le C(u^0, v_p^0) \Big( \|\langle y \rangle^{-2} \psi_x(0,\cdot)\|_{H^k(\RR_+)}  + \| v_p(0,\cdot)\|_{H^{k+3}(\RR_+)} \Big)  +\|F_{px}(0,\cdot)\|_{H^k(\RR_+)}.$$
Again by definition of $F_p$ in \eqref{prandtl1u}, we have  
$$ \| F_{px}(0,y)\|_{H^k(\RR_+)} \le C(u^0, v_p^0, [u_e^1,v_{e}^1](0,\cdot)) \Big( 1 + \| \langle y\rangle^{-m}u_{exx}^1(0, \cdot)\|_{H^k(\RR_+)}\Big) $$
in which $C(u^0, v_p^0, v_{e}^1(0,\cdot))$ depends on high regularity norms of $u^0, v_p^0,$ and $ v_{e}^1(0,\cdot)$. Similarly, $v_{pyy}(0,y)$ can be written in term of $u_{p}(0,y), v_p(0,y), F_{py}(0,y)$, according to \eqref{prandtl1u-dy}. This gives 
$$ \| v_p(0,\cdot)\|_{H^{k+3}(\RR_+)} \le C(u^0, v_p^0,[u_e^1, v_{e}^1](0,\cdot)) \Big( 1 + \| v_p(0,\cdot)\|_{\dot H^{k+1}(\RR_+)} \Big) $$ 
This yields estimates on $v_{px}$ on the boundary as claimed. 
\end{proof}


\begin{lemma}
\label{lem-step1} There exists a positive number $L>0$ so that for each $N$
large, the fourth order elliptic equation: 
\begin{equation}  \label{eqsv-step1}
-v_{yyx}+\frac{u^{0}_{yy}}{u^{0}}v_x +\left\{ \frac{ v_{yy}}{u^{0}}%
\right\}_{yy} =f_{y}+g
\end{equation}
on $[0,L]\times [0,N]$ has a unique solution satisfying initial condition $%
v(0,y) = \bar v_0(y)$ and boundary conditions: $[v,v_y] = 0$ at $y=0,N$,
with sources $f, g$ in weighted $L^2$ spaces. Furthermore, there holds 
\begin{equation}  \label{step1}
\begin{aligned} \sup_{0\leq x\leq L}& \| \langle y \rangle^n \partial
_{x}^{j}v_{yy} \|_{L^2([0,N])} + \| \langle y\rangle ^{n}\partial
_{x}^{j}v_{xy} \|_{L^2}\\ &\le C\sum_{k=0}^1\Big( \|\langle y
\rangle^n\partial_x^k v_{yy}\|_{L^2(\{ x=0\})} + \|\langle y \rangle ^n
\partial_x^k f \|_{L^2}+\|\langle y\rangle^{n+\frac 32}\partial_x^kg\|_{L^2}
\Big) \end{aligned}
\end{equation}
for $j = 0,1$, for any $n\ge 0$, as long as the right-hand side is finite.
\end{lemma}

\begin{proof} Let us choose an orthogonal 
basis $\{e^{i}(y)\}_{i=1}^{\infty }$ in $H^{2}([0,N])$ with $[e^{i},e_{y}^{i}]=0$ at $y = 0,N$, for all $i\ge 1$. The orthogonality is obtained with respect
to the $[[ \cdot ]]$ inner product: 
\begin{equation*}
\lbrack \lbrack e^{i},e^{j}]]=\delta _{ij}.
\end{equation*}%
Such an orthogonal basis exists, since $[[\cdot ]]$ is equivalent to the usual inner product in $H^1([0,N])$.  Then, we introduce the weak formulation of \eqref{eqsv-step1}: 
\begin{equation*}
\lbrack \lbrack v_{x},e^{i}]]+\int \frac{v_{yy}e_{yy}^{i}}{u^{0}}\; dy =\int ( - fe^i_y + g e^i)\; dy 
\end{equation*}%
for all $e^{i}(y)$, $i \ge 1$. Next, for each fixed $k$, we construct an approximate solution in Span$\{e^{i}(y)\}_{i=1}^{k}$ defined as
\begin{equation*}
v^{k}(x,y): =\sum_{i=1}^{k}a^{i}(x)e_{i}(y),
\end{equation*}%
for each $x\in [0,L]$. Then, $v^k(x.\cdot)$ solves  
\begin{equation}
\lbrack \lbrack v_{x}^{k},e^{i}]]+\int \frac{v_{yy}^{k}e_{yy}^{i}}{u^{0}} = \int ( - fe^i_y + g e^i)\; dy   \label{kapproximate}
\end{equation}%
which by orthogonality yields a system of ODE equations:
\begin{equation*}
a_{x}^{i}+\sum_{i=1}^{k}a^{j}\int \frac{e_{yy}^{j}e_{yy}^{i}}{u^{0}}=\int ( - fe^i_y + g e^i)\; dy .
\end{equation*}%
Since $f,g \in L^2(\RR_+)$, the ODE system has the unique smooth solution $a^k$ and hence, $v^k$ is defined uniquely and smooth.  Multiplying \eqref{kapproximate} by $a_x^i$ and taking the sum over $i$, we get  
\begin{equation*}
\lbrack \lbrack v_{x}^{k},v_{x}^{k}]]+\int \frac{v_{yy}^{k}v_{yyx}^{k}}{u^{0}%
} = \int ( - f v_{xy}^k + gv_x^k )\; dy
\end{equation*}%
which is equivalent to
\begin{equation*}
\lbrack \lbrack v_{x}^{k},v_{x}^{k}]]+\frac{1}{2}\frac{d}{dx}\int \frac{%
\{v_{yy}^{k}\}^{2}}{u^{0}}= \int ( - f v_{xy}^k + gv_x^k  - \frac{1}{2} \left\{ \frac{1}{u^{0}}\right\} _{x}\{v_{yy}^{k}\}^{2} \Big)\; dy.
\end{equation*}%
By the positivity estimate \eqref{positivity-intro} and \eqref{low-vy} with $u^0$, we note that $[[v_{x}^{k},v_{x}^{k}]]\geq \theta_0 \|v_{xy}^{k}\|_{L^2(0,N)}^{2}$. The
standard Gronwall inequality then yields 
\begin{equation}\label{Prd1-vk}
\sup_{x\in [0,L]} \int|v^k_{yy}|^2 +\|v_{xy}^{k} \|_{L^2}^{2}  \le C \Big( \|v_{yy}^k\|_{L^2(\{ x=0\})}^2 +   \|f \|_{L^2}^{2}+\|\langle y\rangle^{3/2}g\|_{L^2}^{2} \Big)
\end{equation}%
in which we have used the inequality $\| gv_x^k\|_{L^1} \le C \| v_{xy}^k\|_{L^2}\| \langle y \rangle^{3/2} g\|_{L^2}.$ Taking limit as $k\rightarrow \infty ,$ we obtain the solution to \eqref{eqsv-step1} at once. This also proves the claim \eqref{step1} when $j = n=0$. 

Next, we shall derive high regularity estimates. We take $x$-derivative of (\ref{kapproximate}) to get 
\begin{eqnarray*}
\lbrack \lbrack v_{xx}^{k},e^{i}]]+\int \frac{v_{xyy}^{k}e_{yy}^{i}}{u^{0}}
&=&-\int
\Big( \left\{ \frac{u^0_{yy}}{u^0}\right\}_{x} v_x^k e^i + \left\{ \frac{1}{u^{0}}\right\} _{x}v_{yy}^{k}e_{yy}^{i}\Big) +
\int ( - f_xe^i_y + g_x e^i)\; dy.
\end{eqnarray*}%
Recall now that $v_{xx}^{i}=\sum_{i=1}^{k}a_{xx}^{i}e^{i}$, which is a smooth function, since $a^i, e^i$ are both smooth. Using $v_{xx}^i$ as a test function, we then get \begin{eqnarray*}
&&[[v_{xx}^{k},v_{xx}^{k}]]+\frac{1}{2}\frac{d}{dx}\int \frac{
\{v_{yyx}^{k}\}^2}{u^{0}} \\
&=&-\frac{1}{2}\int \left\{ \frac{1}{u^{0}}\right\}
_{x}\{v_{xyy}^{k}\}^{2}-\int
\Big( \left\{ \frac{u^0_{yy}}{u^0}\right\}_{x} v_x^k v^k_{xx} + \left\{ \frac{1}{u^{0}}\right\} _{x}v_{yy}^{k}v^k_{xxyy}\Big) +
\int ( - f_xv^k_{xxy} + g_x v^k_{xx})\; dy.
\end{eqnarray*}%
The Gronwall inequality, together with \eqref{Prd1-vk}, yields the claim \eqref{step1} for the unweighted estimates, upon integrating by parts in $y$ the third term on the right.  Almost identically, we may now insert the weight function $w(y) = \langle y\rangle ^{2n}$ and take inner products against $w(y)v_{x}^{k}$ and $w(y)v_{xx}^{k}$, respectively in the above energy estimates to obtain
the weighted estimates as claimed. We avoid repeating the details. 
\end{proof}

\begin{proof}[Proof of Lemma \ref{lem-1stP}] To apply the previous step, we first take care of the non-zero boundary conditions \eqref{prandtl1bc}. Indeed, let us take $\chi(\cdot) $ to be a cutoff function near $0$
with $\chi (0)=1$, and introduce 
$$ \bar v = v_p(x,y) - y \chi(y) u^1_{ex}(x,0).$$
Hence, $\bar v=\bar v_{y}=0$ at both $y=0$ and $y=N.$ In addition, from \eqref{prandtl1u-dxy}, $\bar v$ solves 
\begin{equation}\label{eqs-barv}
\begin{aligned}
- \bar v_{xyy} +\frac{u^{0}_{yy}}{u^0} \bar v_{x}   + \Big(\frac{\bar v_{yy}}{u^0} \Big)_{yy}  
&= G_{px}  - \Big( \frac{u^{0}_{yy}}{u^0}\Big)_x v_p + \Big(u_{py}  \Big(\frac{1}{u^0} \Big)_x\Big)_{yy}  
 - (y \chi)_{yy} u^1_{exx}(x,0) \\&\quad+\frac{u^{0}_{yy}}{u^0} y \chi u^1_{exx}(x,0)  + \Big(\frac{(y\chi )_{yy}}{u^0} \Big)_{yy} u^1_{ex}(x,0)  
\\
&=: f_y + g,
\end{aligned}\end{equation}
in which $G_p$ is defined as in \eqref{prandtl1u-dy}. Explicitly, we have defined 
$$ \begin{aligned}
f&:= u_{pyy}  \Big(\frac{1}{u^0} \Big)_x + u_{py}  \Big(\frac{1}{u^0} \Big)_{xy} +  \Big( \frac{ v_{p}^{0}+v_e^1}{u^0}\Big) v_{pyy} + 2 v_{pyy} \Big( \frac{1}{u^0}\Big)_{y} + \frac{v_{ex}^1}{u^0} u_{py} 
\\
g &: = \Big(\frac{1}{u^0}\Big)_x \Big( F_{py} -  u_{p}u_{xy}^{0} - [v_{p}^{0}+v_e^1]u_{pyy} - \sqrt \veps v_{ez}^1 u_{py}\Big) -\Big( \frac{v_{ex}^1}{u^0}\Big)_y u_{py} \\&\quad +
\frac{1}{u^0}\Big( F_{pxy} -  u_{px}u_{xy}^{0} - u_{p}u_{xxy}^{0} - v_{px}^{0} u_{pyy} - \sqrt \veps v_{exz}^1 u_{py} - \sqrt \veps v_{ez}^1 u_{pxy} \Big)
\\&\quad - \Big( \frac{ v_{p}^{0}+v_e^1}{u^0}\Big)_y v_{pyy} - 2 v_{pyy} \Big( \frac{1}{u^0}\Big)_{yy} - 2 u_{pyy} \Big( \frac{1}{u^0}\Big)_{xy} -  u_{pxy}\Big( \frac{1}{u^0}\Big)_{yy}-  u_{py}\Big( \frac{1}{u^0}\Big)_{xyy}
\\&\quad 
  - \Big( \frac{u^{0}_{yy}}{u^0}\Big)_x v_p  - (y \chi)_{yy} u^1_{exx}(x,0) +\frac{u^{0}_{yy}}{u^0} y \chi u^1_{exx}(x,0)  + \Big(\frac{(y\chi )_{yy}}{u^0} \Big)_{yy} u^1_{ex}(x,0)  
\end{aligned}$$
with $$F_p = - u_{ez}^{0}[yu_{px}^{0}+v_{p}^{0}]- yv_{ez}^{1}u_{py}^{0} - u_{e}^{1}u_{px}^{0} -u_{p}^{0}u_{ex}^{1} .
$$
Here, we note that the divergence-free condition is imposed: $u_{px} = -v_{py} = -\bar v_y + (y\chi)_y u_{ex}^1(x,0)$. We construct the unique solution $\bar v$ to the above problem, and hence the solution $v_p$ to (\ref{prandtl1u-dxy}) via a contraction mapping theorem. We shall work with the norm: 
$$|||\bar v|||\equiv \sup_{0\leq x\leq L}\int  |\bar v_{yy}|^{2} + \iint | \bar v_{xy}|^2 \; dxdy .$$ 
Lemma \ref{lem-step1} (or precisely, the estimate \eqref{Prd1-vk}) yields 
\begin{equation}\label{step1-re}
\begin{aligned}
||| \bar v|||  \le C \Big( \|\bar v_{yy}\|_{L^2(\{ x=0\})}^2 +   \|f \|_{L^2}^{2}+\|\langle y\rangle^{3/2}g\|_{L^2}^{2} \Big)
\end{aligned}
\end{equation}
with $(f,g)$ defined as in \eqref{eqs-barv}. Recall that $ \bar v = v_p(x,y) + y \chi(y) v^1_{ez}(x,0)$ with $v_{e}^1$ given on the boundary $x=0$. Hence, $\bar v_{yy}(0,y)$ can be estimates as follows, thanks to Lemma \ref{lem-initialv},
$$  \|v_{pyy}(0,\cdot)\|_{L^2(\RR_+)}   \le C(u^0, v_p^0, [u_e^1,v_{e}^1](0,\cdot)) \Big( 1 + \| \bar u_1\|_{H^{4}(\RR_+)}\Big).
$$
The uniform bound on $\|\bar v_{yy}\|_{L^2(\{ x=0\})}^2$ follows.

Next, let us give bounds on $f,g$. For instance, $|\bar v|  \le y^{3/2} \| v_{yy}\|_{L^2(\RR_+)}$, $|\bar v_y|  \le y^{1/2} \| v_{yy}\|_{L^2(\RR_+)}$, and thus 
$$ \begin{aligned}
\iint \langle y \rangle ^{-n} |\bar v|^2  & \le C \sup_{x} \| v_{yy}\|_{L^2(\RR_+)}^2 \iint \langle y\rangle^{-n+3} \le C L ||| \bar v|||,
\\
 \iint \langle y \rangle ^{-n} |\bar v_y|^2   &\le C \sup_{x} \| v_{yy}\|_{L^2(\RR_+)}^2 \iint \langle y\rangle^{-n+1} \le C L ||| \bar v|||,
 \end{aligned}$$ 
for some large $n$. Such a spatial decay $\langle y \rangle^{-n}$ is produced by the rapid decay property of $u_p^0$. Similarly, we have 
$$ \iint \langle y \rangle^{-n} |u_p|^2 \le C \iint \langle y \rangle^{-n} \Big[ |\bar u_1|^2 + L \int_0^L |u_{px}|^2 \Big] \le CL \Big( \| \bar u_1 \|^2 _{L^2(\RR_+)} + \| u^1_{ex}(x,0)\|_{L^2(0,L)}^2 + ||| \bar v|||\Big) , $$
using the fact that $u_{px} = -v_{py} = -\bar v_y + (y\chi)_y u_{ex}^1(x,0)$. As for $u_{py}$, we use $|u_{py}| \le  |\bar u_{1y}| + \sqrt L\| u_{pxy}\|_{L^2(0,L)}$. Next, we bound $\langle y \rangle^{-n} u_{pyy}$, which by observation the decaying factor $\langle y \rangle^{-n}$ is always present, for large $n$, due to the decay property of $[u_p^0, v_p^0]$. To do so, we use \eqref{prandtl1u} to estimate 
$$ \iint \langle y \rangle^{-n} |u_{pyy}|^2 \le \iint \langle y\rangle^{-n} \Big |u^{0}u_{px} + u_{p}u_{x}^{0} +v_{p}u^{0}_y +[ v_{p}^{0}+ v_e^1]u_{py}   - F_p \Big|^2$$
which is again bounded by $C + C L ||| v|||$. Finally, we note that  
$$\begin{aligned}
 \| u_{exx}^1(x,0)\|_{L^2}^2  
 & \le  \iint |v_{exz}^1(x,z)v_{exzz}^1(x,z)|
\\&\le 
\| v_{exz}^1 \|_{L^{q'}} \| v_{exzz}^1\|_{L^q} \le C(L,q) \veps^{-1+1/q},
\end{aligned}$$
for arbitrary pair $(q,q')$ so that $1/q+1/q' = 1$; here, we take $q\to 1$. 


Taking $L$ sufficiently small in the above estimates and in \eqref{step1-re} yields a uniform bound on $||| \bar v|||$: 
$$ |||\bar v||| \le  C(L,\kappa) \veps^{-\kappa}, $$
for arbitrarily small $\kappa>0$. Since the equation is linear in $\bar v$, this assures the existence of the unique solution to \eqref{eqs-barv} and hence to \eqref{prandtl1u-dxy}. In addition, the above construction can be repeated to
obtain a global solution in $x$ for any given $L$ for the existence of the Euler data. Finally, taking $N\rightarrow \infty $, we obtain the solution to \eqref{prandtl1u} over $[0,\infty]\times \RR_+$. 
The claimed weighted estimates follow similarly, using the rapid decay property of $u_p^0$. 

In addition, the boundedness of $v_p$ follows by the calculation:
\begin{equation} \label{vbound}
\begin{aligned}
| v(x,y)|^{2}&\leq \int_{0}^{y} | v  v_y|dy \le \int_0^\infty \langle y\rangle^{-n} | v| \; dy + \int_0^\infty \langle y \rangle^{n} | v_y|^2 \; dy
\\ &\leq C \sup_{x} \|  v_{yy}\|^2_{L^2(\RR_+)} +  \int_{0}^{\infty}\langle y\rangle ^{n}| v_{y}(0,y)|^{2}dy+ C\iint \langle
y\rangle ^{n}| v_{xy}|^{2}\; dxdy, 
\end{aligned}\end{equation}%
which is bounded thanks to the previous bound on $||| v|||$ and the uniform estimate of $v_p$ on the boundary $x=0$. Similarly, boundedness of $u_p$ follows from the definition 
$$ u_p(x,y) = \bar u_1(y) + \int_0^x v_{py} \; dx $$
in which 
$$ \int_0^L |v_{py}|^2 \; dx \le \iint \langle y \rangle ^{-n} |v_{py}|^2 + \iint \langle y \rangle^{n} |v_{pyy}|^2 $$ 
which is again bounded by $||| v|||$.

To complete the proof of the lemma, we are now concerned with the higher regularity estimate. Again, applying Lemma \ref{lem-step1} to the equation \eqref{eqs-barv} yields 
$$
\begin{aligned}
\sup_{0\leq x\leq L}&  \| \langle y \rangle^n v_{xyy} \|_{L^2(\RR_+)}  + \| \langle y\rangle ^{n}v_{xxy} \|_{L^2}\\
&\le C\sum_{k=0}^1\Big( \|\langle y \rangle^n\partial_x^k v_{yy}\|_{L^2(\{ x=0\})} +   \|\langle y \rangle ^n \partial_x^k f \|_{L^2}+\|\langle y\rangle^{n+\frac 32}\partial_x^kg\|_{L^2} \Big).
\end{aligned}
$$
Let us give bounds on the boundary term on $x=0$. Recall that $ v = v_p(x,y) - y \chi(y) u^1_{ex}(x,0)$. Lemma \ref{lem-initialv} gives
$$  \|\langle y \rangle^n v_{pxyy}(0,\cdot)\|_{L^2(\RR_+)}   \le C_0 \Big( 1 +\| \langle y\rangle^{-m}u_{exx}^1(0, \sqrt \veps \cdot)\|_{H^{1}(\RR_+)}\Big).
$$
Using the inequality  $|Lf(0)| \le \int_0^L| [(L-x)f(x)]_x|\; dx \le \sqrt L \| f\|_{L^2} + L^{3/2}\|f_x\|_{L^2} $, we have  
$$ \begin{aligned}
\| u_{exx}^1(0, \sqrt \veps \cdot)\|_{L^2(\RR_+)} 
&\le C \veps^{-1/4}  \| v_{exz}^1(0,\cdot)\|_{L^2(\RR_+)}
\\&\le C \veps^{-1/4} L^{-1/2} \Big( \| v_{exz}^1(\cdot)\|_{L^2} + \| v_{exz}^1(\cdot)\|_{L^2}^{1/2} \| v_{exxz}^1(\cdot)\|_{L^2}^{1/2}\Big)
\\&\le C \veps^{-1/2} L^{-1/2} 
\end{aligned}$$
and 
$$ \begin{aligned}
\| u_{exxy}^1(0, \sqrt \veps \cdot)\|_{L^2(\RR_+)} 
&\le C \veps^{1/4}  \| v_{exzz}^1(0,\cdot)\|_{L^2(\RR_+)}
\\&\le C \veps^{1/4} L^{-1/2} \Big( \| v_{exzz}^1(\cdot)\|_{L^2} + \| v_{exzz}^1(\cdot)\|_{L^2}^{1/2} \|  v_{exxzz}^1(\cdot)\|_{L^2}^{1/2}\Big)
\\&\le C \veps^{-3/4} L^{-1/2} 
\end{aligned}$$
in which the estimates on $v_e^1$ from Lemma \ref{lem-ve} were used. Also, we have 
$$ \begin{aligned}
| u_{exx}^1(0,0)|^2  &\le  \| u_{exx}^1(0,\cdot)\|_{L^2(\RR_+)}  \|v_{exzz}^1(0,\cdot)\|_{L^2(\RR_+)}
\\&\le C  L^{-1/2}  \| u_{exx}^1(0,\cdot)\|_{L^2(\RR_+)} \Big( \| v_{exzz}^1(\cdot)\|_{L^2} + \|  v_{exzz}^1(\cdot)\|_{L^2}^{1/2} \| v_{exzzz}^1(\cdot)\|_{L^2}^{1/2}\Big)
\\&\le C \veps^{-3/2 } L^{-1} 
\end{aligned}$$
This proves that 
$$ \|\langle y \rangle^n v_{pxyy}(0,\cdot)\|_{L^2(\RR_+)}   \le C \veps^{- 3/4} L^{-1/2} ,$$
uniformly in small $\veps$ and $L$. Next, estimates for $f$ and $g$ are treated similarly as done above. In particular, we note 
$$\begin{aligned}
 \| u_{exxx}^1(x,0)\|_{L^2}^2  
 & \le  \iint |v_{exxz}^1(x,z)v_{exxzz}^1(x,z)|
\\&\le 
\| v_{exxz}^1 \|_{L^{2}} \| v_{exxzz}^1\|_{L^2} \le C(L) \veps^{-2},
\end{aligned}$$
which together with the previous estimates yield the estimate \eqref{prandtl1bound-high}. 

This completes the proof of the lemma. 
\end{proof}

\subsubsection{Cut-off Prandtl layers}

\label{sec-defP1} Finally, we are ready to introduce the Prandtl layers that
we shall use in the boundary layer expansion. Let us define a cutoff
function $\chi (\cdot )$ with support in $[0,1]$, and let $[u_{p},v_{p}]$ be
constructed as in the previous section. We introduce 
\begin{equation}
u_{p}^{1}=\chi (\sqrt{\varepsilon }y)u_{p}+\sqrt{\varepsilon }\chi ^{\prime
}(\sqrt{\varepsilon }y)\int_{0}^{y}u_{p}(x,s)ds,\text{ \ \ \ }v_{p}^{1}=\chi
(\sqrt{\varepsilon }y)v_{p}.  \label{u1p}
\end{equation}%
Clearly, $[u_{p}^{1},v_{p}^{1}]$ is a divergence-free vector field. By the
estimates from Lemma \ref{lem-1stP} on $[u_{p},v_{p}]$, we get 
\begin{equation*}
\begin{aligned} \Big| \sqrt{\varepsilon } \chi ^{\prime }(\sqrt{\varepsilon
}y) \int_0^y u_p (x,s) \; ds\Big| & \le \sqrt \veps y |\chi'(\sqrt \veps
y)|\| u_p\|_\infty \le C(L,\kappa) \varepsilon^{-\kappa}.\end{aligned}
\end{equation*}%
Hence, Lemma \ref{lem-1stP} now reads we have 
\begin{equation}
\begin{aligned} \| [u_p^1, v_p^1]\|_\infty + \sup_{0\leq x\leq L} \| \langle
y \rangle^n v^1_{pyy} \|_{L^2(\RR_+)} + \| \langle y\rangle ^{n}v^1_{pxy}
\|_{L^2} &\le C(L,\kappa)\varepsilon^{-\kappa} \\ \sup_{0\leq x\leq L} \|
\langle y \rangle^n v^1_{pxyy} \|_{L^2(\RR_+)} + \| \langle y\rangle
^{n}v^1_{pxxy} \|_{L^2} & \le C(L)\varepsilon^{-1} \end{aligned}
\label{key-Prandtl1}
\end{equation}%
uniformly in small $\varepsilon ,L$, and for arbitrarily small $\kappa $. In
addition, thanks to the cut-off function, we also have 
\begin{equation}
\begin{aligned} \|v_{px}^{1} \|_{L^2}^{2} &\le \iint_{\{\sqrt \veps y \le
1\}} |v_{px}^{1}|^{2}dxdy\lesssim \varepsilon^{-1/2} \iint \langle y\rangle
^{n}|v_{pxy}^1|^{2}dxdy\le C(L,\kappa)\varepsilon ^{-1/2-2\kappa} \\
\|v_{py}^{1} \|_{L^2}^{2} &\le \iint_{\{\sqrt \veps y \le 1\}}
|v_{py}^{1}|^{2}dxdy\lesssim \varepsilon^{-1/2} \iint \langle y\rangle
^{n}|v_{pyy}^1|^{2}dxdy\le C(L,\kappa)\varepsilon ^{-1/2-2\kappa} \\
\|v_{pxx}^{1}\|_{L^2}^{2} &\le \iint_{\{ \sqrt \veps y\le 1\}}
|v_{pxx}^{1}|^{2}dxdy\lesssim \varepsilon ^{-1/2}\iint \langle y\rangle
^{n}|v_{pxxy}^{1}|^{2}dxdy\le C(L) \varepsilon ^{-5/2}. \end{aligned}
\label{est-vpx}
\end{equation}

We now plug $[u_{p}^{1},v_{p}^{1}]$ into \eqref{prandtl1u}, or equivalently, %
\eqref{Prandtl1-u}. It does not solve it completely, yielding a new error
due to the above cut-off: 
\begin{equation*}
\begin{aligned} R^{u,1}_p&: = \Big( u^{0}\partial_x + u_x^0+ [ v_{p}^{0}+
v_e^1] \partial_y - \partial_y^2 \Big) (\sqrt \veps \chi' (\sqrt \veps y)
\int_0^y u_p\; ds) - 2\sqrt \veps \chi'(\sqrt \veps y) [v_{p}^{0}+ v_e^1]
u_{py} \\&\quad + u_p [v_{p}^{0}+ v_e^1] (\sqrt \veps \chi'(\sqrt \veps y) -
\varepsilon \chi''(\sqrt \veps y)) \\&\quad + (1-\chi (\sqrt \veps y)) \Big(
u_{ez}^{0}[yu_{px}^{0}+v_{p}^{0}] + yv_{ez}^{1}u_{py}^{0} +
u_{e}^{1}u_{px}^{0} + u_{p}^{0}u_{ex}^{1} \Big) \end{aligned}
\end{equation*}%
which contributes into $R^{u,1}$ in \eqref{def-Ru1-mod} (and hence $R_{%
\mathrm{app}}^{u}$). Let us give an estimate on this error term. Recall that 
$u_{p}^{0}$ is rapidly decaying at infinity, and so $u_{x}^{0}=u_{px}^{0}$
also decays rapidly. Hence, the integral $u_{x}^{0}\int_{0}^{y}u_{p}\;ds$ is
uniformly bounded by $\varepsilon ^{-\kappa }$. Together with boundedness of
the constructed Euler and Prandtl layers, we have $L^{2}$ norm of the first
three big terms involving $u_{p}$ in $R_{p}^{u,1}$ is bounded by 
\begin{equation*}
C\varepsilon ^{-\kappa }\sqrt{\varepsilon }\Vert \chi (\sqrt{\varepsilon }%
\cdot )\Vert _{L^{2}}\leq C\varepsilon ^{1/4-\kappa }.
\end{equation*}%
Now, as for the term 
\begin{equation*}
(1-\chi (\sqrt{\varepsilon }y))\Big(%
u_{ez}^{0}[yu_{px}^{0}+v_{p}^{0}]+yv_{ez}^{1}u_{py}^{0}+u_{e}^{1}u_{px}^{0}+u_{p}^{0}u_{ex}^{1}%
\Big)
\end{equation*}%
we note that the Prandtl layers $[u_{p}^{0},v_{p}^{0}]$ is rapidly decaying
in $y\rightarrow \infty $, this error term is thus bounded by $Ce^{-y}$,
which is of order $\varepsilon ^{n}$ in the region where $\sqrt{\varepsilon }%
y\geq 1$, for arbitrary large $n\geq 0$. Combining with the above estimates
proves that 
\begin{equation}
\Vert R_{p}^{u,1}\Vert _{L^{2}}\lesssim C(L,\kappa )\varepsilon ^{1/4-\kappa
}.
\label{def-Ru1p}
\end{equation}

Finally, by a view of definition of $p_p^2$ from \eqref{def-pressure2}, we
estimate $p_{px}^2$: 
\begin{equation*}
\begin{aligned} p_{px}^2 &= \int_y^\infty \partial_x\Big[
[u_{e}^{0}+u_{p}^{0}] v_{px}^{0} + u_p^0
v_{ex}^1+[v_{p}^{0}+v_{e}^{1}]v_{py}^{0} - v_{pyy}^{0} \Big] (x,\theta)\;
d\theta \\ &= \int_y^\infty \Big[ [u_{e}^{0}+u_{p}^{0}] v_{pxx}^{0} + u_p^0
v_{exx}^1+[v_{p}^{0}+v_{e}^{1}]v_{pxy}^{0} - v_{pxyy}^{0} \Big] (x,\theta)\;
d\theta , \end{aligned}
\end{equation*}
the last identity was obtained with a use of the divergence-free condition
on the vector field $[u^0_p, v^0_p]$. We estimate each term on the right.
Thanks to the boundedness of $u_e^0, [u_p^0,v_p^0]$ and $v_e^1$, and the
rapid decay property of the Prandtl layers $[u_p^0, v_p^0]$, we note that 
\begin{equation*}
\begin{aligned} \int_y^\infty [u_{e}^{0}+u_{p}^{0}] v_{pxx}^{0} &\le C
\langle y \rangle^{-n} \| u_e^0 + u_p^0\|_{L^\infty} \| \langle y\rangle^n
v_{pxx}^0\|_{L^2(\RR_+)} \\ \int_y^\infty u_p^0 v_{exx}^1 & \le C \langle y
\rangle^{-n} \| \langle y \rangle^n u_p^0\|_{L^2(\RR_+)} \|
v_{exx}^1(x,\sqrt \veps \cdot)\|_{L^2(\RR_+)} \\ \int_y^\infty
[v_p^{0}+v_e^1] v_{pxy}^{0} &\le C \langle y \rangle^{-n} \| v_p^0 +
v_e^1\|_{L^\infty} \| \langle y\rangle^n v_{pxy}^0\|_{L^2(\RR_+)} \\
\int_y^\infty v_{pxyy}^{0} &\le C \langle y \rangle^{-n} \| \langle
y\rangle^n v_{pxyy}^0\|_{L^2(\RR_+)}. \end{aligned}
\end{equation*}
Hence, taking $n\ge 2$ and using the known bounds on the profile solutions,
we immediately get 
\begin{equation}
\| p_{px}^2\|_{L^2}\lesssim \varepsilon^{-1/4}.  \label{estimate-px2}
\end{equation}

\subsection{Proof of Proposition \protect\ref{prop-approximate}}

\label{sec-app} Having constructed the Euler and Prandtl layers, we now
calculate the remaining errors in $R_{\mathrm{app}}^{u}$ and $R_{\mathrm{app}%
}^{v}$ from \eqref{tangential} and \eqref{normal}, respectively, and hence
complete the proof of Proposition \ref{prop-approximate}. To do so,
collecting errors from $R^{u,0}$ in \eqref{Ru0}, $R^{u,1}$ in %
\eqref{est-Ru1-mod}, the new error $R_{p}^{u,1}$ in \eqref{def-Ru1p}, and
the remaining $\varepsilon $-order terms in $R_{\mathrm{app}}^{u}$, we get 
\begin{equation*}
\begin{aligned} R^u_\mathrm{app} &= E_0 - \varepsilon u^0_{ezz}+ \sqrt \veps
R^{u,1}+ \sqrt \veps R_p^{u,1} + \varepsilon \Big[
[u_{e}^{1}+u_{p}^{1}]\partial_x +v_{p}^{1}\partial _{y} \Big ]
[u_{e}^{1}+u_{p}^{1}] \\&\qquad +\varepsilon p_{px}^2 - \varepsilon
\partial_x^2 \{u_{p}^{0}+\sqrt{\varepsilon }[u_{e}^{1}+u_{p}^{1}] \}
\end{aligned}
\end{equation*}%
in which reading the estimates \eqref{est-Ru1-mod}, \eqref{est-E0}, %
\eqref{def-Ru1p}, \eqref{estimate-px2} and using the fact that $u_{ezz}^{0}$
are bounded in $L^{2}$ in the original coordinates, we immediately have 
\begin{equation*}
\Vert E_{0}-\varepsilon u_{ezz}^{0}+\sqrt{\varepsilon }R^{u,1}+\sqrt{%
\varepsilon }R_{p}^{u,1}+\varepsilon p_{px}^{2}\Vert _{L^{2}}\leq C(L,\kappa
)\varepsilon ^{3/4-\kappa }.
\end{equation*}%
Similarly, using boundedness of $u_{e}^{1},u_{p}^{1},v_{p}^{1}$ and the $%
L^{2}$ bound on the derivatives of $u_{p}^{0},u_{e}^{1},u_{p}^{1}$ and
keeping in mind that the Euler flows are evaluated at $(x,\sqrt{\varepsilon }%
y)$, we get 
\begin{equation*}
\begin{aligned} \varepsilon \| [u_{e}^{1}+u_{p}^{1}]\partial_x
[u_{e}^{1}+u_{p}^{1}] \|_{L^2} &\le \varepsilon [\|
u_{e}^{1}\|_{L^\infty}+\|u_{p}^{1}\|_{L^\infty}][\|
u_{ex}^{1}\|_{L^2}+\|u_{px}^{1} \|_{L^2}] \le C(L,\kappa)
\varepsilon^{3/4-\kappa} \\ \varepsilon \| v_{p}^{1}\partial _{y}
[u_{e}^{1}+u_{p}^{1}] \|_{L^2} &\le \varepsilon \| v_{p}^{1}\|_{L^\infty}
[\| \sqrt \veps u_{ez}^{1}\|_{L^2}+\|u_{py}^{1} \|_{L^2}] \le C(L,\kappa)
\varepsilon^{1-\kappa} \\ \varepsilon \| \partial_x^2
\{u_{p}^{0}+\sqrt{\varepsilon }[u_{e}^{1}+u_{p}^{1}] \}\|_{L^2} & \le
\varepsilon \| u_{pxx}^{0}\|_{L^2}+\varepsilon^{3/2}
[\|u_{exx}^{1}\|_{L^2}+\|u_{pxx}^{1}\|_{L^2}] \le C \varepsilon,
\end{aligned}
\end{equation*}%
in which $\kappa $ is arbitrarily small constant (we choose $\frac{\kappa }{2%
}$ in Lemma \ref{lem-1stP}). This proves that 
\begin{equation*}
\Vert R_{\mathrm{app}}^{u}\Vert _{L^{2}}\leq C(L,\kappa )\varepsilon
^{3/4-\kappa }.
\end{equation*}

Next, we calculate the error $R_{\mathrm{app}}^{v}$ from \eqref{normal}.
Simply collecting the remaining terms in $R^{v,0}$ (see \eqref{est-Rv0}) and
all terms with a factor $\sqrt{\varepsilon }$ or small, we get 
\begin{equation*}
\begin{aligned} R^v_\mathrm{app}& = R^{v,0} + \sqrt \veps \Big[
\{u_{e}^{0}+u_{p}^{0}+\sqrt{\varepsilon }[u_{e}^{1}+u_{p}^{1}]\}\partial
_{x} +\{v_{p}^{0}+v_{e}^{1}+\sqrt{\varepsilon }v_{p}^{1}\}\partial _{y}
\Big] v_{p}^{1} \\ &\quad + \sqrt \veps \Big[ [u_{e}^{1}+u_{p}^{1}]\partial
_{x} +v_{p}^{1}\partial _{y} \Big] [v_{p}^{0}+v_{e}^{1} ] - \sqrt \veps
v_{pyy}^1- \varepsilon \partial_x^2 \{v_{p}^{0}+v_{e}^{1}+\sqrt{\varepsilon
}v_{p}^{1}\}. \end{aligned}
\end{equation*}%
By \eqref{bound-Rv0}, we have $\Vert R^{v,0}\Vert _{L^{2}}\leq C\varepsilon
^{1/4}$. We now estimate the remaining terms one by one in $R_{\mathrm{app}%
}^{v}$. Similarly as above, using the boundedness of all profile solutions,
we get 
\begin{equation*}
\begin{aligned} \sqrt \veps \Big\| \Big[
\{u_{e}^{0}+u_{p}^{0}+\sqrt{\varepsilon }[u_{e}^{1}+u_{p}^{1}]\}\partial
_{x} +\{v_{p}^{0}+v_{e}^{1}+\sqrt{\varepsilon }v_{p}^{1}\}\partial _{y}
\Big] v_{p}^{1}\Big\|_{L^2} &\le C \sqrt \veps \| [u_e^0, u_p^0, v_p^0,
u_e^1, v_e^1]\|_{L^{\infty}} \|\nabla v_p^1\|_{L^2} \\& \le C(L,\kappa)
\varepsilon^{1/4-\kappa} \end{aligned}
\end{equation*}%
upon recalling the bound $\Vert v_{px}^{1}\Vert _{L^{2}}\leq C(L)\varepsilon
^{-1/4-\kappa }$. Next, we have 
\begin{equation*}
\begin{aligned} \sqrt \veps \Big\| \Big[ [u_{e}^{1}+u_{p}^{1}]\partial _{x}
+v_{p}^{1}\partial _{y} \Big] [v_{p}^{0}+v_{e}^{1} ] \Big\|_{L^2} &\le C
\sqrt \veps \| [u_e^1, u_p^1, v_p^1]\|_{L^\infty} \Big( \| v_{px}^0+
v_{ex}^1\|_{L^2} + \| v_{py}^0 + \sqrt \veps v_{ez}^1\|_{L^2}\Big) \\ &\le
C(L,\kappa) \varepsilon^{1/4-\kappa}, \end{aligned}
\end{equation*}%
with noting that $\Vert v_{ex}^{1}(\cdot ,\sqrt{\varepsilon }\cdot )\Vert
_{L^{2}}\leq C\varepsilon ^{-1/4}$. Finally, it is clear that 
\begin{equation*}
\Vert \sqrt{\varepsilon }v_{pyy}^{1}-\varepsilon \partial
_{x}^{2}\{v_{p}^{0}+v_{e}^{1}+\sqrt{\varepsilon }v_{p}^{1}\}\Vert
_{L^{2}}\leq C(L)\varepsilon ^{1/4-\kappa },
\end{equation*}%
since $\Vert \lbrack v_{pyy}^{1},v_{pxx}^{0}]\Vert _{L^{2}}$ is uniformly
bounded by $C(L,\kappa )\varepsilon ^{-\kappa }$, $\Vert v_{exx}^{1}\Vert
_{L^{2}}\leq C\varepsilon ^{-1/4}$, and $\Vert v_{pxx}^{1}\Vert _{L^{2}}\leq
C(L)\varepsilon ^{-5/4}$, as summarized in Section \ref{sec-defP1}. Putting
these together into $R_{\mathrm{app}}^{v}$ and using the fact that $%
\varepsilon \ll L$, we have obtained 
\begin{equation*}
\Vert R_{\mathrm{app}}^{v}\Vert _{L^{2}}\leq C(L,\kappa )\varepsilon
^{1/4-\kappa }.
\end{equation*}

This completes the proof of Proposition \ref{prop-approximate}.

\section{Linear stability estimates}

This section is devoted to prove the following crucial linear stability
estimates for the linearized equations around the constructed approximate
solutions $[u_{\mathrm{app}},v_{\mathrm{app}}]$. Recall (\ref{us}).

\begin{proposition}
\label{prop-stability} Let $[u_s,v_s]$ be the approximate solution defined
as in \eqref{us}. For any given $f,g$ in $L^2$, there exists a positive
number $L$ so that the following linear problem 
\begin{eqnarray}
u_{s}u_{x}+uu_{sx}+v_{s}u_{y}+vu_{sy}+p_{x}-\Delta _{\varepsilon }u &=&f
\label{linearu} \\
u_{s}v_{x}+uv_{sx}+v_{s}v_{y}+vv_{sy}+\frac{p_{y}}{\varepsilon }-\Delta
_{\varepsilon }v &=&g ,  \label{linearv}
\end{eqnarray}%
together with the divergence-free condition $u_{x}+v_{y}=0$ and boundary
conditions 
\begin{equation}  \label{lin-BCs}
\begin{aligned} \lbrack u,v]_{y=0} &=0\text{ (no-slip)},\text{ \ \ \ \ \ \ \
\ \ \ \ }[u,v]_{x=0}=0\text{ \ (Dirichlet)}, \\ p-2\varepsilon u_{x}
&=0,\text{ \ \ \ \ }u_{y}+\varepsilon v_{x}=0\quad \text{ \ \ at }x=L\text{
(Neumann or stress-free)}. \end{aligned}
\end{equation}
has a unique solution $[u,v,p]$ on $[0,L]\times \mathbb{R}_+$. Furthermore,
there holds 
\begin{equation}  \label{stability}
\| \nabla _{\varepsilon }u \|_{L^2}+\|\nabla _{\varepsilon }v\|_{L^2}\quad
\lesssim\quad \|f\|_{L^2}+\sqrt{\varepsilon }\|g\|_{L^2}.
\end{equation}
\end{proposition}


The proof of Proposition \ref{prop-stability} consists of several steps.
First, we construct the solution in the artificial cut-off domain: 
\begin{equation*}
\Omega _{N}: = \{0\leq x\leq L,\quad 0\leq y\leq N\},
\end{equation*}
with the no-slip boundary conditions $[u,v]=0$ prescribed at $y=N$. We shall
apply the standard Schaefer's fixed theorem (see, for instance, \cite[%
Section 9.2.2]{Evans}) for the space $X=\{\|[u,v]\|_{H^{1}(\Omega _{N})}\leq
C\},$ for some fixed constant $C>0.$ Without loss of generality, we may
assume the source terms $f$ and $g$ are smooth. Then the standard regularity
theory for the Stokes problem yields $H^3$ regularity for $[u,v]$, except at
the four corners $[0,0],[0,N]$, $[L,0]$ and $[L,N]$. In addition, by \cite%
{O,OS, BR}, we know that $[u,v]\in H^{3/2+}$ and $p\in H^{1/2+\text{ }}$%
including the four corners, and hence, the $H^1$ norm of $[u,v]$ makes
sense. Indeed by the trace theorem, we have 
\begin{equation*}
\|\nabla \lbrack u,v]\|_{L^{2+}(\Gamma )}+\|p\|_{L^{2+}(\Gamma )}\lesssim
\|[u,v]\|_{H^{3/2+}}+\|p\|_{H^{1/2+}},
\end{equation*}
for any finite piece-wise $C^{1}$ curve $\Gamma.$ We now take $\Gamma
=\Gamma _{\delta }$ to be the curve of intersection of $\Omega _{N}$ and the
circle of radius $\delta $ and centered at the four corners, respectively.
Clearly, 
\begin{equation*}
\int_{\Gamma _{\delta }}|\nabla u|^{2}+|\nabla v|^{2}\lesssim o(1)\Big[ %
|\nabla u|_{L^{2+}(\Gamma _{\delta })}^{2}+|\nabla v|_{L^{2}(\Gamma _{\delta
})}^{2}\Big] \lesssim o(1) \Big[ \|[u,v]\|_{H^{3/2+}}+\|p\|_{H^{1/2+}}\Big]
\end{equation*}%
where $o(1)\rightarrow 0$ as $\delta \rightarrow 0.$ This justifies the
meaning of $H^1$ norm of the solution $[u,v]$ in the presence of corners.

We shall now derive uniform a priori estimates for (\ref{linearu})-%
\eqref{lin-BCs}. Taking the limit $N\rightarrow \infty$ yields the uniform a
priori bound \eqref{stability}. The existence of the solution and hence the
Proposition \ref{prop-stability} would then follow from a direct application
of the Schaefer's fixed point theorem; see \cite[Theorem 4, p. 504]{Evans}
and Section \ref{sec-proofStability}, below. As will be seen shortly, the
positivity estimate \eqref{positivity-intro} plays a crucial role.


\subsection{Energy estimates}

\label{sec-EE}

\begin{lemma}
\label{lem-EE} Let $[u,v]$ be the solution to the problem \eqref{linearu}-%
\eqref{lin-BCs}. Assume that $\varepsilon \ll L$. There holds 
\begin{equation*}
\Vert \nabla _{\varepsilon}u\Vert _{2}^{2}+\int_{x=L}u_{s}(u^{2}+\varepsilon
v^{2})\lesssim L\Vert \nabla _{\varepsilon }v\Vert _{L^{2}}^{2}+\Vert f\Vert
_{L^{2}}^{2}+\varepsilon \Vert g\Vert _{L^{2}}^{2}.
\end{equation*}
\end{lemma}
\begin{proof} We multiply (\ref{linearu}) with $u$ and (\ref{linearv}) with $\varepsilon v$
(or equivalently, take $[u,v]$ as the test function in the weak formulation)
get 
\begin{eqnarray*}
&&\iint [u_{s}u_{x}+u_{sx}u+v_{s}u_{y}+vu_{sy}+p_{x}-\Delta _{\varepsilon
}u]u \\
&&+\iint \varepsilon \lbrack
u_{s}v_{x}+u_{sx}v+v_{s}v_{y}+vv_{sy}+p_{y}-\varepsilon \Delta _{\varepsilon
}v]v \\
&=&\iint uf+\iint \varepsilon vg.
\end{eqnarray*}%
By writing \ $\Delta _{\varepsilon }u=2\varepsilon u_{xx}+(u_{y}+\varepsilon
v_{x})_{y},$ $\Delta _{\varepsilon }v=(u_{y}+\varepsilon v_{x})_{x}+2v_{yy}$
and performing the integration by parts multiple times, the left-hand side
of the above is reduced to 
\begin{eqnarray*}
&&\iint \{u_{s}u_{x}+u_{sx}u+v_{s}u_{y}+vu_{sy}+p_{x}-\Delta _{\varepsilon
}u\}u \\
&&+\iint \varepsilon \lbrack
u_{s}v_{x}+u_{sx}v+v_{s}v_{y}+vv_{sy}+p_{y}-\varepsilon \Delta _{\varepsilon
}v]v \\
&=&-\iint u_{sx}\frac{u^{2}+\varepsilon v^{2}}{2}+u_{s}\frac{%
u^{2}+\varepsilon v^{2}}{2}\Big |_{x=0}^{x=L}+\iint u_{sx}[u^{2}+\varepsilon
v^{2}]-\iint v_{sy}\frac{u^{2}+\varepsilon v^{2}}{2} \\
&&+\iint [vu_{sy}u+\varepsilon v_{sy}v^{2}]+\int pu\Big |_{x=0}^{x=L} \\
&&-2\varepsilon \int u_{x}u\Big |_{x=0}^{x=L}+2\varepsilon \iint
u_{x}^{2}+\iint (u_{y}+\varepsilon v_{x})u_{y} \\
&&-\varepsilon \int (u_{y}+\varepsilon v_{x})v\Big |_{x=0}^{x=L}+\iint
\varepsilon (u_{y}+\varepsilon v_{x})v_{x}+2\iint \varepsilon v_{y}^{2}.
\end{eqnarray*}%
By using the boundary conditions $p=2\varepsilon u_{x}$ at $x=L$ and $[u,v]=$
$0$ at $x=0,$ and the divergence-free condition $u_{sx}+v_{sy}=0$, the
energy estimate now becomes 
\begin{eqnarray*}
&&\int_{x=L}u_{s}\frac{u^{2}+\varepsilon v^{2}}{2}+\iint [2\varepsilon
u_{x}^{2}+(u_{y}+\varepsilon v_{x})^{2}+2\varepsilon v_{y}^{2}] \\
&=&-\iint u_{sx}[u^{2}+\varepsilon v^{2}]-\iint [vu_{sy}u+\varepsilon
v_{sy}v^{2}]+\iint uf+\varepsilon \iint vg.
\end{eqnarray*}%
Here, we note that 
\begin{equation*}
\begin{aligned} \iint [2\varepsilon u_{x}^{2}+(u_{y}+\varepsilon
v_{x})^{2}+2\varepsilon v_{y}^{2}] &= \iint [2\varepsilon
u_{x}^{2}+u^2_{y}+\varepsilon^2 v_{x}+2\varepsilon v_{y}^{2}] + 2\varepsilon
\iint u_y v_x \\&\ge \frac 12 \| \nabla _\veps u\|_{L^2}^2 - 2\varepsilon \|
\nabla_\veps v \|_{L^2}^2,\end{aligned}
\end{equation*}%
in which we have used the Young inequality, giving the estimate $%
2\varepsilon \iint u_{y}v_{x}\leq \frac{1}{2}\Vert u_{y}\Vert
_{L^{2}}^{2}+2\varepsilon \Vert \nabla _{\varepsilon}v\Vert _{L^{2}}^{2}$.
In addition, since $[u,v]=0$ at $x=0$, we have the embedding inequalities $%
\Vert u\Vert _{L^{2}}\leq L\Vert u_{x}\Vert _{L^{2}}=L\Vert v_{y}\Vert
_{L^{2}},\Vert v\Vert _{L^{2}}\leq L\Vert v_{x}\Vert _{L^{2}}$. Hence, 
\begin{eqnarray*}
\iint uf+\varepsilon \iint vg &\leq &\Vert u\Vert _{L^{2}}\Vert f\Vert
_{L^{2}}+\varepsilon \Vert v\Vert _{L^{2}}\Vert g\Vert _{L^{2}} \\
&\leq &L^{2}\{\Vert v_{y}\Vert _{L^{2}}^{2}+\varepsilon \Vert v_{x}\Vert
^{2}\}+\Vert f\Vert _{L^{2}}^{2}+\varepsilon \Vert g\Vert _{L^{2}}^{2}.
\end{eqnarray*}%
Similarly, since $v=0$ at $y=0$, we can estimate $v=\int_{0}^{y}v_{y}\leq 
\sqrt{y}\left\{ \int_{0}^{y}v_{y}^{2}\right\} ^{1/2}.$ Hence, 
\begin{eqnarray*}
&&-\iint u_{sx}[u^{2}+\varepsilon v^{2}]-\iint [v\partial
_{y}u_{s}u+\varepsilon \partial _{y}v_{s}v^{2}] \\
&\leq &L^{2}\iint |u_{sx}|u_{x}^{2}+\varepsilon \iint |yu_{sx}|\left\{
\int_{0}^{y}v_{y}^{2}\right\} +\iint \left\{
\int_{0}^{x}|u_{x}^{2}dx|\right\} ^{1/2}\sqrt{y}\partial _{y}u_{s}\left\{
\int_{0}^{y}|v_{y}^{2}dx|\right\} ^{1/2} \\
&&+\varepsilon \iint y\partial _{y}v_{s}\left\{
\int_{0}^{y}v_{y}^{2}dy\right\} \\
&\leq &\left\{ L^{2}\sup |\partial _{x}u_{s}|+\varepsilon \sup_{x}\int
|y\partial _{y}v_{s}|dy+L\sup_{x}\sqrt{\int y\{\partial _{y}u_{s}\}^{2}dy}%
\right\} \Vert v_{y}\Vert _{L^{2}}^{2}.
\end{eqnarray*}

This proves the claimed inequality in the lemma, with 
\begin{equation*}
\begin{aligned} C(\varepsilon, L,u_s,v_s): = \left\{ L^{2}\sup
|u_{sx}|+\varepsilon \sup_{x}\int |yv_{sy}|dy+L\sup_{x}\sqrt{\int
y\{u_{sy}\}^{2}dy}+L^{2} + 2\varepsilon \right\} . \end{aligned}
\end{equation*}%
It remains to give the bound on the constant $C(\varepsilon ,L,u_{s},v_{s})$%
. We recall that $u_{s}=u_{e}^{0}+u_{p}^{0}+\sqrt{\varepsilon }u_{e}^{1}$
and $v_{s}=v_{p}^{0}+v_{e}^{1}$. Keeping in mind that the zeroth-order
Prandtl layers $[u_{p}^{0},v_{p}^{0}]$ are smooth with arbitrarily high
Sobolev regularity and rapidly decaying in $y$. Hence, 
\begin{equation}
\Vert u_{sx}\Vert _{\infty }\leq \Vert u_{px}^{0}\Vert _{\infty }+\sqrt{%
\varepsilon }\Vert u_{ex}^{1}\Vert _{\infty }\leq C(u_{p}^{0})+\sqrt{%
\varepsilon }\Vert v_{e}^{1}(\cdot )\Vert _{H^{3}}\leq C,  \label{bound-usx}
\end{equation}%
thanks to bounds on the Euler flows, summarized in Section \ref{sec-defE1}.
Similarly, we have 
\begin{eqnarray}
\sup_{x}\int |yv_{sy}|dy &\leq &C(v_{p}^{0})+\varepsilon ^{1/2}\sup_{x}\int
|yv_{ez}^{1}(x,\sqrt{\varepsilon }y)|dy  \notag \\
&\leq &C(v_{p}^{0})+\varepsilon ^{-1/2}\sup_{x}\int |zv_{ez}^{1}|dz
\label{vintegral} \\
&\leq &C(v_{p}^{0})+C\varepsilon ^{-1/2}L^{-1}\Big(\iint
|zv_{ez}^{1}|dxdz+\iint |zv_{exz}^{1}|dxdz\Big)  \notag \\
&\lesssim &1+C\varepsilon ^{-1/2}L^{-1}\Vert \langle z\rangle
^{n}v_{e}^{1}\Vert _{H^{2}}^{2},
\end{eqnarray}%
which is bounded by $C(L)\varepsilon ^{-1/2}$, thanks to the estimates from
Lemma \ref{lem-ve}. Also, we have 
\begin{eqnarray}
\sup_{x}\int y\{u_{sy}\}^{2}dy &\leq &\sup_{x}\int y\Big[\varepsilon
|u_{ez}^{0}|^{2}+|u_{py}^{0}|^{2}+\varepsilon ^{2}|u_{ez}^{1}|^{2}\Big]\;dy 
\notag \\
&\leq &C(u_{p}^{0})+\sup_{x}\int \Big [z|u_{ez}^{0}|^{2}+\varepsilon
z|u_{ez}^{1}|^{2}\Big ]\;dz  \label{uintegral} \\
&\leq &C(u_{e}^{0},u_{p}^{0})+\varepsilon L^{-1}\Big(\iint
z|u_{ez}^{1}|^{2}dxdz+\iint z|u_{exz}^{1}|^{2}dxdz\Big)  \notag \\
&\lesssim &1+\varepsilon L^{-1}\Vert \langle z\rangle ^{n}v_{e}^{1}\Vert
_{H^{2}}^{2},  \notag
\end{eqnarray}%
which is again bounded by $C$, for $\varepsilon \ll L$. Putting this
together into the above definition of $C(\varepsilon ,L,u_{s},v_{s})$ and
the fact that $\varepsilon \ll L$ yield the lemma at once.
\end{proof}

\subsection{Positivity estimates}

\label{sec-positivity} In this section, we establish the following crucial
positivity estimate:

\begin{lemma}
\label{lem-positivity} Let $[u,v]$ be the solution to the problem %
\eqref{linearu}-\eqref{lin-BCs}. Assume that $\varepsilon \ll L$. There
holds 
\begin{equation}
\Vert \nabla _{\varepsilon }v\Vert _{L^{2}}^{2}+\varepsilon
^{2}\int_{x=0}v_{x}^{2}+\varepsilon \int_{x=L}v_{y}^{2}\lesssim \Vert \nabla
_{\varepsilon }u\Vert _{L^{2}}^{2}+L\Vert \nabla _{\varepsilon }v\Vert
_{L^{2}}^{2}+\Vert f\Vert _{L^{2}}^{2}+\varepsilon \Vert g\Vert _{L^{2}}^{2}.
\label{vorticity}
\end{equation}
\end{lemma}
\begin{proof} We start from the identity: $\partial
_{y}\left\{ \frac{v}{u_{s}}\right\} \times $(\ref{linearu})$-\varepsilon
\partial _{x}\left\{ \frac{v}{u_{s}}\right\} \times $(\ref{linearv}).
Formally, this is the vorticity equation multiplied by the test function $%
\frac{v}{u_{s}}$. This yields 
\begin{eqnarray}
&&\iint \partial _{y}\left\{ \frac{v}{u_{s}}\right\}
\{u_{s}u_{x}+u_{sx}u+v_{s}u_{y}+v\partial _{y}u_{s}+p_{x}-\Delta
_{\varepsilon }u\}  \notag \\
&&-\iint \partial _{x}\left\{ \frac{v}{u_{s}}\right\} \{\varepsilon \lbrack
u_{s}v_{x}+u\partial _{x}v_{s}+v_{s}v_{y}+v\partial
_{y}v_{s}]+p_{y}-\varepsilon \Delta _{\varepsilon }v\}  \label{vorticity-est}
\\
&=&\iint \partial _{y}\left\{ \frac{v}{u_{s}}\right\} f-\varepsilon \partial
_{x}\left\{ \frac{v}{u_{s}}\right\} g.  \notag
\end{eqnarray}%
Again, we use the inequality $|v|\leq \sqrt{y}\left\{
\int_{0}^{y}v_{y}^{2}\right\} ^{1/2}$, together with the estimates %
\eqref{vintegral}-\eqref{uintegral} on $[u_{s},v_{s}]$ to estimate the
right-hand side of \eqref{vorticity-est}. We have 
\begin{eqnarray*}
&&\iint \left\{ \frac{v_{y}}{u_{s}}-\frac{v\partial _{y}u_{s}}{u_{s}^{2}}%
\right\} f-\varepsilon \left\{ \frac{v_{x}}{u_{s}}-\frac{v\partial _{x}u_{s}%
}{u_{s}^{2}}\right\} g \\
&\leq &\frac{\sup_{x}|\langle y\rangle ^{1/2}\nabla _{\varepsilon
}u_{s}|_{2}+\sup \Vert u_{s}\Vert _{\infty }}{\{\min u_{s}\}^{2}}\{\Vert
f\Vert _{L^{2}}+\sqrt{\varepsilon }\Vert g\Vert _{L^{2}}\}\Vert \nabla
_{\varepsilon }v\Vert _{L^{2}} \\
&\lesssim &\{\Vert f\Vert _{L^{2}}+\sqrt{\varepsilon }\Vert g\Vert
_{L^{2}}\}\Vert \nabla _{\varepsilon }v\Vert _{L^{2}},
\end{eqnarray*}%
in which the Young inequality can be applied to absorb the $L^{2}$ norm of $%
\nabla _{\varepsilon}v$ to the left hand side of \eqref{vorticity}.

Next, we treat each term on the left-hand side of \eqref{vorticity-est}.
First, integrating by parts multiple times, we have 
\begin{eqnarray}
&&\iint \partial _{y}\left\{ \frac{v}{u_{s}}\right\} \{u_{s}u_{x}+v\partial
_{y}u_{s}\}-\iint \partial _{x}\left\{ \frac{v}{u_{s}}\right\} \varepsilon
u_{s}v_{x}  \notag \\
&=&\iint \partial _{y}\left\{ \frac{v}{u_{s}}\right\}
\{-u_{s}v_{y}+v\partial _{y}u_{s}\}-\iint \left\{ \frac{v_{x}}{u_{s}}-\frac{%
\partial _{x}u_{s}v}{u_{s}^{2}}\right\} \varepsilon u_{s}v_{x}  \notag \\
&=&\iint \left\{ \frac{v_{y}}{u_{s}}-\frac{\partial _{y}u_{s}v}{u_{s}^{2}}%
\right\} \{-u_{s}v_{y}+v\partial _{y}u_{s}\}-\varepsilon \iint
v_{x}^{2}+\iint \frac{\varepsilon \partial _{x}u_{s}vv_{x}}{u_{s}^{2}} 
\notag \\
&=&-\iint v_{y}^{2}+2\iint vv_{y}\frac{\partial _{y}u_{s}}{u_{s}}-\iint 
\frac{\{\partial _{y}u_{s}\}^{2}v^{2}}{u_{s}^{2}}-\varepsilon \iint
v_{x}^{2}+\iint \frac{\varepsilon \partial _{y}u_{s}vv_{x}}{u_{s}^{2}} 
\notag \\
&=&-\iint v_{y}^{2}-\iint \partial _{y}\left\{ \frac{\partial _{y}u_{s}}{%
u_{s}}\right\} v^{2}-\iint v^{2}\frac{\{\partial _{y}u_{s}\}^{2}}{u_{s}^{2}}%
-\varepsilon \iint v_{x}^{2}+\iint \frac{\varepsilon \partial _{y}u_{s}vv_{x}%
}{u_{s}^{2}}  \notag \\
&=&-\iint v_{y}^{2}-\iint \frac{\partial _{yy}u_{s}}{u_{s}}v^{2}-\varepsilon
\iint v_{x}^{2}+\iint \frac{\varepsilon \partial _{y}u_{s}vv_{x}}{u_{s}^{2}}
\notag \\
&=&-\iint u_{s}^{2}|\partial _{y}\left\{ \frac{v}{u_{s}}\right\}
|^{2}-\varepsilon \iint v_{x}^{2}+\iint \frac{\varepsilon \partial
_{y}u_{s}vv_{x}}{u_{s}^{2}},  \notag
\end{eqnarray}%
in which the last equality is precisely due to the positivity estimate (\ref%
{positivity-intro}). From (\ref{low-vy}), we obtain a lower bound 
\begin{equation}
\iint \partial _{y}\left\{ \frac{v}{u_{s}}\right\} \{u_{s}u_{x}+v\partial
_{y}u_{s}\}-\iint \partial _{x}\left\{ \frac{v}{u_{s}}\right\} \varepsilon
u_{s}v_{x}\lesssim -\iint |\nabla _{\varepsilon }v|^{2}+\iint \frac{%
\varepsilon \partial _{y}u_{s}vv_{x}}{u_{s}^{2}},  \label{vort-est1}
\end{equation}%
which crucially yields a bound on the $L^{2}$ norm of $\nabla _{\varepsilon
}v$; or precisely, the $L^{2}$ norm of $\nabla _{\varepsilon }v$ appearing
on the left-hand side of \eqref{vorticity}.

Next, we treat the pressure term. Integrating by parts, with recalling that $%
p=2\varepsilon u_{x}$ at $x=L$, we have 
\begin{equation*}
\begin{aligned} \iint \partial _{y}\left\{ \frac{v}{u_{s}}\right\}
p_{x}-\iint \partial _{x}\left\{ \frac{v}{u_{s}}\right\} p_{y}
&=\int_{x=L}\partial _{y}\left\{ \frac{v}{u_{s}}\right\} p \\&=2\varepsilon
\int_{x=L}\partial _{y}\left\{ \frac{v}{u_{s}}\right\} u_{x} =-2\varepsilon
\int_{x=L}\partial _{y}\left\{ \frac{v}{u_{s}}\right\} v_{y} \\
&=-2\varepsilon \int_{x=L}\frac{v_{y}^{2}}{u_{s}}-2\varepsilon
\int_{x=L}\partial _{y}\left\{ \frac{1}{u_{s}}\right\} vv_{y} \end{aligned}
\end{equation*}%
in which we can estimate 
\begin{eqnarray*}
\varepsilon \int_{x=L}\partial _{y}\left\{ \frac{1}{u_{s}}\right\} vv_{y}
&\leq &\varepsilon \int_{x=L}\frac{\partial _{y}u_{s}}{u_{s}^{2}}vv_{y}\leq
\varepsilon \sup \frac{\partial _{y}u_{s}}{u_{s}^{3/2}}\int_{x=L}v\frac{v_{y}%
}{\sqrt{u_{s}}} \\
&\leq &\{C(u_{ey}^{0},u_{py}^{0})+\varepsilon^{1/2} \|u_{ez}^{1}\|_{\infty
}\}L\varepsilon \|v_{x}\|_{L^2}\left\{ \int_{x=L}\frac{v_{y}^{2}}{u_{s}}%
\right\} ^{1/2}.
\end{eqnarray*}%
Here, thanks to the bounds on the Euler flows, summarized in Section \ref%
{sec-defE1}, we in particular have $\varepsilon^{1/2} \|u_{ez}^{1}\|_{\infty
} \le C \varepsilon^{1/2} \| u_e^1\|_{H^3}\lesssim 1$. Together with the
Young inequality, we thus obtain 
\begin{equation}  \label{pressure}
\begin{aligned} \iint \Big[ \partial _{y}\left\{ \frac{v}{u_{s}}\right\}
p_{x}- \partial _{x}\left\{ \frac{v}{u_{s}}\right\} p_{y} \Big] &\le -\frac
\varepsilon2 \int_{x=L}\frac{v_{y}^{2}}{u_{s}} + C(u_{s}, v_s) L^2 \|
\nabla_\veps v\|_{L^2}^2, \end{aligned}
\end{equation}%
in which we stress that the boundary term is favorable.

%

Next, we shall treat terms involving the Laplacian. Again, we recall from 
\cite{O,OS} that the Stokes problem yields $[u,v]\in H^{3/2+}$ and $p\in
H^{1/2+},$ and so their traces $[\nabla u,\nabla v]\in L^{2+}(\Gamma )$ and $%
p\in L^{2+}(\Gamma )$ on any smooth curve $\Gamma .$ Moreover, $[u,v]\in
H^{2}$ and $p\in H^{1}$ away from the four corners of $[0,L]\times \lbrack
0,N]$. Hence, we can evaluate 
\begin{eqnarray}
I:= &&\iint \Big[-\partial _{y}\left\{ \frac{v}{u_{s}}\right\} \Delta
_{\varepsilon }u+\partial _{x}\left\{ \frac{v}{u_{s}}\right\} \varepsilon
\Delta _{\varepsilon }v\Big]  \notag \\
&=&\iint \Big[-\partial _{y}\left\{ \frac{v}{u_{s}}\right\}
[u_{yy}+\varepsilon u_{xx}]+\partial _{x}\left\{ \frac{v}{u_{s}}\right\}
\varepsilon \lbrack v_{yy}+\varepsilon v_{xx}]\Big]  \notag \\
&=&\iint \Big[-\left\{ \frac{v_{y}}{u_{s}}-\frac{v\partial _{y}u_{s}}{%
u_{s}^{2}}\right\} [u_{yy}+\varepsilon u_{xx}]+\left\{ \frac{v_{x}}{u_{s}}-%
\frac{v\partial _{x}u_{s}}{u_{s}^{2}}\right\} \varepsilon \lbrack
v_{yy}+\varepsilon v_{xx}]\Big]  \notag \\
&=&\iint \Big[\frac{u_{yy}u_{x}}{u_{s}}+\varepsilon \frac{u_{xx}u_{x}}{u_{s}}%
+\varepsilon \frac{v_{x}v_{yy}}{u_{s}}+\varepsilon ^{2}\frac{v_{x}v_{xx}}{%
u_{s}}\Big]  \notag \\
&&+\iint \Big[\left\{ \frac{v\partial _{y}u_{s}}{u_{s}^{2}}\right\}
[u_{yy}+\varepsilon u_{xx}]-\frac{v\partial _{x}u_{s}}{u_{s}^{2}}%
[\varepsilon v_{yy}+\varepsilon ^{2}v_{xx}]\Big].  \notag
\end{eqnarray}%
Now taking integration by parts respectively in each integration above, with
a special attention on the boundary contributions, we get 
\begin{eqnarray}
I &=&-\frac{1}{2}\int_{x=L}\frac{u_{y}^{2}}{u_{s}}+\frac{1}{2}\iint \partial
_{x}\left\{ \frac{1}{u_{s}}\right\} u_{y}^{2}-\iint \partial _{y}\left\{ 
\frac{1}{u_{s}}\right\} u_{y}u_{x}  \label{Lap-vortterm} \\
&&+\frac{\varepsilon }{2}\int_{x=L}\frac{u_{x}^{2}}{u_{s}}-\frac{\varepsilon 
}{2}\iint \partial _{x}\left\{ \frac{1}{u_{s}}\right\} u_{x}^{2}  \notag \\
&&-\frac{\varepsilon }{2}\int_{x=L}\frac{v_{y}^{2}}{u_{s}}+\frac{\varepsilon 
}{2}\iint \partial _{x}\left\{ \frac{1}{u_{s}}\right\} v_{y}^{2}-\varepsilon
\iint \partial _{y}\left\{ \frac{1}{u_{s}}\right\} v_{y}v_{x}  \notag \\
&&+\frac{\varepsilon ^{2}}{2}\int_{x=L}\frac{v_{x}^{2}}{u_{s}}-\frac{%
\varepsilon ^{2}}{2}\int_{x=0}\frac{v_{x}^{2}}{u_{s}}-\frac{\varepsilon ^{2}%
}{2}\iint \partial _{x}\left\{ \frac{1}{u_{s}}\right\} v_{x}^{2}  \notag \\
&&-\iint \partial _{y}\left\{ \frac{v\partial _{y}u_{s}}{u_{s}^{2}}\right\}
u_{y}-\varepsilon \iint \partial _{x}\left\{ \frac{v\partial _{y}u_{s}}{%
u_{s}^{2}}\right\} u_{x}+\varepsilon \int_{x=L}\left\{ \frac{v\partial
_{y}u_{s}}{u_{s}^{2}}\right\} u_{x}  \notag \\
&&+\varepsilon \iint \partial _{y}\left\{ \frac{v\partial _{x}u_{s}}{%
u_{s}^{2}}\right\} v_{y}+\varepsilon ^{2}\iint \partial _{x}\left\{ \frac{%
v\partial _{x}u_{s}}{u_{s}^{2}}\right\} v_{x}-\varepsilon
^{2}\int_{x=L}\left\{ \frac{v\partial _{y}u_{s}}{u_{s}^{2}}\right\} v_{x}. 
\notag
\end{eqnarray}%
Let us first take care of boundary contributions. Notice that there is only
one boundary term at $x=0$, which is a favorable term: $-\frac{\varepsilon
^{2}}{2}\int_{x=0}\frac{v_{x}^{2}}{u_{s}}$. Now as for boundary terms at $x=L
$, one can use the fact that $u_{y}+\varepsilon v_{x}=0$ and $u_{x}+v_{y}=0$
at $x=L$, and hence, the boundary contributions at $x=L$ can be simplified
as 
\begin{eqnarray*}
B^{L}:= &&-\frac{1}{2}\int_{x=L}\frac{u_{y}^{2}}{u_{s}}+\frac{\varepsilon }{2%
}\int_{x=L}\frac{u_{x}^{2}}{u_{s}}-\frac{\varepsilon }{2}\int_{x=L}\frac{%
v_{y}^{2}}{u_{s}}+\frac{\varepsilon ^{2}}{2}\int_{x=L}\frac{v_{x}^{2}}{u_{s}}
\\
&&+\varepsilon \int_{x=L}\left\{ \frac{v\partial _{y}u_{s}}{u_{s}^{2}}%
\right\} u_{x}-\varepsilon ^{2}\int_{x=L}\left\{ \frac{v\partial _{y}u_{s}}{%
u_{s}^{2}}\right\} v_{x} \\
&=&-\varepsilon \int_{x=L}\left\{ \frac{v\partial _{y}u_{s}}{u_{s}^{2}}%
\right\} v_{y}-\varepsilon \int_{x=L}\left\{ \frac{v\partial _{y}u_{s}}{%
u_{s}^{2}}\right\} u_{y} \\
&=&-\varepsilon \int_{x=L}\left\{ \frac{v\partial _{y}u_{s}}{u_{s}^{2}}%
\right\} v_{y}+\varepsilon \int_{x=L}\partial _{y}\left\{ \frac{v\partial
_{y}u_{s}}{u_{s}^{2}}\right\} u,
\end{eqnarray*}%
which can be estimated by 
\begin{eqnarray*}
&&B^{L}\leq \varepsilon \sqrt{L}\sup_{x}\left\vert \frac{\partial _{y}u_{s}}{%
u_{s}}\right\vert \Vert v_{x}\Vert _{L^{2}}\sqrt{\int_{x=L}\frac{v_{y}^{2}}{%
u_{s}}}+\varepsilon \sqrt{L}\sup_{x}\left\vert \frac{\partial _{y}u_{s}}{%
u_{s}}\right\vert \Vert u_{x}\Vert _{L^{2}}\sqrt{\int_{x=L}\frac{v_{y}^{2}}{%
u_{s}}} \\
&&\qquad +\varepsilon L\left\{ \sup_{x}\sqrt{\int \frac{y|\partial
_{yy}u_{s}|}{u_{s}^{2}}dy}+\sup_{x}\sqrt{\int \frac{y\{\partial
_{y}u_{s}\}^{2}}{u_{s}^{3}}dy}\right\} \Vert u_{x}\Vert _{L^{2}}\sqrt{%
\int_{x=L}\frac{v_{y}^{2}}{u_{s}}}.
\end{eqnarray*}%
We estimate norms on $u_{s}$ in the same lines as done in \eqref{vintegral}-%
\eqref{uintegral}. Recall that $u_{s}=u_{e}^{0}+u_{p}^{0}+\sqrt{\varepsilon }%
u_{e}^{1}$ and $u_{s}$ is bounded below away from zero thanks to the
assumption \eqref{nozero}. Now, similarly as done for \eqref{uintegral}, we
have 
\begin{eqnarray}
\sup_{x}\int y|u_{syy}|dy &\leq &\sup_{x}\int y\Big[\varepsilon
|u_{ezz}^{0}|+|u_{pyy}^{0}|+\varepsilon ^{3/2}|u_{ezz}^{1}|\Big]\;dy  \notag
\\
&\leq &C(u_{p}^{0})+\sup_{x}\int \Big [z|u_{ezz}^{0}|+\varepsilon
^{1/2}z|u_{ezz}^{1}|\Big ]\;dz  \label{uintegral-dy2} \\
&\leq &C(u_{e}^{0},u_{p}^{0})+\varepsilon ^{1/2}L^{-1}\Big(\iint
z|u_{ezz}^{1}|dxdz+\iint z|u_{exzz}^{1}|dxdz\Big)  \notag \\
&\leq &C(u_{e}^{0},u_{p}^{0})+\varepsilon ^{1/2}L^{-1}\Vert \langle z\rangle
^{n}v_{e}^{1}\Vert _{W^{3,q}}  \notag \\
&\lesssim &1+C(L)\varepsilon ^{-1/2+1/q}, \notag
\end{eqnarray}%
for any $q>1$. Here, Lemma \ref{lem-ve} was used. Taking $q\rightarrow 1$ in
the above estimates so that $-1/2+1/q>0$, the above is bounded uniformly by a
constant $C(u_{s},v_{s})$, which is independent of small $\varepsilon ,L$.
The integral $\int y|u_{sy}|^{2}\;dy$ is already estimated in %
\eqref{uintegral}. Also, we get 
\begin{equation}
\Vert u_{sy}\Vert _{\infty }\leq \sqrt{\varepsilon }\Vert u_{ez}^{0}\Vert
_{\infty }+\Vert u_{py}^{0}\Vert _{\infty }+\varepsilon \Vert
u_{ez}^{1}\Vert _{\infty }\leq C(u_{e}^{0},u_{p}^{0})+\varepsilon \Vert
v_{e}^{1}(\cdot )\Vert _{H^{3}}\leq C.  \label{bound-usy}
\end{equation}%
This together with the Young inequality yields 
\begin{equation}
B^{L}\leq \frac{\varepsilon }{4}\int_{x=L}\frac{v_{y}^{2}}{u_{s}}%
+C(u_{s},v_{s})L\Vert \nabla _{\veps}v\Vert _{2}^{2},  \label{BL}
\end{equation}%
upon using the divergence-free condition $u_{x}=-v_{y}$. Here, $\nabla _{%
\veps}=(\sqrt{\varepsilon }\partial _{x},\partial _{y})$. The boundary term
in \eqref{BL} can be absorbed into the good boundary term in \eqref{pressure}%
.

We now combine the untreated terms on the left-hand side of %
\eqref{vorticity-est}, all the interior terms in \eqref{Lap-vortterm}, and
the last term in (\ref{vort-est1}), altogether. We shall use the standard
embedding inequalities $\Vert u\Vert _{L^{2}}\leq L\Vert u_{x}\Vert _{L^{2}}$%
, $|v|\leq \sqrt{y}\left\{ \int_{0}^{y}v_{y}^{2}\right\} ^{1/2}$ and $%
|u|\leq \sqrt{y}\left\{ \int_{0}^{y}u_{y}^{2}\right\} ^{1/2}$. Respectively,
we have 
\begin{eqnarray*}
R^{0}:= &&\iint \frac{1}{2}\partial _{x}\left\{ \frac{1}{u_{s}}\right\}
u_{y}^{2}-\partial _{y}\left\{ \frac{1}{u_{s}}\right\} u_{y}u_{x} \\
&&\iint -\frac{\varepsilon }{2}\partial _{x}\left\{ \frac{1}{u_{s}}\right\}
u_{x}^{2}+\frac{\varepsilon }{2}\partial _{x}\left\{ \frac{1}{u_{s}}\right\}
v_{y}^{2}-\varepsilon \partial _{y}\left\{ \frac{1}{u_{s}}\right\}
v_{y}v_{x}-\frac{\varepsilon ^{2}}{2}\partial _{x}\left\{ \frac{1}{u_{s}}%
\right\} v_{x}^{2} \\
&&-\iint \partial _{y}\left\{ \frac{v\partial _{y}u_{s}}{u_{s}^{2}}\right\}
u_{y}-\varepsilon \partial _{x}\left\{ \frac{v\partial _{y}u_{s}}{u_{s}^{2}}%
\right\} u_{x}+\varepsilon \partial _{y}\left\{ \frac{v\partial _{x}u_{s}}{%
u_{s}^{2}}\right\} v_{y}+\varepsilon ^{2}\partial _{x}\left\{ \frac{%
v\partial _{x}u_{s}}{u_{s}^{2}}\right\} v_{x} \\
&&+\iint \partial _{y}\left\{ \frac{v}{u_{s}}\right\} \{\partial
_{x}u_{s}u+v_{s}u_{y}\}-\varepsilon \partial _{x}\left\{ \frac{v}{u_{s}}%
\right\} \{u\partial _{x}v_{s}+v_{s}v_{y}+v\partial _{y}v_{s}\}+\frac{%
\varepsilon \partial _{y}u_{s}vv_{x}}{u_{s}^{2}}.
\end{eqnarray*}%
Let us give estimates on each term on the right. We claim that 
\begin{equation}
R^{0}\lesssim \Vert \nabla _{\varepsilon }u\Vert _{L^{2}}^{2}+\Vert \nabla
_{\varepsilon }u\Vert _{L^{2}}\Vert \nabla _{\varepsilon }v\Vert
_{L^{2}}+C(L)\sqrt{\varepsilon }\Vert \nabla _{\varepsilon }v\Vert
_{L^{2}}^{2},  \label{claim-R0}
\end{equation}%
\textit{Proof of (\ref{claim-R0}).} First, we have 
\begin{equation*}
\iint \frac{1}{2}\partial _{x}\left\{ \frac{1}{u_{s}}\right\}
u_{y}^{2}-\partial _{y}\left\{ \frac{1}{u_{s}}\right\} u_{y}u_{x}\leq C\Vert
\nabla u_{s}\Vert _{\infty }\Vert u_{y}\Vert _{L^{2}}(\Vert u_{y}\Vert
_{L^{2}}+\Vert v_{y}\Vert _{L^{2}})
\end{equation*}%
in which the bounds \eqref{bound-usx} and \eqref{bound-usy} gives $\Vert
\nabla u_{s}\Vert _{\infty }\lesssim 1$. Next, upon recalling the definition 
$\nabla _{\veps}=(\sqrt{\varepsilon }\partial _{x},\partial _{y})$, the H%
\"{o}lder inequality yields 
\begin{equation*}
\begin{aligned} \iint &-\frac{\varepsilon }{2}\partial _{x}\left\{
\frac{1}{u_{s}}\right\} u_{x}^{2}+\frac{\varepsilon }{2}\partial _{x}\left\{
\frac{1}{u_{s}}\right\} v_{y}^{2}-\varepsilon \partial _{y}\left\{
\frac{1}{u_{s}}\right\} v_{y}v_{x}-\frac{\varepsilon ^{2}}{2}\partial
_{x}\left\{ \frac{1}{u_{s}}\right\} v_{x}^{2} \\ &\le C \| \nabla u_s
\|_\infty\Big( \|\nabla _{\varepsilon }u\|_{L^2}^{2}+\|\nabla _{\varepsilon
}u\|_{L^2}\|\nabla _{\varepsilon }v\|_{L^2}+\sqrt \varepsilon \|\nabla
_{\varepsilon }v\|_{L^2}^{2}\Big) \end{aligned}
\end{equation*}%
which again gives the bound as claimed, since $\Vert \nabla u_{s}\Vert
_{\infty }\lesssim 1$. We now estimate the third line on the right of $R^{0}$%
. We first have 
\begin{equation*}
\begin{aligned} \iint \partial _{y}\left\{ \frac{v\partial
_{y}u_{s}}{u_{s}^{2}}\right\} u_{y} &= \iint \left\{ \frac{\partial
_{y}u_{s}}{u_{s}^{2}}\right\} u_{y} v_y + \iint \partial _{y}\left\{
\frac{\partial _{y}u_{s}}{u_{s}^{2}}\right\} u_{y} v \\ &\le C \Big( \|
u_{sy}\|_\infty + \sup_x \Big(\int_0^\infty y (|u_{syy}|^2 + |u_{sy}|^4)
\Big)^{1/2}\Big) \| u_y\|_{L^2} \| v_y\|_{L^2}, \end{aligned}
\end{equation*}%
and 
\begin{equation*}
\begin{aligned} \iint \varepsilon \partial _{x}\left\{ \frac{v\partial
_{y}u_{s}}{u_{s}^{2}}\right\} u_{x} &= - \varepsilon \iint \left\{
\frac{\partial _{y}u_{s}}{u_{s}^{2}}\right\} v_{y} v_x -\varepsilon \iint
\partial _{x}\left\{ \frac{\partial _{y}u_{s}}{u_{s}^{2}}\right\} v_{y} v \\
&\le C \Big( \sqrt \veps \| u_{sy}\|_\infty + \varepsilon \sup_x \Big(
\int_0^\infty y (|u_{sxy}|^2 + |u_{sx}u_{sy}|^2) \Big)^{1/2}\Big) \|
\nabla_\veps v\|^2_{L^2}. \end{aligned}
\end{equation*}%
The last two terms on the third line on the right of $R^{0}$ can be
estimated very similarly. We give bounds on the norms of $[u_{s},v_{s}]$.
Similarly as done in \eqref{uintegral-dy2}, we get 
\begin{eqnarray*}
\sup_{x}\int y|u_{syy}|^{2}dy &\leq &\sup_{x}\int y\Big[\varepsilon
^{2}|u_{ezz}^{0}|^{2}+|u_{pyy}^{0}|^{2}+\varepsilon ^{3}|u_{ezz}^{1}|^{2}%
\Big]\;dy \\
&\leq &C(u_{p}^{0})+\sup_{x}\int \Big [\varepsilon
z|u_{ezz}^{0}|^{2}+\varepsilon ^{2}z|u_{ezz}^{1}|^{2}\Big ]\;dz \\
&\leq &C(u_{e}^{0},u_{p}^{0})+\varepsilon ^{2}L^{-1}\Big(\iint
z|u_{ezz}^{1}|^{2}dxdz+\iint z|u_{exzz}^{1}|^{2}dxdz\Big) \\
&\leq &C(u_{e}^{0},u_{p}^{0})+\varepsilon ^{2}L^{-1}\Vert \langle z\rangle
^{n}v_{e}^{1}\Vert _{H^{3}}^{2} \\
&\lesssim &1+C(L)\varepsilon ,
\end{eqnarray*}%
and 
\begin{eqnarray*}
\sup_{x}\int y|u_{sy}|^{4}dy &\leq &\sup_{x}\int y\Big[\varepsilon
^{2}|u_{ez}^{0}|^{4}+|u_{py}^{0}|^{4}+\varepsilon ^{4}|u_{ez}^{1}|^{4}\Big]%
\;dy \\
&\leq &C(u_{e}^{0},u_{p}^{0})+\varepsilon ^{3}L^{-1}\Big(\iint
z|u_{ez}^{1}|^{4}dxdz+\iint z|u_{exz}^{1}|^{4}dxdz\Big) \\
&\leq &C(u_{e}^{0},u_{p}^{0})+\varepsilon ^{3}L^{-1}\Vert \langle z\rangle
^{n}v_{e}^{1}\Vert _{W^{2,4}}^{2} \\
&\lesssim &1+C(L)\varepsilon ^{3},
\end{eqnarray*}%
both of which are thus bounded, thanks to the estimates from Lemma \ref%
{lem-ve}. Same bounds can be given for the weighted integrals of $%
\varepsilon |u_{sxy}|^{2}$ and $\varepsilon |u_{sx}|^{4}$, using the extra
factor of $\varepsilon $ in these integrals. In addition, we have 
\begin{eqnarray*}
\sup_{x}\int \varepsilon y|u_{sxx}|dy &\leq &\sup_{x}\int \varepsilon y\Big[%
|u_{pxx}^{0}|+\varepsilon ^{1/2}|u_{exx}^{1}|\Big]\;dy \\
&\leq &C(u_{p}^{0})+\varepsilon ^{1/2}\sup_{x}\int z|v_{exz}^{1}|\Big ]\;dz
\\
&\leq &C(u_{p}^{0})+\varepsilon ^{1/2}L^{-1}\Big(\iint
z|v_{exz}^{1}|dxdz+\iint z|v_{exxz}^{1}|dxdz\Big) \\
&\leq &C(u_{e}^{0},u_{p}^{0})+\varepsilon ^{1/2}L^{-1}\Vert \langle z\rangle
^{n}v_{e}^{1}\Vert _{W^{3,q}}^{2} \\
&\lesssim &1+C(L)\varepsilon ^{-3/2+2/q},
\end{eqnarray*}%
which are again bounded, thanks to the Euler bounds from Lemma \ref{lem-ve},
with $q$ being arbitrarily close to $1$.

We now give bounds on the last line on the right of $R^{0}$. We have 
\begin{equation*}
\begin{aligned} 
\iint& \partial _{y}\left\{ \frac{v}{u_{s}}\right\}
\{\partial _{x}u_{s}u +v_{s}u_{y}\} 
\\&= \iint \partial _{y}\left\{
\frac{1}{u_{s}}\right\} \{\partial _{x}u_{s}u +v_{s}u_{y}\} v + \iint
\left\{ \frac{1}{u_{s}}\right\} \{\partial _{x}u_{s}u +v_{s}u_{y}\} v_y
\\&\le C \Big( \sup_x \Big( \int_0^\infty y( |u_{sy}u_{sx}| + |u_{sy}v_s|^2
+ |u_{sx}|^2 )\Big)^{1/2} + \|v_s\|_\infty \Big) \| u_y \|_{L^2}
\|v_y\|_{L^2} \end{aligned}
\end{equation*}%
in which we note that $v_{s}$ is uniformly bounded. The estimate %
\eqref{uintegral} gives the weighted bound on $u_{sy}$. Next, we estimate 
\begin{eqnarray*}
\sup_{x}\int y\{u_{sx}\}^{2}dy &\leq &\sup_{x}\int y\Big[|u_{px}^{0}|^{2}+%
\varepsilon |u_{ex}^{1}|^{2}\Big]\;dy \\
&\leq &C(u_{p}^{0})+\sup_{x}\int z|v_{ez}^{1}|^{2}\;dz \\
&\leq &C(u_{p}^{0})+C\Big(\int z|v_{ez}^{1}|_{|_{x=0}}^{2}\;dz+\iint
z|v_{ez}^{1}|^{2}dxdz+\iint z|v_{exz}^{1}|^{2}dxdz\Big), \\
&\lesssim &1.
\end{eqnarray*}%
This gives the desired bound on the first term on the last line in $R^{0}$.
Next, we have 
\begin{equation*}
\begin{aligned} \varepsilon & \iint \partial _{x}\left\{ \frac{v}{u_{s}}\right\} \{u\partial _{x}v_{s}+v_{s}v_{y}
+v\partial _{y}v_{s}\} \\ &=
 \varepsilon \iint \partial _{x}\left\{ \frac{1}{u_{s}}\right\} \{u\partial _{x}v_{s}
+v_{s}v_{y}+v\partial _{y}v_{s}\} v + \varepsilon \iint \left\{ \frac{1}{u_{s}} \right\} 
\{u\partial _{x}v_{s}+v_{s}v_{y}+v\partial _{y}v_{s}\} v_x \\&\le C \varepsilon 
\sup_x \Big( \int_0^\infty y ( |u_{sx}v_{sx}| + |u_{sx}v_s|^2 + |u_{sx}v_{sy}| ) \Big)^{1/2}
 \| u_y \|_{L^2} \|v_y\|_{L^2} \\&\quad +C(L) \varepsilon \Big(  ||v_{sx}||_\infty 
+ \|v_s\|_\infty +\sup_x \Big( \int_0^\infty y |v_{sy}|^2 \Big)^{1/2}  \Big)\| u_x \|_{L^2} \|v_x\|_{L^2} \end{aligned}
\end{equation*}%
in which the last estimate used the inequality: $\Vert u\Vert _{L^{2}}\leq L\Vert u_{x}\Vert _{L^{2}}$ and the divergence-free condition $v_y = -u_x$. 
As estimated above, it remains to give a uniform estimate on 
\begin{equation}
||v_{sx}||_{\infty }\leq ||v_{px}^{0}||_{\infty }+||v_{ex}^{1}||_{\infty
}\lesssim 1+\Vert v_{e}^{1}\Vert _{W^{2,q}}^{2}\leq C(L)  \notag
\end{equation}%
thanks to the estimates on $v_{e}^{1}$, with $q>2$. Finally, we estimate 
\begin{equation*}
\varepsilon \iint \frac{\partial _{y}u_{s}vv_{x}}{u_{s}^{2}}\leq
C\varepsilon \sup_{x}\Big(\int_{0}^{\infty }y|u_{sy}|^{2}\Big)^{1/2}\Vert
v_{y}\Vert _{L^{2}}\Vert v_{x}\Vert _{L^{2}}
\end{equation*}%
in which the integral $\int y|u_{sy}|^{2}\;dy$ is already estimated in %
\eqref{uintegral}. Putting all above estimates together, we have completed
the proof for the claim \eqref{claim-R0}. 

Finally, using the Young inequality and the smallness of $\varepsilon $, the 
$L^{2}$ norm of $\nabla _{\veps}v$ can be absorbed into the left-hand side
of \eqref{vorticity}.

This completes the proof of the positivity estimate and the lemma.
\end{proof}

\subsection{Proof of Proposition \protect\ref{prop-stability}}

\label{sec-proofStability}

The proof of Proposition \ref{prop-stability} now follows straightforwardly
from the energy estimate (Lemma \ref{lem-EE}) and the positivity estimate
(Lemma \ref{lem-positivity}), as a direct application of the Schaefer's
fixed point theorem; see \cite[Theorem 4, p. 504]{Evans}. Indeed, first
combining these estimates together and choosing $L$ sufficiently small, we
get 
\begin{equation*}
\|\nabla _{\varepsilon }u\|_{L^2}+\|\nabla _{\varepsilon }v\|_{L^2}\quad
\le\quad C(u_s, v_s) \Big[ \|f\|_{L^2}+\sqrt{\varepsilon }\|g\|_{L^2}\Big]
\end{equation*}
uniformly in $N$. Taking $N \to \infty$ yields the stability estimate %
\eqref{stability}.

To apply the fixed point theorem of Schaefer (\cite[Theorem 4, p. 504]{Evans}%
), we consider the following system

\begin{equation*}
\begin{aligned} u + p_{x}-\Delta _{\varepsilon }u &=\lambda \Big( f + u -
[u_{s}u_{x}+uu_{sx}+v_{s}u_{y}+vu_{sy}]\Big) \\ \frac v \veps
+\frac{p_{y}}{\varepsilon }-\Delta _{\varepsilon }v &=\lambda \Big( g +
\frac v\veps - [u_{s}v_{x}+uv_{sx}+v_{s}v_{y}+vv_{sy}]\Big) \\ u_x + v_y
&=0, \end{aligned}
\end{equation*}
or in the operator form $S [u,v] = \lambda T[u,v]$, for parameter $\lambda
\in [0,1]$. The existence of a solution to \eqref{linearu}-\eqref{linearv}
is equivalent to the existence of a fixed point of $S^{-1}T$. The
compactness of the operator $S^{-1}T$ follows directly from that of the
Stokes operator. To derive uniform bounds on the set of solutions $%
[u^\lambda, v^\lambda]$, we may rewrite the above system as 
\begin{equation*}
\begin{aligned} (1-\lambda)u + \lambda
[u_{s}u_{x}+uu_{sx}+v_{s}u_{y}+vu_{sy}]+ p_{x}-\Delta _{\varepsilon }u
&=\lambda f \\ (1-\lambda)\frac v \veps + \lambda
[u_{s}v_{x}+uv_{sx}+v_{s}v_{y}+vv_{sy}] +\frac{p_{y}}{\varepsilon }-\Delta
_{\varepsilon }v &=\lambda g \end{aligned}
\end{equation*}
together with the divergence-free condition and the same boundary conditions %
\eqref{lin-BCs}. The uniform estimates now follow almost identically from
the above energy estimates and positivity estimates. We omit to repeat the
details. This completes the proof of Proposition \ref{prop-stability}.

\section{$L^\infty$ estimates}

\label{sec-sup}

In order to perform nonlinear iteration, we shall need to derive bounds in $%
L^\infty$ for the solution. We prove the following:

\begin{lemma}
\label{elliptic}Consider the scaled Stokes system%
\begin{equation*}
p_{x}-\Delta _{\varepsilon }u=f,\text{ \ \ \ \ \ \ \ \ \ \ \ \ \ \ \ }\frac{%
p_{y}}{\varepsilon }-\Delta _{\varepsilon }v=g
\end{equation*}%
together with the divergence-free condition $u_{x}+v_{y}=0$ and the (same)
boundary conditions 
\begin{equation}  \label{Stokes-BCs}
\begin{aligned} \lbrack u,v]_{y=0} &=0\text{ (no-slip)},\text{ \ \ \ \ \ \ \
\ \ \ \ }[u,v]_{x=0}=0\text{ \ (Dirichlet)}, \\ p-2\varepsilon u_{x}
&=0,\text{ \ \ \ \ }u_{y}+\varepsilon v_{x}=0\quad \text{ \ \ at }x=L\text{
(Neumann or stress-free)}. \end{aligned}
\end{equation}
Then, there holds 
\begin{equation*}
\varepsilon ^{\frac{\gamma }{4}}\Vert u\Vert _{\infty }+\varepsilon ^{\frac{%
\gamma }{4}+\frac{1}{2}}\Vert v\Vert _{\infty }\lesssim C_{\gamma, L} \Big \{%
\| u\|_{H^1} + \sqrt \varepsilon \| v\|_{H^1} +\|f\|_{L^2}+\sqrt{\varepsilon 
}\|g\|_{L^2}\Big \},
\end{equation*}
for some constant $C_{\gamma, L}$.
\end{lemma}

\begin{proof}
Since our rectangle domain can be covered by two $C^{0,1}$ charts, we may apply the standard 
extension theorem (see, for instance, \cite[Theorem 5.4]{PNV}) such that there exists $\bar{u}\in
H^{1+\beta }(\mathbb{R}^{2})$ and $\bar{v}\in H^{1+\beta }(\mathbb{R}%
^{2})$, for $\beta \in (0,1)$,  such that $[\bar{u}, \bar v] = [u,v]$ in $[0,L]\times \RR_+$, and 
\begin{equation*}
\|\bar{u}\|_{H^{1+\beta }}\leq C_{\beta ,L}\|u\|_{H^{1+\beta }},\text{ \
\ \ \ \ \ \ }\|\bar{v}\|_{H^{1+\beta }}\leq C_{\beta ,L
}\|v\|_{H^{1+\beta }},
\end{equation*}%
for some constant $C_{\beta,L}$ that depends only on $\beta$ and $L$. 
By the Sobolev's imbedding in $\RR^2$ and an interpolation inequality for $\bar{u}
$ and $\bar{v}$, we have for any $0<\tau <\alpha ,$ 
\begin{eqnarray*}
\Vert u\Vert _{\infty } &\leq &\|\bar{u}\|_{\infty }\leq \|\bar{u}%
\|_{H^{1+\tau }} \\
&\leq &C_{\tau ,L }[\|\bar{u}\|_{H^{1}}]^{\frac{\alpha -\tau }{\alpha }%
}[\|\bar{u}\|_{H^{1+\alpha }}]^{\frac{\tau }{\alpha }} \\
&\leq &C_{\tau ,\alpha ,L }[\|u\|_{H^{1}}]^{\frac{\alpha -\tau }{\alpha }%
}[\|u\|_{H^{1+\alpha }}]^{\frac{\tau }{\alpha }},\end{eqnarray*}%
and similarly, 
$$\varepsilon ^{1/2}\Vert v\Vert _{\infty } \le C_{\tau ,\alpha ,L
}[\varepsilon ^{1/2}\|v\|_{H^{1}}]^{\frac{\alpha -\tau }{\alpha }%
}[\varepsilon ^{1/2}\|v\|_{H^{1+\alpha }}]^{\frac{\tau }{\alpha }}
$$
for some constant $C_{\tau ,\alpha ,L
}$. Here, we have used the standard interpolation between Sobolev spaces $H^{1},H^{1+\tau },$
and $H^{1+\alpha }$, with $0<\tau <\alpha $. We note that thanks to our uniform estimates for $\|u\|_{H^{1}}$ and $
\varepsilon ^{1/2}\|v\|_{H^{1}},$ it suffices to give estimates on the $H^{1+\alpha}$ norm of $[u,\sqrt \veps v]$.

In what follows, we fix $\alpha\sim 1/2$ and take $\tau$ so that $\tau <<\alpha$. We claim that there exists a possibly large
number $m_{\alpha }>0$ such that there holds 
\begin{equation}
\Vert u\Vert _{H^{1+\alpha }}+\sqrt{\varepsilon }\Vert v\Vert _{H^{1+\alpha
}}\lesssim \varepsilon ^{-m_{\alpha }}\{\|f\|_{L^2}+\sqrt \veps \|g\|_{L^2}+\| u\|_{H^1} + \sqrt \veps \| v\|_{H^1} \}  \label{claim}.
\end{equation}%
Given the claim, we then have 
$$ \Vert u\Vert _{\infty}+\sqrt{\varepsilon }\Vert v\Vert _{\infty} \le C_{\tau, \alpha, L} \veps^{-\frac{m_\alpha \tau}{\alpha}}\Big\{ \|f\|_{L^2}+\sqrt \veps\|g\|_{L^2} + \| u\|_{H^1} + \sqrt \veps \| v\|_{H^1} \Big\}.$$
The lemma would then be proved at once by choosing $\tau \ll \alpha$ so that $\frac{m_\alpha \tau}{\alpha} \le \frac\gamma 4$. 

We shall now prove the claim for $\alpha = 1/2$ and $m_\alpha = 7/4$. To do so,  let us introduce the (original) scaling: 
\begin{equation}
u_{\varepsilon }(x,y)\equiv u(L x,L \frac{y}{\sqrt{\varepsilon }}),%
\text{ \ \ \ \ \ \ \ }v_{\varepsilon }(x,y)\equiv \sqrt{\varepsilon }v(L
x,L \frac{y}{\sqrt{\varepsilon }}),\text{ \ \ \ \ }p_{\varepsilon
}(x,y)\equiv \frac{L }{\varepsilon }p(L x,L \frac{y}{\sqrt{%
\varepsilon }}).  \label{scaling}
\end{equation}%
Clearly, direct calculations yield $u_{\varepsilon x}+v_{\varepsilon y}=0$ and 
\begin{eqnarray*}
\partial _{xx}u_{\varepsilon }(x,y) &=&L ^{2}u_{xx}(L x,L \frac{y}{%
\sqrt{\varepsilon }}),\text{ \ \ } \quad \partial _{yy}u_{\varepsilon }(x,y)=\frac{%
L ^{2}}{\varepsilon }u_{yy}(L x,L \frac{y}{\sqrt{\varepsilon }}), \\
\partial _{xx}v_{\varepsilon }(x,y) &=&L ^{2}\sqrt{\varepsilon }%
v_{xx}(L x,L \frac{y}{\sqrt{\varepsilon }}),\text{ \ \ }\quad \partial
_{yy}v_{\varepsilon }(x,y)=\frac{L ^{2}}{\varepsilon ^{1/2}}v_{yy}(L
x,L \frac{y}{\sqrt{\varepsilon }}) \\
\partial _{x}p_{\varepsilon }(x,y) &=&\frac{L ^{2}}{\varepsilon }%
p_{x}(L x,L \frac{y}{\sqrt{\varepsilon }}),\text{ \ \ }\quad \partial
_{y}p_{\varepsilon }(x,y)=\frac{L ^{2}}{\varepsilon ^{3/2}}p_{y}(L
x,L \frac{y}{\sqrt{\varepsilon }}).
\end{eqnarray*}%
Plugging these in the Stokes problem, we yield a normalized Stokes system: 
\begin{equation}
p_{\varepsilon x}-\Delta u_{\varepsilon }=\varepsilon ^{-1}L ^{2}f(L
x,L \frac{y}{\sqrt{\varepsilon }}),\text{ \ \ \ }p_{\varepsilon y}-\Delta
v_{\varepsilon }=\varepsilon ^{-\frac{1}{2}}L ^{2}g(L x,L \frac{y}{%
\sqrt{\varepsilon }})  \label{normalstokes}
\end{equation}%
in a fixed domain $0\leq x\leq 1$ and $0\leq y\leq \infty \,$\ with boundary
conditions%
\begin{eqnarray*}
\lbrack u_{\varepsilon },v_{\varepsilon }] &=&0\text{ on both boundaries: }\{y=0\}\text{ and }\{x=0\} \\
p_{\varepsilon }+2u_{\varepsilon x} &=&0,\text{\ \ }v_{\varepsilon
x}+u_{\varepsilon y}=0\text{ at }x=1.
\end{eqnarray*}%
We now invoke the standard elliptic estimate for the Stoke problem in such a fixed domain. Recall that there holds the Poincare's inequality: $\|u_{\varepsilon
}\|_{L^2}\leq \|u_{\varepsilon x}\|_{L^2},\|v_{\varepsilon }\|_{L^2}\leq
\|v_{\varepsilon x}\|_{L^2}$. Next, the standard energy estimates yield  
\begin{eqnarray}
\|\nabla u_{\varepsilon }\|_{L^2}+\|\nabla v_{\varepsilon }\|_{L^2} &\leq
&C\{\varepsilon ^{-1}L ^{2}\|f(Lx, \frac{Ly}{\sqrt{\varepsilon }}%
)\|_{L^2}+\varepsilon ^{-\frac{1}{2}}L ^{2}\|g(Lx, \frac{Ly}{\sqrt{\varepsilon }}%
)\|_{L^2}\}  \label{h1estimate} \\
&\leq &C\varepsilon ^{-3/4}L \Big[  \|f\|_{L^2}+\sqrt \varepsilon
\|g\|_{L^2}\Big ].  \notag
\end{eqnarray}%

Next, let us give an $L^{2}$ estimate on the pressure $p_{\varepsilon }$. First, 
for $h\in L^{2}$, we show that there is a vector-valued function $\phi
\in H^{1}$ such that $\phi (0,y)\equiv 0,$ $\phi (x,0)\equiv 0,$ $\nabla
\cdot \phi =h$ and $\|\phi \|_{H^{1}}\lesssim \|h\|_{L^2}.$ Indeed, we
can decompose $h=\sum_{n=0}^{\infty }\mathbf{1}_{n\leq y<n+1}(y)h, $
 for characteristic functions $\mathbf{1}_{n\leq y<n+1}(y)$.  Hence, for each $n$, the function $q_n=
\mathbf{1}_{n\leq y<n+1}(y)h$ is supported in a unit square $[0,1]\times [n,n+1]$. By \cite[page 27]{O}, we
can then find $\phi _{n}\in H^{1},$ such that $\nabla \cdot \phi _{n}=h_n$ on the unit square,
with $\phi _{n}(0,y)\equiv \phi _{n}(x,n)\equiv \phi
_{n}(x,n+1) \equiv 0$. Furthermore, we have $\|\phi
_{n}\|_{H^{1}}\leq C\|h_n\|_{L^2},$ uniformly
in $n.$ If we now define $\phi \equiv \sum_{n=1}^{\infty }\phi _{n}$, it then follows that
 $\|\phi \|_{H^{1}}\leq C\|h\|_{L^2},\nabla \cdot \phi =h$ and $\phi
(0,y)\equiv 0,$ $\phi (x,0)\equiv 0.$ 

The pressure estimate now follows directly from the existence of the vector field $\phi$. Indeed, we approximate $p_\veps$ by smooth functions of the form $q= \nabla \cdot \phi$ so that $\| p_\veps \|_{L^2} \lesssim \| q\|_{L^2} = \| \nabla \phi\|_{L^2}.$	Then, we can use the vector field $\phi$ as a test function to the Stokes problem. This, together with the Young inequality, immediately yields 
\begin{eqnarray*}
\|p_{\varepsilon }\|_{L^2} &\leq &C\Big \{\|\nabla u_{\varepsilon }\|_{L^2}+\|\nabla
v_{\varepsilon }\|_{L^2}+\varepsilon ^{-1}L ^{2}\|f(Lx, \frac{Ly}{\sqrt{\varepsilon }}%
)\|_{L^2}+\varepsilon ^{-\frac{1}{2}}L ^{2}\|g(Lx, \frac{Ly}{\sqrt{\varepsilon }}%
)\|_{L^2}\Big \} \\
&\leq &C\varepsilon ^{-3/4}L \Big[  \|f\|_{L^2}+\sqrt \varepsilon
\|g\|_{L^2}\Big ]\end{eqnarray*}%
thanks to the estimate \eqref{h1estimate}. 

Now, it remains to derive estimates in the higher regularity norms. We multiply the Stokes system by an arbitrary cut-off function $\chi (x,\frac{y}{\sqrt{\varepsilon }})$
to obtain:%
\begin{eqnarray*}
&&\{\chi p_{\varepsilon }\}_{x}-\Delta \{\chi u_{\varepsilon }\}=\chi
_{x}p_{\varepsilon }+u_{\varepsilon }\Delta \chi +2\nabla \chi \cdot \nabla
u_{\varepsilon }+\chi \varepsilon ^{-1}L ^{2}f \\
&&\{\chi p_{\varepsilon }\}_{y}-\Delta \{\chi v_{\varepsilon }\}=\chi
_{y}p_{\varepsilon }+v_{\varepsilon }\Delta \chi +2\nabla \chi \cdot \nabla
v_{\varepsilon }+\chi \varepsilon ^{-\frac{1}{2}}L ^{2}g.
\end{eqnarray*}
If we choose $\chi = \chi _{1}(x,\frac{y}{\sqrt{\varepsilon }})$ which has a compact
support away from the corners, then the Stokes problem has an $H^{2}$
estimate so that 
\begin{eqnarray*}
&&\|\chi _{1}p_{\varepsilon }\|_{H^{1}}+\|\chi _{1}u_{\varepsilon
}\|_{H^{2}}+\|\chi _{1}v_{\varepsilon }\|_{H^{2}} \\
&\leq &\|\chi _{x}p_{\varepsilon }+u_{\varepsilon }\Delta \chi +2\nabla \chi
\cdot \nabla u_{\varepsilon }+\chi \varepsilon ^{-1}L ^{2}f\|_{L^2} +\|\chi _{y}p_{\varepsilon }+v_{\varepsilon }\Delta \chi +2\nabla \chi
\cdot \nabla v_{\varepsilon }+\chi \varepsilon ^{-\frac{1}{2}}L
^{2}g\|_{L^2} \\
&\leq &\varepsilon ^{-1}\{\|u_{\varepsilon }\|_{L^2}+\|v_{\varepsilon
}\|_{L^2}\}+\varepsilon ^{-\frac{1}{2}}\{\|\nabla u_{\varepsilon
}\|_{L^2}+\|\nabla v_{\varepsilon }\|_{L^2}+\|p_{\varepsilon
}\|_{L^2}\}+\varepsilon ^{-3/4}L \|f\|_{L^2}+\varepsilon ^{-1/4}L \|g\|_{L^2}
\\
&\leq &\varepsilon ^{-5/4}L \Big\{ \|f\|_{L^2}+\sqrt \varepsilon
\|g\|_{L^2}\Big\},
\end{eqnarray*}%
upon using the estimate \eqref{h1estimate}. This implies the following uniform estimate for the unscaled solution $%
[u,v,p] $ via the change of variables \eqref{scaling}: 
\begin{eqnarray*}
\|\chi _{1}u_{\varepsilon }\|_{H^{2}} &\approx&\varepsilon ^{-1+\frac{1}{4}%
}\|\nabla ^{2}(\chi_1 u)\|_{L^{2}}+\varepsilon ^{-\frac{1}{2}+\frac{1}{4}}\|\nabla
(\chi_1 u) \|_{L^{2}}+\varepsilon ^{1/4}\|\chi_1u\|_{L^2}, \\
\|\chi _{1}v_{\varepsilon }\|_{H^{2}} &\approx&\varepsilon ^{1/2}\varepsilon ^{-1+%
\frac{1}{4}}\|\nabla ^{2}(\chi_1 v)\|_{L^{2}}+\varepsilon ^{\frac{1}{4}}\|\nabla
(\chi_1 v)\|_{L^{2}}+\varepsilon ^{1/2+1/4}\|\chi_1 v\|_{L^2},
\end{eqnarray*}%
which immediately yield \begin{equation*}
\begin{aligned}
\varepsilon ^{-1+\frac{1}{4}}&\|\nabla ^{2}(\chi_1 u)\|_{L^{2}}+\varepsilon ^{\frac{1}{%
2}}\varepsilon ^{-1+\frac{1}{4}}\|\nabla ^{2}(\chi_1 v)\|_{L^{2}}
\\& \leq \varepsilon ^{-5/4}L \Big\{ \|f\|_{L^2}+\sqrt \varepsilon
\|g\|_{L^2}\Big\} + \veps^{-1/4} \Big\{ \| u\|_{H^1} + \sqrt \veps \| v\|_{H^1}\Big\}.
\end{aligned}\end{equation*}

Next, we choose the cut-off function $\chi =\chi _{2}(x,\frac{y}{\sqrt{\varepsilon }})$ with support near
the corners, around which we do not have any $H^{2}$ estimate of the solution. However, thanks to \cite{OS}, we do have a weaker estimate: precisely, 
\begin{eqnarray*}
\|\chi _{2}p_{\varepsilon }\|_{H^{3/2}}+\|\chi _{2}u_{\varepsilon
}\|_{H^{1+3/2}}+\|\chi _{2}v_{\varepsilon }\|_{H^{1+3/2}} &\leq &\|\chi
_{2x}p_{\varepsilon }+\Delta \chi _{2}u_{\varepsilon }+2\nabla \chi
_{2}\cdot \nabla u_{\varepsilon }+\chi _{2}f^{{}}\|_{L^2} \\
&&+\|\chi _{2y}p_{\varepsilon }+\Delta \chi _{2}v_{\varepsilon }+2\nabla
\chi _{2}\cdot \nabla v^{\varepsilon }+\chi _{2}g\|_{L^2} \\
&\leq &\varepsilon ^{-1}\{\varepsilon ^{-3/4}L \|f\|_{L^2}+\varepsilon
^{-1/4}L \|g\|_{L^2}\}.
\end{eqnarray*}%
Noting that by scaling via \eqref{scaling} there hold \begin{equation*}
\|\nabla \{\chi _{2}u_{\varepsilon }\}\|_{H^{1/2}}=\varepsilon ^{-\frac{3}{4}%
+\frac{1}{4}}\|\nabla \{\chi _{2}u\}\|_{H^{1/2}},\text{ \ \ \ }\|\nabla
\{\chi _{2}v_{\varepsilon }\}\|_{H^{1/2}}=\varepsilon ^{\frac{1}{2}%
}\varepsilon ^{-\frac{3}{4}+\frac{1}{4}}\|\nabla \{\chi _{2}v\}\|_{H^{1/2}}
\end{equation*}%
we thus obtain 
\begin{equation*}
\varepsilon ^{-\frac{3}{2}+\frac{1}{4}}\|\nabla \{\chi
_{2}u\}\|_{H^{1/2}}+\varepsilon ^{\frac{1}{2}}\varepsilon ^{-\frac{3}{4}+%
\frac{1}{4}}\|\nabla \{\chi _{2}v\}\|_{H^{1/2}}\leq \varepsilon
^{-1}\{\varepsilon ^{-3/4}L \|f\|_{L^2}+\varepsilon ^{-1/4}L \|g\|_{L^2}\}.
\end{equation*}

Finally, combining the estimates on $\chi_1[u,v]$ and $\chi_2[u,v]$ yields at once the claimed bound \eqref{claim} for $\alpha = 1/2$ and $m_\alpha = 7/4$. The lemma is thus proved. 
\end{proof}

\section{Proof of the main theorem}

\label{sec-proof} We are now ready to give the proof of our main theorem.
Consider the nonlinear scaled Navier-Stokes equations \eqref{scaledns} and
write the solutions $[U^\veps, V^\veps, P^\veps]$ in the asymptotic
expansion \eqref{expansion}: 
\begin{equation}  \label{expansion1}
\begin{aligned} ~[U^\veps, V^\veps, P^\veps](x,y) &= [u_\mathrm{app},
v_\mathrm{app}, p_\mathrm{app}] (x,y) + \varepsilon^{\gamma+\frac 12}
[u^\veps, v^\veps, p^\veps](x,y) \end{aligned}
\end{equation}%
with the approximate solutions constructed as in Proposition \ref%
{prop-approximate}. We shall now study the equations for the remainder
solutions $[u^\veps, v^\veps, p^\veps]$. Let us recall the leading
approximate solutions (without having the Prandtl layers $[u_p^1,v_p^1]$): 
\begin{equation}  \label{us-re}
u_{s} (x,y) = u_{e}^{0}+u_{p}^{0}+\sqrt{\varepsilon }u_{e}^{1}, \qquad
v_{s}(x,y) = v_{p}^{0}+v_{e}^{1} .
\end{equation}
Then, the remainder solutions $[u^\veps ,v^\veps ,p^\veps]$ solve 
\begin{subequations}
\begin{align}
u_{s}u^\veps_{x}+u^\veps u_{sx}+v_{s}u^\veps_{y}+v^\veps
u_{sy}+p^\veps_{x}-\Delta _{\varepsilon }u^\veps &=R_{1}(u^\veps,v^\veps)
\label{remainderu} \\
u_{s}v^\veps_{x}+u^\veps v_{sx}+v_{s}v^\veps_{y}+v^\veps v_{sy}+\frac{%
p^\veps_{y}}{\varepsilon }-\Delta _{\varepsilon }v^\veps
&=R_{2}(u^\veps,v^\veps)  \label{remainderv} \\
u^\veps_x + v^\veps_y & =0  \label{divuv}
\end{align}
in which $\Delta_\eps = \partial_y^2 + \varepsilon \partial_x^2$ and the
remainders $R_1(u^\veps,v^\veps), R_2(u^\veps,v^\veps)$ are defined by 
\end{subequations}
\begin{equation}  \label{remainders}
\begin{aligned} R_{1}(u^\veps,v^\veps) &: = \varepsilon^{-\gamma - \frac
12}R^u_\mathrm{app} - \sqrt \veps \Big[ (u_p^1 + \varepsilon^{\gamma}
u^\veps) u^\veps_x + u^\veps u^1_{px} + (v_p^1 + \varepsilon^{\gamma}
v^\veps) u^\veps_{y} + v^\veps u_{py}^1\Big] \\ R_{2}(u^\veps,v^\veps) &: =
\varepsilon^{-\gamma - \frac 12}R^v_\mathrm{app} - \sqrt \veps \Big[ (u_p^1
+ \varepsilon^\gamma u^\veps) v^\veps_{x}+u^\veps v^1_{px}+(v_p^1 +
\varepsilon^\gamma v^\veps)v^\veps_{y}+v^\veps v^1_{py} \Big]. \end{aligned}
\end{equation}
Here, the errors $R^u_\mathrm{app}, R^v_\mathrm{app}$ from the approximation
are estimated in Proposition \ref{prop-approximate}.

We shall apply the standard contraction mapping theorem for the existence of
the solutions to the nonlinear problem. Indeed, we introduce the function
space $X$, endowed with the norm: 
\begin{equation*}
\|[u^\veps,v^\veps]\|_{X}\equiv \|\nabla _{\varepsilon
}u^\veps\|_{L^2}+\|\nabla _{\varepsilon }v^\veps\|_{L^2}+\varepsilon ^{\frac{%
\gamma }{2}}\|u^\veps\|_{\infty }+\varepsilon ^{\frac{1}{2}+\frac{\gamma }{2}%
}\|v^\veps\|_{\infty }.
\end{equation*}
Now, for each $[u^\veps,v^\veps]\in X$, we solve the following linear
problem for $[\bar u^\veps, \bar v^\veps]$: 
\begin{equation}  \label{eqs-iteration}
\begin{aligned} u_{s}\bar{u}^\veps_{x}+\bar{u}^\veps
u_{sx}+v_{s}\bar{u}^\veps_{y}+\bar{v}^\veps u_{sy}+\bar{p}^\veps_{x}-\Delta
_{\varepsilon }\bar{u}^\veps &=R_{1}(u^\veps,v^\veps), \\
u_{s}\bar{v}^\veps_{x}+\bar{u}^\veps
v_{sx}+v_{s}\bar{v}_{y}^\veps+\bar{v}^\veps
v_{sy}+\frac{\bar{p}^\veps_{y}}{\varepsilon }-\Delta _{\varepsilon
}\bar{v}^\veps &=R_{2}(u^\veps,v^\veps). \end{aligned}
\end{equation}%
We are in the position to use the linear stability estimates obtained in
Proposition \ref{prop-stability}, yielding 
\begin{eqnarray}  \label{est-baruv}
\|\nabla _{\varepsilon }\bar{u}^\veps\|_{L^2}+\|\nabla _{\varepsilon }\bar{v}%
^\veps\|_{L^2} &\leq &\|R_{1}(u^\veps,v^\veps)\|_{L^{2}}+\sqrt{\varepsilon }%
\|R_{2}(u^\veps,v^\veps)\|_{L^2}.
\end{eqnarray}
We give estimates on the remainders $R_1, R_2$, defined as in %
\eqref{remainders}. Proposition \ref{prop-approximate} yields 
\begin{equation*}
\varepsilon^{-\gamma -\frac 12} \Big[\| R^u_\mathrm{app} \|_{L^2} + \sqrt
\varepsilon \| R^v_\mathrm{app}\|_{L^2} \Big] \le C(L, u_s, v_s)
\varepsilon^{-\gamma - \kappa+ \frac 14},
\end{equation*}
for arbitrary $\kappa>0$. In what follows, we take any $\gamma<1$ and take $%
\kappa$ so that $\gamma + \kappa < 1/4$. In addition, by a view of the
definition of $\| \cdot \|_X$, the divergence-free condition $u_x = -v_y$,
and the fact that $[u_p^1, v_p^1]$ are nearly bounded, we estimate 
\begin{equation*}
\begin{aligned} \sqrt \veps \| (u_p^1 + \varepsilon^{\gamma} u^\veps)
u^\veps_x\|_{L^2} &\le \sqrt \veps (\| u_p^1\|_\infty + \varepsilon^\gamma
\|u^\veps\|_\infty )\| u^\veps_x\|_{L^2} \\& \le \varepsilon^{1/2-\kappa}
\Big( C(u_p^1) + \varepsilon^{\frac \gamma 2} \|[u^\veps,v^\veps]\|_{X} \Big
) \|[u^\veps,v^\veps]\|_{X} \\ \sqrt \veps \| (v_p^1 + \varepsilon^{\gamma}
v^\veps) u^\veps_y\|_{L^2} &\le \sqrt \veps (\| v_p^1\|_\infty +
\varepsilon^\gamma \|v^\veps\|_\infty )\| u^\veps_y\|_{L^2} \\& \le\Big(
C(v_p^1) \varepsilon^{1/2-\kappa} + \varepsilon^{\frac \gamma 2}
\|[u^\veps,v^\veps]\|_{X} \Big ) \|[u^\veps,v^\veps]\|_{X} . \end{aligned}
\end{equation*}
Also, using the inequality $|[u^\veps, v^\veps] |\le \sqrt y \| [u^\veps_y,
v^\veps_y] \|_{L^2(\mathbb{R}_+)}$ and the uniform $H^1$ bounds obtained in %
\eqref{key-Prandtl1} on $u_p^1$, we get 
\begin{equation*}
\begin{aligned} \sqrt \veps \| u^\veps u^1_{px} + v^\veps u_{py}^1\|_{L^2}
&\le \sqrt \veps \sup_x\Big( \| \langle y \rangle^n u^1_{px}\|_{L^2(\RR_+)}
+ \| \langle y \rangle ^n \| u^1_{py}\|_{L^2(\RR_+)} \Big) \| [u^\veps_y,
v^\veps_y]\|_{L^2} \\& \le \varepsilon^{1/2-\kappa} C(u_p^1)
\|[u^\veps,v^\veps]\|_{X}. \end{aligned}
\end{equation*}

Similarly for terms in $R_2$, we have 
\begin{equation*}
\begin{aligned} \sqrt \veps \| (u_p^1 + \varepsilon^{\gamma} u^\veps)
v^\veps_x\|_{L^2} &\le \sqrt \veps (\| u_p^1\|_\infty + \varepsilon^\gamma
\|u^\veps\|_\infty )\| v^\veps_x\|_{L^2} \\& \le \varepsilon^{-\kappa} \Big(
C(u_p^1) + \varepsilon^{\frac \gamma 2} \|[u^\veps,v^\veps]\|_{X} \Big )
\|[u^\veps,v^\veps]\|_{X} \\ \sqrt \veps \| (v_p^1 + \varepsilon^{\gamma}
v^\veps) v^\veps_y\|_{L^2} &\le \sqrt \veps (\| v_p^1\|_\infty +
\varepsilon^\gamma \|v^\veps\|_\infty )\| v^\veps_y\|_{L^2} \\ &\le
\varepsilon^{1/2-\kappa} \Big( C(v_p^1) + \varepsilon^{\frac \gamma 2}
\|[u^\veps,v^\veps]\|_{X} \Big ) \|[u^\veps,v^\veps]\|_{X} \end{aligned}
\end{equation*}
and 
\begin{equation*}
\begin{aligned} \sqrt \veps \| u^\veps v^1_{px} + v^\veps v_{py}^1\|_{L^2}
&\le \sqrt \veps \Big(\| u^\veps\|_\infty \| v^1_{px}\|_{L^2} + \sup_x\|
\langle y \rangle ^n \| v^1_{py}\|_{L^2(\RR_+)} \| v^\veps_y\|_{L^2}\Big)
\\& \le C(v_p^1) \varepsilon^{1/4-\frac \gamma
2-\kappa}\|[u^\veps,v^\veps]\|_{X}, \end{aligned}
\end{equation*}
in which we have used the bound \eqref{est-vpx}: $\| v_{px}^1 \|_{L^2}
\lesssim \varepsilon^{-1/4-\kappa}$.

Combining the above estimates into \eqref{est-baruv} yields 
\begin{eqnarray}  \label{est-R12proof}
&&\|R_{1}(u^\veps,v^\veps)\|_{L^{2}}+\sqrt{\varepsilon }\|R_{2}(u^\veps,v^%
\veps)\|_{L^2}  \notag \\
&&\qquad \leq C(L, u_s, v_s) \varepsilon^{-\gamma -\kappa + \frac 14} +
\varepsilon ^{- \frac{\gamma }{2}+ \frac
12}\{\|[u^\veps,v^\veps]\|_{X}+\|[u^\veps,v^\veps]\|_{X}^{2}\} ,
\end{eqnarray}
upon noting that $\gamma + \kappa \le \frac 14$ and $\varepsilon \ll 1$, and
hence the bound on the gradient of $\bar u^\veps, \bar v^\veps$ by %
\eqref{est-baruv}. It remains to bound the sup norm. We now use the
estimates on the Stokes problem from Lemma \ref{elliptic} with 
\begin{equation*}
\begin{aligned}f&=-u_{s}u^\veps_{x}-u_{sx}u^\veps-v_{s}u^\veps_{y}-v^\veps
u_{sy}+R_{1} \\ g&=-u_{s}v^\veps_{x}-u^\veps v_{sx}-v_{s}v^\veps_{y}-v^\veps
v_{sy}+R_{2}.\end{aligned}
\end{equation*}
The Stokes estimates, together with \eqref{est-baruv}, yield 
\begin{eqnarray}
\varepsilon ^{\frac{\gamma }{2}}\|\bar{u}^\veps\|_{\infty }+\varepsilon ^{%
\frac{1}{2}+\frac{\gamma }{2}}\|\bar{v}^\veps\|_{\infty } &\leq &C_{\gamma
}\varepsilon^{\frac \gamma 4}\{\|\nabla _{\varepsilon }\bar{u}%
^\veps\|_{L^2}+\|\nabla _{\varepsilon }\bar{v}^\veps\|_{L^2}\}  \notag \\
&&+\varepsilon^{\frac \gamma
4}\|-u_{s}u^\veps_{x}-u_{sx}u^\veps-v_{s}u^\veps_{y}-v^\veps
u_{sy}+R_{1}\|_{L^2}  \label{supestimate} \\
&&+\varepsilon^{\frac \gamma 4+ \frac 12}\|-u_{s}v^\veps_{x}-u^\veps
v_{sx}-v_{s}v^\veps_{y}-v^\veps v_{sy}+R_{2}\|_{L^2} .  \notag
\end{eqnarray}%
The estimates \eqref{est-baruv} and \eqref{est-R12proof} have yielded the
desired estimates on $\nabla_\veps [\bar u^\veps, \bar v^\veps]$ and on $%
R_1, R_2$. It remains to give estimates on terms on the right-hand side of
the above that involve $[u_s, v_s]$. Indeed, since $[u_s, v_s]$ are nearly
bounded, we have 
\begin{equation*}
\begin{aligned} \varepsilon^{\frac \gamma 4} \|u_{s}u^\veps_{x}+
v_{s}u^\veps_{y} \|_{L^2} &\le \varepsilon^{\frac \gamma 4-\kappa}C(L, u_s,
v_s) ( \| u^\veps_x\|_{L^2} + \| u^\veps_y \|_{L^2} ) \le \varepsilon^{\frac
\gamma 4}C(u_s, v_s) \| [u^\veps, v^\veps\|_{X} \\ \varepsilon^{\frac \gamma
4} \|u_{sx}u^\veps +v^\veps u_{sy}\|_{L^2} &\le\varepsilon^{\frac \gamma 4}
\Big( \sup_x \|\sqrt y u_{sx}\|_{L^2(\RR_+)} + \sup_x \|\sqrt y
u_{sy}\|_{L^2(\RR_+)} \Big) \|[u^\veps_y,v^\veps_y]\|_{L^2} \\& \le
\varepsilon^{\frac \gamma 4-\kappa}C(L, u_s, v_s) \| [u^\veps,
v^\veps\|_{X}, \end{aligned}
\end{equation*}
in which the last inequality used the estimates \eqref{uintegral} and the
following estimate: 
\begin{eqnarray*}
\sup_{x}\int y|u_{sx}|^{2}dy &\le& \sup_{x}\int y\Big[|u^0_{px}|^2 +
\varepsilon |u^1_{ex}|^2 \Big] \; dy  \notag \\
&\leq &C(u_{p}^{0})+ \sup_{x} \int z |v_{ez}^{1}|^2 \; dz \\
&\le &C(u_p^0)+ L^{-1}\Big( \iint z|v_{ez}^{1}|^2 dxdz+\iint
z|v_{exz}^{1}|^2dxdz \Big)  \notag \\
&\le & C(L, u_s,v_s).
\end{eqnarray*}
Similarly, we estimate 
\begin{equation*}
\begin{aligned} \varepsilon^{\frac \gamma 4+\frac 12} \|u_{s}v^\veps_{x}+
v_{s}v^\veps_{y} \|_{L^2} &\le \varepsilon^{\frac \gamma 4+\frac
12-\kappa}C(L, u_s, v_s) ( \| v^\veps_x\|_{L^2} + \| v^\veps_y \|_{L^2}) \le
\varepsilon^{\frac \gamma 4}C(u_s, v_s) \| [u^\veps, v^\veps\|_{X} \\
\varepsilon^{\frac \gamma 4+\frac 12} \|v_{sx}u^\veps +v^\veps
v_{sy}\|_{L^2} &\le\varepsilon^{\frac \gamma 4+\frac 12} \sup_x \Big( \|
\sqrt y v_{sx}\|_{L^2} + \|\sqrt y u_{sx}\|_{L^2(\RR_+)} \Big)\|[u^\veps_y
,v^\veps_y]\|_{L^2} \\& \le \varepsilon^{\frac \gamma 4 -\kappa }C(L, u_s,
v_s) \| [u^\veps, v^\veps\|_{X}, \end{aligned}
\end{equation*}
in which we have used 
\begin{eqnarray*}
\sup_{x}\int y|v_{sx}|^{2}dy &\le& \sup_{x}\int y\Big[|v^0_{px}|^2 +
|v^1_{ex}|^2 \Big] \; dy  \notag \\
&\leq &C(v_{p}^{0})+ \varepsilon^{-1}\sup_{x} \int z |v_{ex}^{1}|^2 \; dz \\
&\le &C(u_p^0)+\varepsilon^{-1} L^{-1}\Big( \iint z|v_{ex}^{1}|^2 dxdz+\iint
z|v_{exx}^{1}|^2dxdz \Big)  \notag \\
&\le &\varepsilon^{-1}C(L, u_s,v_s).
\end{eqnarray*}
Hence, putting the above estimates together into the sup estimate %
\eqref{supestimate}, we have proved 
\begin{equation}  \label{est-supproof}
\begin{aligned} \varepsilon ^{\frac{\gamma }{2}}\|\bar{u}^\veps\|_{\infty
}+\varepsilon ^{\frac{1}{2}+\frac{\gamma }{2}}\|\bar{v}^\veps\|_{\infty } &
\le C(u_s, v_s) \varepsilon^{-\frac{3\gamma}{4} + \frac 14 -\kappa} +
\varepsilon ^{- \frac{\gamma }{4}+ \frac 12
-\kappa}\{\|[u^\veps,v^\veps]\|_{X}+\|[u^\veps,v^\veps]\|_{X}^{2}\}
\\&\qquad + \varepsilon^{\frac \gamma 4-\kappa }C(L, u_s, v_s) \| [u^\veps,
v^\veps\|_{X}. \end{aligned}
\end{equation}

Thus, by definition and the assumption that $\gamma +\kappa \le \frac 14$,
the estimates \eqref{est-R12proof} and \eqref{est-supproof} yield 
\begin{equation*}
\|[\bar{u}^\veps,\bar{v}^\veps]\|_{X}\leq C(u_s,v_s) \Big( 1+
\varepsilon^{\frac \gamma 4 -\kappa}C(L) \|[u^\veps,v^\veps]\|_{X}+
\varepsilon^{\frac 38}\|[u^\veps,v^\veps]\|_{X}^{2} \Big).
\end{equation*}%
This proves that the operator $[u^\veps,v^\veps] \mapsto [\bar u^\veps, \bar
v^\veps]$ via the problem \eqref{eqs-iteration} maps the ball $\{
\|[u^\veps,v^\veps]\|_{X}\leq 2C(u_s,v_s)\}$ in $X$ into itself, for
sufficiently small $\varepsilon$. Moreover, it follows in the similar lines
of estimates that 
\begin{equation*}
\|[\bar{u}^\veps_{1}-\bar{u}^\veps_{2},\bar{v}^\veps_{1}-\bar{v}%
^\veps_{2}]\|_{X}\leq C(L, u_s, v_s)\varepsilon ^{\frac{\gamma }{4} -
\kappa}
\{\|[u^\veps_{1},v^\veps_{1}]\|_{X}+\|[u^\veps_{2},v^\veps_{2}]\|_{X}\}\|[u^%
\veps_{1}-u^\veps_{2},v^\veps_{1}-v^\veps_{2}]\|_{X},
\end{equation*}%
for every two pairs $[u^\veps_1,v^\veps_1]$ and $[u^\veps_2,v^\veps_2]$ in $%
X $. This proves the existence of the solution to \eqref{eqs-iteration} with 
$[\bar u^\veps, \bar v^\veps] = [u^\veps,v^\veps]$ via the standard
contraction mapping theorem, for sufficiently small $\varepsilon$. The main
theorem is proved.

\bigskip

\emph{Acknowledgement}: The authors wish to thank H. Dong, J. Guzman, and I.
Tice for their discussions and pointing out references \cite{BR, O, OS} on
regularity of Stokes problems in a domain with corners. Y. Guo's research is
supported in part by NSFC grant 10828103 and NSF grant DMS-0905255, and T. Nguyen's research was supported in part
by the NSF under grant DMS-1405728.

\end{document}